\documentclass{mcom-l}

\pdfoutput=1

\newcommand{\CompressedVersion}[2]{#1}

\usepackage[colorlinks=true, linkcolor=blue,citecolor=blue, urlcolor=blue]{hyperref}
\usepackage[latin1]{inputenc}
\usepackage{mystyle,amssymb,dsfont,amsmath,mathtools}

\newtheorem{theorem}{Theorem}[section]
\newtheorem{proposition}[theorem]{Proposition}
\newtheorem{lemma}[theorem]{Lemma}
\theoremstyle{definition}
\newtheorem{definition}[theorem]{Definition}
\newtheorem{example}[theorem]{Example}

\renewcommand{\paragraph}[1]{\subsubsection{#1}}

\newtheorem{assumptions}{Assumptions}{\bfseries}{\rmfamily}
 
\CompressedVersion{\graphicspath{{./images-low/}} 
}{\graphicspath{{./images/}}} 

\begin{document}

\title[Scaling Algorithms for Unbalanced Optimal Transport Problems]{Scaling Algorithms for\\Unbalanced Optimal Transport Problems}


\author[L. Chizat]{L\'ena\"ic Chizat}
\address{CEREMADE, CNRS, Universit\'e Paris-Dauphine, INRIA Project team Mokaplan}
\curraddr{}
\email{\{chizat,schmitzer,vialard\}@ceremade.dauphine.fr}
\thanks{}

\author[G. Peyr\'e]{Gabriel~Peyr\'e}
\address{CNRS and DMA, \'Ecole Normale Sup\'erieure, INRIA Project team Mokaplan}
\curraddr{}
\email{gabriel.peyre@ens.fr}
\thanks{}

\author[B. Schmitzer]{Bernhard~Schmitzer}

\author[F-X. Vialard]{Fran\c cois-Xavier~Vialard}



\date{}

\dedicatory{}


\begin{abstract}
	This article introduces a new class of fast algorithms to approximate variational problems involving unbalanced optimal transport.
	While classical optimal transport considers only normalized probability distributions, it is important for many applications to be able to compute some sort of relaxed transportation between arbitrary positive measures.
	A generic class of such ``unbalanced'' optimal transport problems has been recently proposed by several authors.
	In this paper, we show how to extend the, now classical, entropic regularization scheme to these unbalanced problems. 
	This gives rise to fast, highly parallelizable algorithms that operate by performing only diagonal scaling (i.e.\ pointwise multiplications) of the transportation couplings. They are generalizations of the celebrated Sinkhorn algorithm. 
	We show how these methods can be used to solve unbalanced transport, unbalanced gradient flows, and to compute unbalanced barycenters. We showcase applications to 2-D shape modification, color transfer, and growth models.  
\end{abstract}

\maketitle


\section{Introduction}

Optimal transport (OT) is a standard way to lift a metric defined on some ``ground'' space $X$ to a metric on probability distributions (positive Radon measures with unit mass) $\Mm_+(X)$. Initially formulated by Monge~\cite{Monge1781} as a non-convex optimization problem over transport maps, its modern formulation as a linear program is due to Kantorovitch~\cite{Kantorovich42}, and it has been revitalized thanks to the groundbreaking work of Brenier~\cite{Brenier91}. We refer to the monographs~\cite{cedric2003topics,SantambrogioBook} for a more detailed background on the theory of OT. 

While initially developed by theoreticians, OT is now becoming popular in applied fields, and we refer for instance to application for color manipulation in image processing~\cite{RabinPapadakisSSVM}, reflectance interpolation in computer graphics~\cite{Bonneel-displacement}, image retrieval in computer vision~\cite{RubTomGui00} and statistical inference in machine learning~\cite{Solomon-ICML}. A key limitation of classical OT is that it requires the input measures to be normalized to unit mass, which is a problematic assumption for many applications that require either to handle arbitrary positive measures and/or to allow for only partial displacement of mass. All these applications might indeed benefit from OT algorithms that can handle mass variation (creation or destruction) as well as mass transportation.

While many proposals have been made to account for these ``unbalanced'' optimal transport problems, they used to be tailored for specific applications. There have been several recent works (reviewed below) that try to put all these proposals (and some new ones) into a common generic framework. These emerging theoretical advances call for algorithms that extend existing fast OT methods to these new unbalanced problems. It is precisely the goal of this paper to show how a popular numerical approach to OT -- namely entropic regularization -- does extend in a very natural and efficient way to solve a variety of unbalanced problems, including unbalanced OT, barycenters and gradient flows.

\subsection{Previous Works}
\label{sec: previous works}
\paragraph{Algorithms for optimal transport.}

The Kantorovitch formulation of OT as a linear program, when restricted to sums of Diracs, can be directly tackled using simplex or interior point methods. For the special case of optimal linear assignment (i.e.\  for sums of the same number of uniform-mass Diracs) one can also use combinatorial algorithms such as the Hungarian method~\cite{Burkard09} or the auction algorithm~\cite{Bertsekas-OptimalTransport-1989}. The time complexity of these algorithms is roughly cubic in the number of Diracs, and hence they do not scale to very large problems.

In the specific case of the squared Euclidean cost, it is possible to make use of the geodesic structure of the OT distance, and reparametrize this as a convex optimization problem, as proposed by Benamou and Brenier~\cite{benamou2000computational}, see also~\cite{2014-papadakis-siims} for a discussion on the use of first order non-smooth optimization schemes.
In \cite{SchmitzerShortCuts2015} a multi-scale algorithm is developed that consistently leverages the structure of geometric transport problems to accelerate linear program solvers.

A last class of approaches deals with semi-discrete problems, when one measure has a density, and the second one is a weighted sum of Dirac masses. This problem, introduced by Alexandrov and Pogorelov as a theoretical tool, can be solved with geometric tools when using the squared Euclidean cost~\cite{AurenhammerHA98} (see also~\cite{Merigot11} for the development of efficient algorithms using methods from computational geometry and~\cite{Levy3d} for 3-D computations).

\paragraph{Entropic transport.}

These computational methods thus cannot cope with large scale problems with arbitrary transportation costs. A recent class of approaches, initiated and revitalized by the paper of Marco Cuturi~\cite{CuturiSinkhorn}, proposes to compute an approximate transport coupling using entropic regularization. This idea has its origins in many different fields, most notably it is connected with Schr\"odinger's problem in statistical physics~\cite{Schroedinger31,LeonardSchroedinger} and with the iterative scaling algorithm by Sinkhorn~\cite{Sinkhorn64} (also known as IPFP~\cite{DemingStephanIPFP}) which, given a square matrix with positive entries, aims at finding two vectors of positive numbers---so-called scalings---that makes it a bistochastic matrix after multiplying rows and columns by these vectors. 
This entropic smoothing can be interpreted as a strictly convex barrier for positivity, but its main computational advantage is that it leads to very simple closed form expressions for all steps of the algorithm, which would not be possible when using different regularization functionals.
Several follow-up articles~\cite{CuturiBarycenter,2015-benamou-cisc} to~\cite{CuturiSinkhorn} have shown that the same strategy can also be used to tackle the computation of barycenters for the Wasserstein distance (as initially formulated by~\cite{agueh2015optimal}), and for solving OT problems on geometric domains (such as regular grids or triangulated meshes) using convolution and the heat diffusion kernel~\cite{2015-solomon-siggraph}. 
Some theoretical properties of this regularization are studied in~\cite{2015-carlier-convergence}, including the $\Gamma$-convergence of the regularized problem toward classical transport when regularization vanishes. 

\paragraph{Optimization with Bregman divergences.}

The success of this entropic regularization scheme is tightly linked with use of the Kullback-Leibler (KL) divergence as a natural Bregman divergence~\cite{bregman1967relaxation} for the optimization on the space of positive Radon measures. Not only is this divergence quite natural, but it also leads to simple formulas for the computation of projectors and so-called proximal operators (see below for a definition) for many functions typically involved in OT.  
The most simple algorithm, which is actually at the heart of Sinkhorn's iterations, is the iterative projection on affine subspaces for the KL divergence.
A refined version of these iterations, which works for arbitrary convex sets (not just affine spaces) is the so-called Dykstra's algorithm~\cite{Dykstra83}, which can be interpreted (just like iterative projections) as an iterative block-coordinates minimization on a dual problem.  
Dykstra's algorithm is known to converge when used in conjunction with Bregman divergences~\cite{bauschke-lewis,CensorReich} (see~\cite{2015-Peyre-siims} for details on the underlying idea for sums of two arbitrary functions).
Many other first order proximal methods for Bregman divergences exists. The most simple one is the proximal point algorithm~\cite{EcksteinProxPoint}, but most proximal splitting schemes have been extended to this setting, such as for instance ADMM~\cite{WangBanerjee-ADMM}, primal-dual schemes~\cite{ChambollePock-div} and forward-backward~\cite{NguyenFB2015}.

The algorithm we propose in this article can be seen as special instance of Dykstra's iterations, but with an extremely simple (both conceptually and algorithmically) structure, which we refer to as a ``scaling'' algorithm. It extends Sinkhorn's iterations to more complex problems. This structure is due to the fact that the functions involved in OT problems make use of the marginals of the couplings that are being optimized. 

\paragraph{Unbalanced transport: from theory to numerics.}

There has been a large number of proposals to extend OT methods to arbitrary ``unbalanced'' positive measures. Let us for instance quote the Kantorovitch-Rubinstein dual-Lipschitz norms~\cite{hanin1992kantorovich}, optimal partial transport~\cite{figalli2010optimal} and geodesic computations with source terms~\cite{piccoli2014generalized,OTmaasrumpf}. Most of these approaches, and much more, can be seen as special instances of a generic class of OT-like problems, that have been proposed independently in~\cite{LieroMielkeSavareLong} and~\cite{2015-chizat-unbalanced}.
These unbalanced problems can be formulated in several ways, that are equivalent (under some restrictive conditions on the cost): a geodesic (dynamic) formulation with a source term~\cite{kondratyev2015,2015-chizat-interpolating,LieroMielkeSavareShort}, a static formulation with two semi-couplings~\cite{2015-chizat-unbalanced} and a static formulation with approximate marginal constraints~\cite{LieroMielkeSavareLong}. 
From a numerical perspective, the last formulation (approximate marginals constraints) is the most simple to handle, since, as we show in the following section, it only involves a minor modification of the initial linear program, that has to be turned into a convex problem involving $\phi$-divergences (see Definition \ref{def_divergence}). This is the one that we consider in this article. 
Note that the use of such a relaxed formulation, in conjunction with entropic smoothing has been introduced, without proof of convergence, in~\cite{FrognerNIPS} for application in machine learning.

\paragraph{Gradient flows.}

Beyond OT problems and barycenter problems, a popular use of Wasserstein distances is to study gradient flows, following the formalism of ``minimizing movements'' detailed in~\cite{ambrosio2006gradient}. This corresponds to discrete implicit stepping (i.e.\ proximal maps) for the Wasserstein distance (instead of more common Euclidean or Hilbertian metrics). In some cases, these time-discrete flows can be shown to converge to continuous flows that solve a suitable PDE, as the temporal stepsize tends to 0. The most famous example is the gradient flow of the entropy for the Wasserstein metric, which solves the diffusion equation~\cite{jordan1998variational}. Non-linear PDEs are considered for example in~\cite{otto2001geometry}. An application to imaging can be found in~\cite{Burger-JKO}. Another use of these implicit steps is to construct minimizing flows of non-smooth functionals, for instance to model crowd motions~\cite{maury2010macroscopic}. 

A large variety of dedicated numerical schemes has been proposed for spatial discretization and solving of these time-discrete flows, such as for instance finite differences~\cite{burger2010mixed}, finite volumes~\cite{CarrilloFiniteVolume} and Lagrangian schemes~\cite{JDB-JKO,benamou2015augmented}. A bottleneck of these schemes is the high computational complexity due to the resolution of an OT-like problem at each time step. The use of entropic regularization has been proposed recently in~\cite{2015-Peyre-siims} and studied theoretically in~\cite{2015-carlier-convergence}. A chief advantage of this approach is that each step can be solved efficiently on a regular grid with fast Sinkhorn iterations involving only Gaussian convolutions, but this comes at the price of additional diffusivity introduced by the approximation, which makes it unsuitable to capture sharp features of solutions. 

It is possible to extend these gradient flows by replacing the Wasserstein distance by more general unbalanced distances. This allows to define flows over arbitrary positive measures, hence involving creation and destruction of mass, which is crucial to model growth phenomena, such as for instance the Hele-Shaw model of tumor evolution~\cite{PerthameTumor}. An analysis of such flows based on a splitting scheme has been recently provided in~\cite{GallouetMonsaingeon2015}. It is one of the goals of this article to propose a versatile algorithm, based on iterative scalings, to approximate numerically these unbalanced flows.

\subsection{Contributions and Outline}
The main contribution of this article is to define a class of iterative scaling algorithms to solve the entropic approximation of a variety of unbalanced optimal transport problems. 
%
%
First, in Section \ref{sec-ot-like}, we exhibit a common structure to most optimization problems related to OT: a nonnegativity constraint, a linear transport cost and convex functions acting on the marginals of the optimized couplings. 
For solving the entropic regularization of these problems, we then introduce in Section \ref{sec_algorithm_analysis} a generic ``scaling'' algorithm which is a direct generalization of Sinkhorn's algorithm. 
In a continuous setting, we show under some assumptions that the iterates are well defined and correspond to alternating optimization on the dual and we prove linear convergence in a particular key case. 
In a discrete setting, we show in Section \ref{sec-algorithm} that this algorithm converges as soon as the dual problem is well-posed. We also give a simple description of the algorithm, introduce a stabilization scheme to reach very small values of the regularization parameter and sketch a generalization of this algorithm.
Finally, in Section \ref{sec_applications}, we showcase the application of these methods to 1-D and 2-D unbalanced optimal transport, generalizations of Wasserstein barycenters and gradient flows with growth.
The main advantages of our approach are that it is simple (it only involves matrix multiplication and pointwise elementary operations), quite generic (it applies to most known OT-like problems), enjoys fast convergence (linear convergence is observed, and show in particular cases) and is highly parallelizable.
The code to reproduce the results of this article is available online\footnote{\url{https://github.com/lchizat/optimal-transport}}.
Note also that an efficient numerical implementation of the scaling algorithm developed in this paper is studied
further in~\cite{SchmitzerScaling2016}, see
Section~\ref{sec:RelationSparseScaling} for a  discussion.

\subsection{Notation}
\label{sec_notation}
The space of nonnegative finite Radon measures\footnote{Borel measures which are inner regular. If $T$ is Polish, all Borel measures are inner regular.}  on a (Hausdorff) topological space $T$ is denoted by $\Mm_+(T)$, and the vector space it generates by $\Mm(T)$. For a measure on a product space $\mu\in \Mm(X \times Y)$, $P^X_\# \mu$ (or sometimes $P^1_\# \mu$ if $X=Y$) denotes its first marginal and $P^Y_\# \mu$ (or $P^2_\# \mu$) its second marginal.

We denote by $\d t, \d x, \dots$ the reference measures on some measurable spaces $T,X,\dots$. They are not assumed to be the Lebesgue measure (but most of the time, they are probability measures). If $(T,\d t)$ is a measured space, we denote by $\Lun(T)$ and $\Linf(T)$ the usual spaces of equivalence classes of measurable functions from $T$ to $\RR$ which are, respectively, absolutely integrable and essentially bounded. When there is a subscript $+$ appended to the name of a space of functions, it refers to the cone of nonnegative functions in that space.

Divergence functionals (or $\phi$-divergences) are denoted by a calligraphic letter $\Divergm$ when they act on measures (as defined in Definition \ref{def_divergence}), by straight letters $\Diverg$ when they act on functions (as defined in \eqref{eq_divergencefunctions}) and with an overline $\ol{\Diverg}$ when they act on two real numbers (as in \eqref{eq_divergencefunctions}). More generally, functionals on measures are denoted by calligraphic letters ($\mathcal{F}, \mathcal{G},\dots$) and functionals on functions by straight capital letters ($F,G,\dots$).

The same convention holds for the Kullback-Leibler divergence (which is just a particular case of divergence functionals) thus denoted by $\KLm$, $\KL$ or $\ol{\KL}$ depending on the nature of its arguments. Moreover, we generalize the definition of $\KL$ to families of functions as follows: if $(X, \d x)$ and $(Y,\d y)$ are two measured spaces and $r,s : X\times Y \to \RR^n$ are families of measurable functions, the KL-divergence is defined as
\eql{
\label{eq_KLdens}
\KL(r|s) \eqdef \sum_i \int_{X\times Y} \left[ \log \left(\frac{r_i(x,y)}{s_i(x,y)} \right) r_i(x,y) -r_i(x,y) +s_i(x,y) \right]\d x \d y 
}
(with $0\log(0/0)=0$) if $r,s\geq0$ a.e.\ and $(s_i(x,y)=0) \Rightarrow (r_i(x,y)=0)$ a.e., and $\infty$ otherwise.

The generalization of some notations to families of functions is often implicit. For instance, if $(r_i)_{i=1}^n$ and $(s_i)_{i=1}^n$ are two families of functions,  we write
\eq{
\langle r,s \rangle \eqdef \sum_{i=1}^n \int_{X \times Y} r_i(x,y)\, s_i(x,y) \d x \d y
}
and the projection operators are defined componentwise, for $i\in \{1,\dots,n\}$, as
\eql{\label{eq_ProjectionDens}
(P^X_\# r)_i (x) = \int_Y r_i(x,y) \d y \qandq (P^Y_\# r)_i (y) = \int_X r_i(x,y) \d x  \, .
}

If $X$ and $Y$ are finite spaces (i.e.\ contain only a finite number of points), we represent functions on $X$ by vectors denoted by bold letters $\mathbf{a}$ and functions on $X\times Y$ by matrices denoted by capital bold letters $\mathbf{A}$. In this context the notations $\mathbf{a}\odot\mathbf{b}$ and $\mathbf{a}\oslash\mathbf{b}$ denote, respectively entrywise multiplication and entrywise division with convention $0/0=0$ between vectors.

The conjugate of a convex function $f$ is denoted by $f^*$ and its subdifferential is denoted by $\partial f$. Some reminders on convex analysis are given in Appendix \ref{sec:ApxConvex}. The indicator of some convex set $C$ is 
\eq{
	\iota_C(a) = \choice{
		0 \qifq a \in C, \\
		+\infty \quad \text{otherwise}.
		}
}

Finally, the proximal operator plays a central role in this article. In full generality, if $E$ a set, $\Diverg: E\times E\to \RR_+\cup \{\infty\}$ is a divergence which measures ``closeness'' between points in $E$, it is defined for $F:E\to \RR\cup \{\infty\}$ and $x\in E$ as
\eq{
\prox^D_F(x) = \argmin_{y\in E} F(y) + \Diverg(y|x)\, .
}


\section{Unifying Formulation of Transport-like Problems}
\label{sec-ot-like}
We review below a variety of variational problems related to optimal transport that can all be recast as a generic variational problem, involving transport and functions on the marginals. The notion of ``divergence'' functionals, introduced in optimal transport by \cite{LieroMielkeSavareLong}, makes this unification possible and is defined in the following.

\subsection{Divergence Functionals}
\label{sec_divergencefunc}
Divergences are functionals which, by looking at the pointwise ``ratio'' between two measures, give a sense of how close they are. They have nice analytical and computational properties and are built from \emph{entropy functions}.

\begin{definition}[Entropy function]
\label{def_entropy}
A function $\phi : \RR \to \RR \cup \{\infty\}$ is an entropy function if it is lower semicontinuous, convex, $\dom \phi\subset [0,\infty[$ and satisfies the following feasibility condition:  $\dom \phi \; \cap\;  ]0, \infty[\; \neq \emptyset$. The speed of growth of $\phi$ at $\infty$ is described by 
\eq{
\phi'_\infty = \lim_{x\uparrow +\infty} \varphi(x)/x \in \RR \cup \{\infty\} \, .
}
%
\end{definition}

If $\phi'_\infty = \infty$, then $\phi$ grows faster than any linear function and $\phi$ is said \emph{superlinear}. Any entropy function $\phi$ induces a $\phi$-divergence (also known as Csisz\'ar divergence) as follows.
\begin{definition}[Divergences]
\label{def_divergence}
Let $\phi$ be an entropy function.
For $\mu,\nu \in \Mm(T)$, let $\frac{\d \mu}{\d \nu} \nu + \mu^{\perp}$ be the Lebesgue decomposition\footnote{The Lebesgue decomposition Theorem asserts that, given $\nu$, $\mu$ admits a unique decomposition as the sum of two measures $\mu^s + \mu^\perp$ such that $\mu^s$ is absolutely continuous with respect to $\nu$ and $\mu^\perp$ and $\nu$ are singular.} of $\mu$ with respect to $\nu$. The divergence $\Divergm_\phi$ is defined by
\eq{
\Divergm_\phi (\mu|\nu) \eqdef \int_T \phi\left(\frac{\d \mu}{\d \nu} \right) \d \nu 
+ \phi'_\infty \mu^{\perp}(T)
}
if $\mu,\nu$ are nonnegative and $\infty$ otherwise.
\end{definition}%
The proof of the following Proposition can be found in~\cite[Thm 2.7]{LieroMielkeSavareLong}.
\begin{proposition}
If $\phi$ is an entropy function, then $\Divergm_\phi$ is jointly $1$-homogeneous, convex and weakly* lower semicontinuous in $(\mu,\nu)$.
\end{proposition}
%
The Kullback-Leibler divergence, also known as the relative entropy, plays a central role in this article.
\begin{example}
\label{ex_KLdiv}
The Kullback-Leibler divergence $\KLm \eqdef \Divergm_{\phi_{\KL}}$ is the divergence associated to the entropy function $\phi_{\KL}$, given by
\eql{\label{eq-defn-kl}
	\phi_{\KL}(s)= \begin{cases}
		s\log(s)-s+1 & \textnormal{for } s>0 , \\
		1 & \textnormal{for } s=0 , \\
		+\infty & \textnormal{otherwise.}
		\end{cases}
}
When restricted to densities on a measured space, it is also the Bregman divergence~\cite{bregman1967relaxation} associated to (minus) the entropy.
\end{example}
The following divergences will also be considered as examples.
\begin{example}
The total variation distance $\TV(\mu|\nu) \eqdef |\mu-\nu|_{\TV}$ is the divergence associated to:
\eql{\label{eq-defn-tv}
	\phi_{\TV}(s)= \begin{cases}
		|s-1| & \textnormal{for } s\geq0 , \\
		+\infty & \textnormal{otherwise.}
		\end{cases}
}
\end{example}
\begin{example}
The equality constraint $\iota_{\{=\}}(\mu| \nu)$ which is $0$ if $\mu=\nu$ and $\infty$ otherwise, is the divergence associated to $\varphi=\iota_{\{1\}}$.
\end{example}
\begin{example}
\label{ex_RG}
One can generalize the latter and define a ``range constraint'', denoted $\RG_{[\alpha,\beta]}(\mu|\nu)$, as the divergence which is zero if $\alpha \, \nu \leq \mu \leq \beta \, \nu$ with $0\leq\alpha\leq\beta\leq \infty$, and $\infty$ else. This is the divergence associated to $\phi = \iota_{[\alpha, \beta]}$.
\end{example}

\subsection{Balanced Optimal Transport}
\label{subsec_balanced}

Using divergences, the classical ``balanced'' optimal transport problem can be defined as follows.

\begin{definition}[Balanced Optimal Transport]
Let $X$ and $Y$ be Hausdorff topological spaces, let $c : X \times Y \rightarrow \RR \cup \{\infty\}$ be a lower semi-continuous function. The optimal transport problem between $\mu\in \Mm_+(X)$ and $\nu\in \Mm_+(Y)$ is
\eql{\label{eq-ot}
	 \uinf{\gamma \in \Mm_+(X\times Y)}  \int_{X\times Y} \!\!\! c\, \d\gamma  
	 + \iota_{\{=\}}(P^X_\# \gamma | \mu) 
	 + \iota_{\{=\}}(P^Y_\# \gamma | \nu) .
}
\end{definition}

\begin{proposition}
If $c$ is lower bounded and \eqref{eq-ot} is feasible, then the infimum is attained.
\end{proposition}
The proof of this result, as well as  a general theory of optimal transport can be found in \cite{Villani-OptimalTransport-09}.
If one interprets $c(x,y)$ as the cost of transporting a unit mass from $x$ to $y$, the problem is then to find the cheapest way to move a mass initially distributed according to $\mu$ toward the distribution $\nu$. If $\gamma$ is additionally constrained to be of the form $(\mathrm{id}\times T)_\# \mu$ for some map $T:X\to Y$, then \eqref{eq-ot} is called the Monge problem.

\begin{example}
An important special case is when $X=Y$ and $c$ is the power of a distance on $X$. Then, the optimal cost (i.e. the value of~\eqref{eq-ot}) is itself the power of a distance on the space of probability measures on $X$. This distance is often referred to as ``Wasserstein'' distance and denoted by $\Wass$, although this denomination is disputed. We refer to \cite[Chap.\ 6, Bibliographical Notes]{Villani-OptimalTransport-09} for a discussion of the historical context.
\end{example}

\subsection{Unbalanced Optimal Transport}
\label{subsec_unbalanced}
Standard optimal transport only allows meaningful comparison of measures with the same total mass: whenever $\mu(X) \neq \nu(Y)$, there is no feasible $\gamma$ in \eqref{eq-ot}. Several propositions have been made to circumvent this limitation by defining ``unbalanced'' transport problems in various dynamic and static formulations (see Section~\ref{sec: previous works}).
The following formulation with relaxed marginal constraints is best suited for the numerical schemes presented in this article.

\begin{definition}[Unbalanced Optimal Transport {\protect \cite[Problem 3.1]{LieroMielkeSavareLong}}]
	\label{def:SoftMarginalFormulation}
	Let $X$ and $Y$ be Hausdorff topological spaces, let $c: X\times Y \rightarrow [0,\infty]$ be a lower semi-continuous function and let $\Divergm_{\varphi_1}$, $\Divergm_{\varphi_2}$ be two divergences over $X$ and $Y$, as in Definition \ref{def_divergence}. For $\mu \in \Mm_+(X)$ and $\nu \in \Mm_+(Y)$, the unbalanced soft-marginal transport problem is
	\begin{align}	
		\label{eq-unbalanced-ot}
		 \inf_{\gamma \in \Mm_+(X\times Y)}
			\int_{X \times Y}\!\!\! c\,\d\gamma + \Divergm_{\varphi_1}(P^X_\# \gamma | \mu)
			+ \Divergm_{\varphi_2}(P^Y_\# \gamma | \nu)\,.
	\end{align}
\end{definition}
\begin{proposition}

\label{prop_UOT_existence}
Assume that \eqref{eq-unbalanced-ot} is feasible. If $\phi_1$ and $\phi_2$ are superlinear, then the infimum is attained. This is also the case if $c$ has compact sublevel sets and $\phi'_{1\infty} + \phi'_{2\infty} + \inf c > 0$.
\end{proposition}
The proof of this result as well as a thorough study of duality properties of this problem can be found in \cite{LieroMielkeSavareLong}.
Note that~\eqref{eq-ot} is a particular case of~\eqref{eq-unbalanced-ot} when the domains of the entropy functions $\phi_i$ are the singleton $\{1\}$. More generally, if the functions $\phi_i$ admit unique minima at $1$, \eqref{eq-unbalanced-ot} can be viewed as a relaxation of the initial problem~\eqref{eq-ot} where the ``hard'' marginal constraints are replaced by ``soft'' constraints, penalizing the deviation of the marginals of $\gamma$ from $\mu$ and $\nu$.
Now we discuss a specific case of particular interest.
\begin{definition}[Wasserstein-Fisher-Rao Distance]
\label{def:WassF}
Take $X=Y$, let $d$ be a distance on $X$ and $\lambda \geq 0$. For the cost
\eql{ \label{logcost}
	c(x,y) = -\log\pa{ \cos_+^2\pa{d(x,y)}}
}
with $\cos_+ : z \mapsto \cos(z\wedge\frac{\pi}{2})$ and $\Divergm_{\phi_1}=\Divergm_{\phi_2}=\lambda \KLm$, we define $\WassF_\lambda(\mu,\nu)$ as the square root of the minimum in \eqref{eq-unbalanced-ot} (as a function of the measures $(\mu,\nu) \in \Mm_+(X)^2$). 
We simply write $\WassF \eqdef \WassF_1$ for $\lambda=1$.
\end{definition}

As shown in~\cite{LieroMielkeSavareLong}, $\WassF$ defines a distance on $\Mm_+(X)$, which is equal to the \emph{Wasserstein-Fisher-Rao} geodesic distance introduced simultaneously and independently in~\cite{LieroMielkeSavareLong,kondratyev2015,2015-chizat-interpolating} (it is named \emph{Hellinger-Kantorovich} in \cite{LieroMielkeSavareLong}).
An alternative static formulation is given in \cite{2015-chizat-unbalanced}.
Other important cases include:
\begin{itemize}
\item The \emph{Gaussian Hellinger-Kantorovich} distance $\GHK$ is obtained by taking the cost $c=d^2$ ($d$ still a metric) and  $\Diverg_{\phi_i}=\lambda \KL$ with $\lambda > 0$. It has been introduced in~\cite{LieroMielkeSavareLong} where it is also shown that when $X$ is a geodesic space, $\WF$ is the geodesic distance generated by $\GHK$.

\item The \emph{optimal partial transport} problem, which is obtained by taking a cost function bounded from below by $-2\lambda$ and $\Diverg_{\phi_i}=\lambda \TV$, with $\lambda>0$. Originally, optimal partial transport refers to a problem where the marginal constraints are replaced by a constraint on the total mass of the marginals (and domination constraints); here $\lambda$ is the Lagrange multiplier associated to the mass constraint (see \cite{caffarelli2010free}).
\item With the divergence $\RG$ as in example \ref{ex_RG}, one imposes a range constraint on the marginal:
\eq{
\RG(P^X_\# \gamma | \mu) < + \infty
\qquad \Leftrightarrow \qquad
\alpha \mu \leq P^X_\# \gamma \leq \beta \mu
}
(and similarly for $P^Y_\# \gamma$).
\end{itemize}

The following result is a minor extension of results from \cite{LieroMielkeSavareLong,2015-chizat-unbalanced} and is useful if one wishes to vary the parameter $\lambda$ in numerical applications.

\begin{proposition}
$\WassF_\lambda$ defines a distance if $0\leq\lambda\leq1$ (degenerate if $\lambda=0$). The upper bound on $\lambda$ is necessary when $X$ is a geodesic space of diameter greater than $\pi/2$.
\end{proposition}
\begin{proof}
The case $\lambda=0$ is trivial. If $\lambda>0$, when dividing \eqref{eq-unbalanced-ot} by $\lambda$, one obtains the new cost $-2\log(\cos_+^{1/\lambda}(d(x,y)))$. But the function $f:d \mapsto \arccos(\cos^{\frac1{\lambda}}(d))$ defined on $[0,\pi/2]$ is increasing, positive, satisfies $\{0\}=f^{-1}(0)$ and for $x\in ]0,\pi/2[$ it holds
\[
f'' = \frac{1}{\la} \frac{\cos^{\frac{1}{\la}-2}}{(1-\cos^{\frac{2}{\la}})^2} \left( 1 - \frac{1}{\la} \sin^2 -\cos^{\frac{2}{\la}}\right)
\, .
\]
From the convexity inequality $\text{sign} [(1-X)^\frac{1}{\la}-1 + \frac{1}{\la} X] = \text{sign} (\frac{1}{\la}-1) $ it follows that $f$ is strictly concave on $[0,\pi/2]$ if $\la<1$ and strictly convex if $\la>1$.
Thus if $\la \leq 1$, $f \circ d$ still defines a distance on $X$ and consequently $\WassF_\lambda$ too. 

If $X$ is a geodesic space of diameter greater than $\pi/2$, take $(x_1,x_2) \in X^2$ such that $d(x_1,x_2)=\pi/2$. From~\cite[Corollary 8.3]{LieroMielkeSavareLong}, $(\Mm(X), \WassF_1)$ is itself a geodesic space. Consequently, there exists a midpoint $\mu \in \Mm_+(X)$, i.e.\ such that \eq{
\WassF_1(\delta_{x_1},\mu) = \WassF_1(\delta_{x_2},\mu) =\frac12 \WassF_1(\delta_{x_1},\delta_{x_2})\, .
}
From \cite[Theorem 8.6]{LieroMielkeSavareLong} and the characterization of geodesics in $\text{Cone}(X)$ it holds $0<d(x,x_i)<\pi/2$ for  $\mu$ a.e.\ $x\in X$. This implies, for $\lambda>1$, and $\mu$ a.e.\ $x$, $-\log(\cos_+^{1/\lambda}(d(x_i,x)))<-\log(\cos_+(d(x_i,x)))$. Thus 
$(\WassF^2_\lambda /\lambda)(\delta_{x_i},\mu)<\WassF^2_1(\delta_{x_i},\mu)$. But, for $\lambda\geq 1$, $(\WassF^2_\lambda/\lambda)(\delta_{x_1},\delta_{x_2})=2\KL(0|1)=2$ this leads to
\eq{
\WassF_\lambda(\delta_{x_1},\mu)+ \WassF_\lambda(\delta_{x_2},\mu) < \WassF_\lambda(\delta_{x_1},\delta_{x_2})
}
and the triangle inequality property is lost. 
\end{proof}


\subsection{Barycenter Problem and Extensions}
\label{subsec_barycenter}

The problem of finding an ``average'' measure $\sigma \in \Mm_+(Y)$ which minimizes the sum of the (possibly unbalanced) transport cost toward every measure of a family $(\mu_k)_{k=1}^n \in \Mm_+(X)^n$ is of theoretical and practical interest (see Section~\ref{sec: previous works}). To formalize this problem, consider a family of costs functions $(c_k)_{k=1}^n $ on $X \times Y$ and a family of divergences $(\Divergm_{\phi_{k,1}},\Divergm_{\phi_{k,2}})_{k=1}^n$. The problem is to solve
\eq{
\inf_{\sigma\in \Mm_+(Y)} \inf_{\substack{(\gamma_k)_{k=1}^n \in \\ \Mm_+(X\times Y)^n}} \sum_{k=1}^n  
			\left( \int_{X\times Y} c_k \, \d \gamma_i  
			+ \Divergm_{\varphi_{k,1}}(P^X_\# \gamma_i | \mu_i)
			+ \Divergm_{\varphi_{k,2}}(P^Y_\# \gamma_i | \sigma) \right) \, .
}
By exchanging the infima, this is equivalent to
\eql{
\label{eq_barycenter}
\inf_{\substack{(\gamma_k)_{k=1}^n \in \\ \Mm_+(X\times Y)^n}} 
	\sum_{k=1}^n  \int_{X\times Y} \!\!\!\! c_k \, \d \gamma_k 
	+ \sum_{k=1}^n \Divergm_{\phi_{i,1}}(P^X_\# \gamma_k | \mu_k) 
	+ \!\! \inf_{\sigma\in \Mm_+(Y)} \sum_{k=1}^n \Divergm_{\phi_{k,2}}(P^Y_\# \gamma_k| \sigma)\, .
}
Note that while the object of interest is the minimizer $\sigma$ and not the family of couplings, we will see in Section \ref{sec_appli_bary} that the computation of $\sigma$ is a byproduct of the ``scaling'' algorithm defined below. 

The two following examples are specific instances of so-called Fr\'echet means, defined in any complete metric space $(E,d)$ as solutions to
\eq{
\argmin_{\sigma \in E} \sum_k \al_k \, d(\mu_k, \sigma)^2
}
where $(\alpha_k)_{k=1}^n$ is a family of nonnegative weights and $(\mu_k)_{k=1}^n\in E^n$.
\begin{example}[Wasserstein barycenters]
Let $X=Y$ and let $c$ be the quadratic cost $(x,y) \mapsto d(x,y)^2$ for a metric $d$ and define $c_k = \alpha_k \, c$. Let all the divergences be the equality constraints. Then \eqref{eq_barycenter} is a formulation of the Fr\'echet means in the Wasserstein space, also known as Wasserstein barycenters (see \cite{agueh2011barycenters}).
\end{example}

\begin{example}[$\WF$ barycenters]
Let  $c$ be as in \eqref{logcost}, define $c_k=\alpha_k \, c$ and let $\Divergm_{\phi_{k,1}} = \Divergm_{\phi_{k,1}}= \alpha_i \KLm$ for all $k\in \{1,\dots,n\}$. Then \eqref{eq_barycenter} is a formulation of Fr\'echet means for the $\WF$ distance.
\end{example}


\subsection{Gradient Flows}
\label{subsec_gradientflows}

\paragraph{General outline.}
Given a metric space $(T,d)$ and some lower-semicontinuous function $\mathcal{G} : T \rightarrow \RR \cup \{ \infty\}$ with compact sublevel sets, a time-discrete gradient flow with step size $\tau>0$ starting from an initial point $\mu_{0} \in T$ corresponds to the computation of a sequence $(\mu^\tau_k)_{k\in \NN}$ via
\begin{align*}
	\mu^\tau_0 & \eqdef \mu_0, &
	\mu^\tau_{k+1} & \in \uargmin{\mu \in T} \mathcal{G}(\mu) + \frac{d^2(\mu^\tau_k,\mu)}{2\tau} \, .
\end{align*}
This scheme, which is a particular case of De Giorgi's \emph{minimizing movements}, generalizes the notion of implicit Euler steps to approximate gradient flows in Euclidean spaces, which are recovered when $(T,d)=(\RR^n,|\cdot|)$. From a theoretical point of view, the object of interest is the limit trajectory when $\tau \rightarrow 0$ of the piecewise constant interpolation
\eq{
\mu^\tau (t) \eqdef \mu^\tau_k  \text{ for all } t \in [k\tau, (k + 1)\tau[\, .
}
Initiated by \cite{jordan1998variational}, the study of such flows when $T$ is the space of probability measures endowed with the Wasserstein metric $d=\W$ (see Section \ref{subsec_balanced}) has led to considerable advances in the theoretical study of PDEs and their numerical resolution (see Section~\ref{sec: previous works}).
The recent introduction of ``unbalanced transport'' metrics paves the way for further applications, since it is now possible to consider the whole space of nonnegative measures, endowed for instance with the metric $\WassF$, as considered in \cite{GallouetMonsaingeon2015}.

\paragraph{Minimization problem.}
For gradient flows based on an optimal transport metric, such as $\W$, $\WF$ or $\GHK$, each step requires to solve, after swapping the two infima, a problem of the form
\eql{
\label{eq-grad-flow}
\inf_{\gamma \in \Mm_+(X\times X)}
\int _{X\times Y} \!\!\! c \, \d \gamma 
+ \Divergm_{\phi_1}(P^1_\# \gamma | \mu_k^\tau) 
+ \!\!\!  \inf_{\mu \in \Mm_+(X) } \left( \Divergm_{\phi_2}(P^2_\# \gamma | \mu) + 2 \tau \mathcal{G}(\mu) \right)
}
where $\phi_1$, $\phi_2$ are entropy functions and $\mathcal{G}$ is a lower semicontinuous functional which we also assume convex.
This problem, as well as variants, involving so-called ``splitting'' techniques (see \cite{GallouetMonsaingeon2015}) fit into the framework developed below. In particular, we show in Section \ref{sec_applicationGF} that the computation of $\mu_{k+1}$ is a byproduct of the minimization of \eqref{eq-grad-flow} with the algorithm defined below.

\paragraph{Flows associated to $\WF$.}
The distance $\WF$ was only introduced recently, so a sound theoretical analysis of the corresponding gradient flows and their limit PDEs is not yet available (and is not the subject of the present article). Meanwhile, we present here heuristic arguments in a smooth setting (see also \cite[Section 3.2]{kondratyev2015}) which lead to the evolution equation
\eql{
\label{eq_GFforWF}
\partial_t \mu = \diverg (\mu \nabla \mathcal{G}'(\mu)) -4 \mu \mathcal{G}'(\mu)
}
for the limit trajectory of $\mu^\tau(t)$ as $\tau \rightarrow 0$, where $\mathcal{G}'(\mu)$ is, if it exists, the unique function such that $\frac{\d }{\d \epsilon} \mathcal{G}(\mu + \epsilon \chi)\vert_{\epsilon=0} = \int_X \mathcal{G}'(\mu) \d \chi$ for all perturbations admissible perturbation $\chi$.
\begin{proof}[Heuristic argument]
When restricted to measures $\mu$ with positive density of Sobolev regularity, $\WF$ is the weak Riemannian metric associated to the tensor
\eq{
g(\mu)(\delta \mu, \delta \mu) \eqdef \inf_{v, \alpha}  \int_X |v(x)|^2\mu(x) \d x + \frac14 \int_X |\alpha(x)|^2 \mu(x) \d x
}
where, for a small variation $\delta \mu$, one searches over the decompositions $\delta \mu = -\diverg(v\mu) + \alpha \mu$ into displacement (given by the velocity field $v\in \Ldeux(X,\mu)^d$) and growth (given by the rate of growth $\alpha\in \Ldeux(X,\mu)$) (see \cite{2015-chizat-unbalanced}).
A new step $\mu_{k+1}^\tau$ is given from $\mu^\tau_k$ through the resolution of
\eq{
\mu_{k+1}^\tau \in \argmin_{\mu\in \mathcal{M}_+(X)} \mathcal{G}(\mu) + \frac{1}{2\tau} \WF^2(\mu,\mu_k^\tau) \, .
}
Searching the minimizer in the form $\mu = \mu^\tau_k +\tau (-\diverg (v\mu^\tau_k) + \alpha \mu_k^\tau)$ with unknown $(v,\alpha)$,
this can be rewritten, in first order of $(\tau v,\tau \alpha)$, as
\eq{
\inf_{v,\, \alpha} \mathcal{G}(\mu_k^\tau) + \int_X \left( \tau \mathcal{G}'(\mu_k^\tau) ( \alpha \mu_k^\tau- \diverg (v \mu_k^\tau))  + \frac{\tau^2}{2\tau}  \left( |v(x)|^2  + \frac14 \alpha(x)^2\right) \mu_k^\tau(x)\right) \d x\, .
}
The first order optimality conditions yield $v = -\nabla \mathcal{G}'(\mu)$ and $\alpha = -4\mathcal{G}'(\mu)$. One thus obtains
\eq{
\frac{1}{\tau} (\mu- \mu_k^\tau) = \diverg (\mu \nabla \mathcal{G}'(\mu)) -4 \mu \mathcal{G}'(\mu)
}
which yields \eqref{eq_GFforWF} as $\tau \to 0$.
\end{proof}


\subsection{Generic Formulation}
\label{sec:GenericFormulation}

Consider two convex and lower semicontinuous functions $\mathcal{F}_1$ and $\mathcal{F}_2$ defined on $\Mm_+(X)^n$ and $\Mm_+(Y)^n$ respectively. The variety of problems reviewed above can be seen as special cases of
\eql{\label{eq-general}
\min_{\gamma \in \Mm^n_+(X \times Y)} \mathcal{J}(\gamma) 
\quad \text{where} \quad
\mathcal{J}(\gamma)  \eqdef \langle c , \gamma \rangle 
+ \mathcal{F}_1(P^X_\# \gamma) 
+ \mathcal{F}_2(P^Y_\# \gamma)\, 
}
where $\langle c, \gamma \rangle = \sum_{k=1}^n \int c_k \d  \gamma_i $ is a short notation for the total transport cost of a family of couplings $(\gamma_k)_{k=1}^n$ w.r.t.\ a family of cost functions $(c_k)_{k=1}^n$ and the projection operators act elementwise. More precisely, they are recovered with the following choices of functions:
\begin{itemize}
\item Balanced OT \eqref{eq-ot}:  
$\mathcal{F}(\sigma)= \iota_{\{=\}} (\sigma | \mu)$; 
\item Unbalanced OT \eqref{eq-unbalanced-ot}: 
$\mathcal{F}(\sigma) = \Divergm_{\phi}( \sigma | \mu)$; 
\item Barycenters \eqref{eq_barycenter}: 
$\mathcal{F}(\sigma) = \inf_{\omega \in \Mm^n(X)} \left\{ \sum_{k=1}^n \Divergm_{\phi_k}(\sigma_k|\omega_k)+ \iota_D(\omega) \right\}$, 
where $D$ is the set of families of measures for which all components are equal,
\begin{equation*}
	D \eqdef \left\{ (\omega_1,\ldots,\omega_n) \in \Mm^n(X) : \omega_i = \omega_j \textnormal{ for } i,j=1,\ldots,n \right\}\,; 
\end{equation*}
\item Gradient flows \eqref{eq-grad-flow}: 
$ \mathcal{F}(\sigma) =  \inf_{\omega \in \Mm(X) } \left\{ \Divergm_{\phi}(\sigma|\omega) + 2\tau \mathcal{G}(\omega) \right\}$.
\end{itemize}
It is remarkable that, even if $\mathcal{F}$ is sometimes defined through an auxiliary minimization problem, it can still be handled efficiently in some cases by the scaling algorithm detailed below. For instance, functions of the form
\eql{
\label{eq_marginal_functional}
 \mathcal{F}(\sigma) =  \inf_{\omega \in \Mm^n(X) } \sum_{k=1}^n \Divergm_{\phi_k}(\sigma_k|\omega_k) + \mathcal{G}(\omega)\,,
 }
where $\mathcal{G}$ is convex and $(\phi_k)_n$ is a family of entropy functions, are still convex and can model a great variety of problems. A graphical interpretation of these problems is suggested in Figure \ref{fig_generic_problem} along with some examples.

\begin{figure}
\centering

\pgfmathsetmacro{\x}{.8}
\pgfmathsetmacro{\y}{1.8}

\begin{subfigure}{0.45\linewidth} 
\centering
\resizebox{\linewidth}{!}{
\begin{tikzpicture}
\fill[color=black!30, fill=black!5] (-.8,\x-.3) rectangle (4.8,\y+.5);
\draw[left]  (-1,\x/2+\y/2) node {$X$} ;
\fill[color=black!30, fill=black!5] (-.8,-\x+.3) rectangle (4.8,-\y-.5);
\draw[left]  (-1,-\x/2-\y/2) node {$Y$} ;

\draw[color=black!80] (-.2,\y-.2) rectangle (4.2,\y+.2);
\draw[right]  (4.15,\y) node {\small{$\mathcal{G}^X$}} ;
\draw[color=black!80] (-.2,-\y+.2) rectangle (4.2,-\y-.2);
\draw[right]  (4.15,-\y) node {\small{$\mathcal{G}^Y$}} ;

\draw[<->,shorten >=0.1cm,shorten <=.1cm] 	(0,-\x)    	-- 	(0,\x);
\draw[<->,shorten >=0.1cm,shorten <=.1cm] 	(1,-\x)    	-- 	(1,\x);
\draw[<->,shorten >=0.1cm,shorten <=.1cm] 	(4,-\x)    	-- 	(4,\x);

\draw (0,0) node [left] {$c_1$}; \draw (1,0) node [left] {$c_2$}; \draw (4,0) node [left] {$c_n$};
\draw (0,\x/2+\y/2) node [left] {\small{$\phi^X_1$}}; 
\draw (1,\x/2+\y/2) node [left] {\small{$\phi^X_2$}}; 
\draw (4,\x/2+\y/2) node [left] {\small{$\phi^X_n$}};

\draw (0,-\x/2-\y/2) node [left] {\small{$\phi^Y_1$}}; 
\draw (1,-\x/2-\y/2) node [left] {\small{$\phi^Y_2$}}; 
\draw (4,-\x/2-\y/2) node [left] {\small{$\phi^Y_n$}};

\draw (2.7,\x) node [left] {$\dots$}; \draw (2.7,-\x) node [left] {$\dots$};
\draw (2.7,\y) node [left] {$\dots$}; \draw (2.7,-\y) node [left] {$\dots$};

\draw[dotted,shorten >=0.1cm,shorten <=.1cm] 	(0,\x)    	-- 	(0,\y);
\draw[dotted,shorten >=0.1cm,shorten <=.1cm] 	(0,-\x)    	-- 	(0,-\y);
\draw[dotted,shorten >=0.1cm,shorten <=.1cm] 	(1,\x)    	-- 	(1,\y);
\draw[dotted,shorten >=0.1cm,shorten <=.1cm] 	(1,-\x)    	-- 	(1,-\y);
\draw[dotted,shorten >=0.1cm,shorten <=.1cm] 	(4,\x)    	-- 	(4,\y);
\draw[dotted,shorten >=0.1cm,shorten <=.1cm] 	(4,-\x)    	-- 	(4,-\y);

\draw  (0,-\x) node {$\circ$} ; \draw  (0,\x) node {$\circ$} ;
\draw  (1,-\x) node {$\circ$} ; \draw  (1,\x) node {$\circ$} ;
\draw  (4,-\x) node {$\circ$} ; \draw  (4,\x) node {$\circ$} ;

\draw  (0,-\y) node {$\circ$} ; \draw  (0,\y) node {$\circ$} ;
\draw  (1,-\y) node {$\circ$} ; \draw  (1,\y) node {$\circ$} ;
\draw  (4,-\y) node {$\circ$} ; \draw  (4,\y) node {$\circ$} ;
\end{tikzpicture}
}
\caption{}
\end{subfigure}%
\begin{subfigure}{0.25\linewidth} 
\centering
 \resizebox{!}{3.cm}{
\begin{tikzpicture}
\draw  (0,\y) node {} ;\draw  (0,-\y) node {} ;
\draw[<->,shorten >=0.1cm,shorten <=.1cm] 	(0,-\x)    	-- 	(0,\x);
\draw[<->,shorten >=0.1cm,shorten <=.1cm] 	(1,-\x)    	-- 	(1,\x);
\draw[<->,shorten >=0.1cm,shorten <=.1cm] 	(2,-\x)    	-- 	(2,\x);

\draw (0,0) node [left] {\small{$c_1$}}; \draw (1,0) node [left] {\small{$c_2$}}; \draw (2,0) node [left] {\small{$c_3$}};

\draw  (0,-\x) node {$\circ$} ; \draw  (0,\x) node {$\bullet$} ;
\draw  (1,-\x) node {$\circ$} ; \draw  (1,\x) node {$\bullet$} ;
\draw  (2,-\x) node {$\circ$} ; \draw  (2,\x) node {$\bullet$} ;
\draw  (.5,-\x) node {$=$} ; \draw  (1.5,-\x) node {$=$} ;
\draw (0,\x) node [right] {\small{$\mu_1$}};
\draw (1,\x) node [right] {\small{$\mu_2$}};
\draw (2,\x) node [right] {\small{$\mu_3$}};
\end{tikzpicture}
}
\caption{}
\end{subfigure}%
\begin{subfigure}{0.09\linewidth} 
\centering
 \resizebox{!}{3.cm}{
\begin{tikzpicture}
\draw (0,0) node [left] {$c$};
\draw (0,\x/2+\y/2) node [left] {\small{$\phi^X$}};
\draw (0,-\x/2-\y/2) node [left] {\small{$\phi^Y$}}; 
\draw (0,\y) node [right] {\small{$\mu$}};
\draw (0,-\y) node [right] {\small{$\nu$}}; 
\draw[<->,shorten >=0.1cm,shorten <=.1cm] 	(0,-\x)    	-- 	(0,\x);
\draw[dotted,shorten >=0.1cm,shorten <=.1cm] 	(0,\x)    	-- 	(0,\y);
\draw[dotted,shorten >=0.1cm,shorten <=.1cm] 	(0,-\x)    	-- 	(0,-\y);
\draw  (0,-\x) node {$\circ$} ; \draw  (0,\x) node {$\circ$} ;
\draw  (0,-\y) node {$\bullet$} ; \draw  (0,\y) node {$\bullet$} ;
\draw  (-.5,0) node {} ;
\draw  (.5,0) node {} ;
\end{tikzpicture}
}
\caption{}
\end{subfigure}%
\begin{subfigure}{0.25\linewidth} 
\centering
 \resizebox{!}{3.cm}{
\begin{tikzpicture}
\draw (0,0) node [left] {$c$};
\draw (0,\x/2+\y/2) node [left] {\small{$\phi_{\KL}$}};
\draw (0,-\x/2-\y/2) node [left] {\small{$\phi_{\KL}$}}; 
\draw (0,\y) node [left] {\small{$\rho^1_k$}};
\draw (0,-\y) node [left] {\small{$\rho^1$}}; 
\draw[<->,shorten >=0.1cm,shorten <=.1cm] 	(0,-\x)    	-- 	(0,\x);
\draw[dotted,shorten >=0.1cm,shorten <=.1cm] 	(0,\x)    	-- 	(0,\y);
\draw[dotted,shorten >=0.1cm,shorten <=.1cm] 	(0,-\x)    	-- 	(0,-\y);
\draw  (0,-\x) node {$\circ$} ; \draw  (0,\x) node {$\circ$} ;
\draw  (0,-\y) node {$\circ$} ; \draw  (0,\y) node {$\bullet$} ;

\draw (1,0) node [right] {$c$};
\draw (1,\x/2+\y/2) node [right] {\small{$\phi_{\KL}$}};
\draw (1,-\x/2-\y/2) node [right] {\small{$\phi_{\KL}$}}; 
\draw (1,\y) node [right] {\small{$\rho^2_k$}};
\draw (1,-\y) node [right] {\small{$\rho^2$}}; 
\draw[<->,shorten >=0.1cm,shorten <=.1cm] 	(1,-\x)    	-- 	(1,\x);
\draw[dotted,shorten >=0.1cm,shorten <=.1cm] 	(1,\x)    	-- 	(1,\y);
\draw[dotted,shorten >=0.1cm,shorten <=.1cm] 	(1,-\x)    	-- 	(1,-\y);
\draw  (1,-\x) node {$\circ$} ; \draw  (1,\x) node {$\circ$} ;
\draw  (1,-\y) node {$\circ$} ; \draw  (1,\y) node {$\bullet$} ;

\draw[color=black!80] (-.5,-\y+.25) rectangle (1.5,-\y-.25);
\draw[right]  (1.5,-\y) node {\small{$\mathcal{G}$}} ;
\draw  (-1,0) node {} ;
\draw  (2,0) node {} ;
\end{tikzpicture}
}
\caption{}
\end{subfigure}%

\caption{Representation of some problems of the form \eqref{eq-general}. Each circle represents a measure (filled with black if is fixed by an equality constraint); the terms in the functional $\mathcal{J}$, see \eqref{eq-general}, are represented by arrows (for transport terms), dotted lines (for divergence terms) and rectangles (for the $\mathcal{G}$ terms). (A) With $\mathcal{F}_1$ and $\mathcal{F}_2$ of the form \eqref{eq_marginal_functional} (B) Balanced barycenter (C) Unbalanced optimal transport (D) $\WF$ gradient flow of 2 species with interaction.}
\label{fig_generic_problem}
\end{figure}
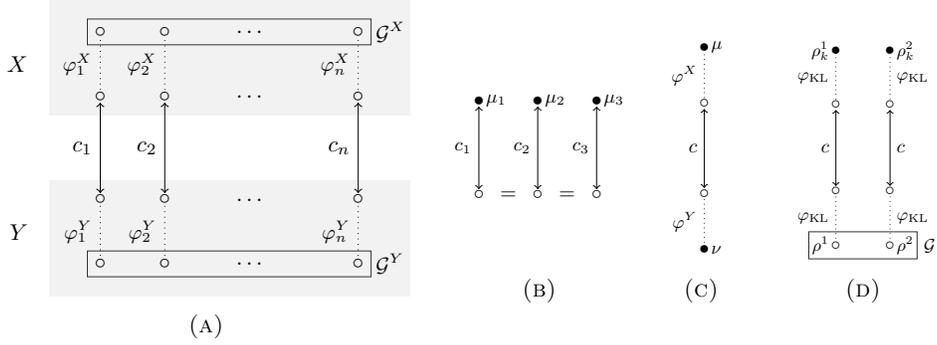


\section{Entropic Regularization and Iterative Scaling Algorithm}
\label{sec_algorithm_analysis}
In this section, we introduce and describe an iterative scaling algorithm which solves a regularized version of the generic variational problem \eqref{eq-general}. This analysis is carried out in a continuous setting.
%
\subsection{Entropic Regularization}
\label{subsec:EntropicRegularization}

Following several recent articles, discussed in Section~\ref{sec: previous works}, we consider an approximation of problem \eqref{eq-general} where an entropy term is used in place of the nonnegativity constraint on $\gamma$.
Let $\d x$ and $\d y$ be reference probability measures on $X$ and $Y$ and assume that the product measure $\d x \d y$ is the reference measure on $X\times Y$ (one could consider more general reference measures on $X\times Y$ but this would require to invoke the concept of disintegration of measures).
We define (minus) the entropy of a family of couplings $(\gamma_k)_{k=1}^n$ as 
\eq{
\mathcal{H}(\gamma) \eqdef \sum_{k=1}^n \int_{X\times Y} r_k(\log(r_k)-1)\d x \d y\, .
}
if each $\gamma_k$ admits a nonnegative density $r_k$ w.r.t. $\d x \d y$ (with the convention $0\log 0 = 0$) and $\infty$ otherwise.
Note that each term is, up to a constant, the Kullback-Leibler divergence $\KLm$ of $\gamma_i$ with respect to $\d x \d y$, as defined in Example \ref{ex_KLdiv}.
Given a small regularization parameter $\epsilon>0$, we consider 
\eql{\label{eq-general-regul}
	\umin{\gamma \in \Mm^n(X \times Y)}
		 \mathcal{J}(\gamma)
		 + \epsilon \, \mathcal{H}(\gamma).
}
Recalling that $ \mathcal{J}(\gamma)= \langle c , \gamma \rangle 
+ \mathcal{F}_1(P^X_\# \gamma) 
+ \mathcal{F}_2(P^Y_\# \gamma)\,$, this can be rewritten, up to a constant, as
\eql{\label{eq-general-regul_bis}
	\umin{\gamma \in \Mm^n(X \times Y)}
		 \mathcal{F}_1(P^X_\# \gamma) + \mathcal{F}_2(P^Y_\# \gamma)
		 + \epsilon \sum_{k=1}^n \KLm(\gamma_k|e^{-c_k/\epsilon}\d x \d y).
}
Adding the entropy term can be interpreted in the following ways, detailed for $n=1$ for simplicity.

\paragraph{Interpretation 1.}
The resolution of~\eqref{eq-general-regul} can be understood as a regularization scheme, since the term $\epsilon\, \mathcal{H}(\gamma)$ makes the problem strictly convex.
In the limit $\epsilon \rightarrow 0$, the unique solution to~\eqref{eq-general-regul} should converge to a solution of the original problem. This is shown in the case of balanced transport~\eqref{eq-ot} in~\cite{2015-carlier-convergence,leonard2012schrodinger}. 
In our setting, a general result is beyond the scope of this article. Nevertheless, the finite dimensional case is insightful: if $\d x \d y$ has full support, then one recovers the minimizer of $\mathcal{J}$ of minimal entropy.
\begin{proposition}
\label{prop_converg_regul}
Consider the case where $n=1$ (one coupling) and $X$ and $Y$ are finite spaces. Assume that $\mathcal{J}$ is convex, lower semicontinuous and that $A \eqdef \argmin_{\gamma \ll \d x \d y} \mathcal{J}(\gamma)$ is nonempty. Let $(\epsilon_k)_{k\in \NN}$ be a sequence of strictly positive real numbers converging to $0$. Then the sequence $(\gamma_k)_{k\in \NN}$ of minimizers of \eqref{eq-general-regul} is well defined and converges to $\argmin_{\gamma \in A} \mathcal{H}(\gamma)$.
\end{proposition}
\begin{proof}
Since the objective function in \eqref{eq-general-regul} is coercive, strictly convex, lower semicontinuous and its domain is nonempty, there exists a unique minimizer $\gamma_k$ of \eqref{eq-general-regul} associated to $\epsilon_k$ for all $k \in \NN$. Now, let $\bar{\gamma}\in A$. For all $k\in \NN$, the optimality of $\gamma_k$ implies
\[
\gamma_k \in \enscond{ \gamma \in \Mm_+(X \times Y)}{\mathcal{H}(\gamma) \leq \mathcal{H}(\bar{\gamma}) }
\]
where $\mathcal{H}(\bar{\gamma})$ is finite since it is a finite sum of real numbers. As $\mathcal{H}$ is coercive and lower semicontinuous, this set is compact. This implies the existence of cluster points for $(\gamma_k)$: let $\gamma^*$ be one of them. 
One has that $\gamma^*$ belongs to $A$ since  for all $k$, $|\mathcal{J}(\gamma_k)-\mathcal{J}(\bar{\gamma})|\leq 2\epsilon_k \mathcal{H}(\bar{\gamma}) \to 0$ and $\mathcal{J}(\gamma^*) \leq \liminf_{k \to \infty} \mathcal{J}(\gamma_k)$. 
Moreover, as $\mathcal{H}(\gamma^*) \leq \mathcal{H}(\bar{\gamma})$ and $\bar{\gamma}$ is arbitrarily chosen in $A$, it holds $\gamma^* \in \argmin_{\gamma \in A} \mathcal{H}(\gamma)$. 
By strict convexity, this cluster point is unique and $\gamma_k \to \gamma^*$.
\end{proof}

\paragraph{Interpretation 2.}
A second way to interpret~\eqref{eq-general-regul} is as a proximal step for the Kullback-Leibler divergence. Indeed, by denoting $\gamma^{(0)}\eqdef\d x \d y$ and $\gamma^{(1)}$ the solution to~\eqref{eq-general-regul}, one has
$\gamma^{(1)} = \prox^{\KL}_{\mathcal{J}/\epsilon}(\gamma^{(0)})$. It is possible to iterate this relation and define
\eq{
	\gamma^{(\ell+1)} \eqdef  \prox^{\KL}_{\mathcal{J}/\epsilon}(\gamma^{(\ell)}).
}
which is the so-called proximal-point algorithm for the $\KL$ divergence, known to converge to the solution of~\eqref{eq-general} under the same conditions as in Proposition \ref{prop_converg_regul} (see \cite{EcksteinProxPoint}).
It is a simple exercise to show, that for balanced optimal transport performing $\ell$ proximal steps with stepsize $\epsilon$ corresponds to a single step with stepsize $\epsilon/\ell$.
While this precise relation is not true for the more general transport problems discussed here, we still find in practice that with small enough $\epsilon$ a single iteration yields sufficient accuracy (see Section \ref{sec:LogDomain} for numerical handling of small $\epsilon$ values).
\paragraph{Interpretation 3.}
Beyond computational aspects, for the specific case of standard OT with quadratic cost, the entropic regularization admits several interpretations as a stochastic version of OT. For instance, it is obtained by considering an optimal matching problem where there is (a specific model of) unknown heterogeneity in the preferences within each class to be matched \cite{Galichon-Entropic}. What is more, it amounts to computing the law of motion of the (indistinguishable) particules of a gaz which follow a Brownian motion, conditionally to the observation of its density at times $t=0$ and $t=1$, the so-called Schr\"odinger bridge problem \cite{LeonardSchroedinger}.
\paragraph{Choice of the reference measure.}
The choice of $\d x$ and $\d y$ is critical for obtaining the weak convergence to a solution of \eqref{eq-general} when $\epsilon \to 0$. An obvious necessary condition is that the support of $\d x \d y$ should contain the support of an optimizer of the unregularized problem \eqref{eq-general}.
For instance, when solving an unbalanced optimal transport problem between $\mu$ and $\nu$, the choice $\d x = \mu/\mu(X)$ and $\d y = \nu/\nu(Y)$ is not suitable if one of the divergences is not superlinear because  $\Divergm_\phi (\sigma | \mu)<\infty$ does not imply $\sigma\ll\mu$. In this case, the support of the reference measure $\d x \d y$ should contain the sets (well defined under the hypotheses of Proposition \ref{prop_UOT_existence})
\eq{
	\{ (x,\argmin_y c(x,y)) ; x\in X ), \qandq \{(\argmin_x c(x,y),y ) ; y\in Y \}\, .
}

Also, for the barycenter \eqref{eq_barycenter} or the gradient flow problems~\eqref{eq-grad-flow}, one does not have access in general to a reference measure according to which an optimizer is absolutely continuous. For numerics, the choice of the reference measures then corresponds to the choice of a discretization grid on which we find approximate solutions to the original problem.


\subsection{Reformulation using Densities and Duality}
\label{sec_density_setting}
From the definition of the entropy $\mathcal{H}$, any feasible $\gamma$ of the generic problem \eqref{eq-general-regul} admits an $\Lun$ density with respect to the reference measure $\d x \d y$. Accordingly, for the convenience of the analysis, we reformulate \eqref{eq-general-regul_bis} as a variational problem on measurable functions:
\eql{\label{eq-general-regul-KL}
\tag{$P_{\epsilon}$}
	\umin{r \in \Lun(X\times Y)^n} 
		 F_1( P^X_\# r )
		 + F_2( P^Y_\# r  )
		 + \epsilon \KL(r |K).
}
where $F_1(s) \eqdef \mathcal{F}_1(s\d x)$, $F_2(s) \eqdef \mathcal{F}_2(s\d y)$, $K \in \Linf_+(X\times Y)^n$ is defined componentwise by
\begin{equation}
	\label{eq_GibbsKernel}
	K_k(x,y) \eqdef e^{-\frac{c_k(x,y)}{\epsilon}}
\end{equation}
for $k\in \{1,\dots,n\}$, and with the convention $\exp (-\infty) = 0$, the projection operator acts on each component of $r$ and $\KL$ is the sum of the Kullback-Leibler divergences on each component (see the notations in Section \ref{sec_notation}).
In this Section, we only make the following general assumptions on the objects involved in \eqref{eq-general-regul-KL}:
\begin{assumptions}\hfill
\begin{enumerate}[(i)]
\item $(X,\d x)$ and $(Y,\d y)$ are probability spaces (i.e.\ measured spaces with unit total mass). The product space $X\times Y$ is equipped with the product measure $\d x \d y \eqdef \d x \otimes \d y$;
	\label{asp_reference_measure}
\item $F_1 : \Lun(X)^n \to \RR \cup \{\infty\}$ and $F_2 : \Lun(Y)^n \to \RR \cup \{\infty\}$ are weakly lower semicontinuous, convex and proper functionals;
\item $K\in \Linf_+(X\times Y)^n$ and $\epsilon>0$.
\end{enumerate}
\end{assumptions}
We begin with a general duality result, which is an application of Fenchel-Rockafellar Theorem (see Appendix \ref{sec:ApxConvex}). It is similar to duality results that can be found in the literature on entropy minimization \cite{borwein1991duality} but the functional we consider is more general.
\begin{theorem}[Duality]
\label{prop_duality}
	The dual problem of \eqref{eq-general-regul-KL} is
		\eql{\label{eq-dual-pbm}
		\tag{$D_{\epsilon}$}
		\usup{ (u,v) \in \Linf(X)^n\times \Linf(Y)^n } - F_1^*(-u)-F_2^*(-v) - 
		\epsilon \dotp{ e^{(u \oplus v)/\epsilon}-1}{ K  } 
	}
where $u \oplus v : (x,y) \mapsto u(x)+v(y)$. Strong duality holds, i.e.\
$\min \eqref{eq-general-regul-KL} = \sup \eqref{eq-dual-pbm}$ and the minimum of \eqref{eq-general-regul-KL} is attained for a unique $r = (r_k)_{k=1,\dots,n} \in \Lun(X\times Y)^n$. Moreover, $u$ and $v$ maximize \eqref{eq-dual-pbm} if and only if
	\eql{\label{eq-primal-dual}
	\begin{cases}
		-u \in \partial F_1 (P^X_\# r) &\\
		  -v \in \partial F_2 (P^Y_\# r) &
	\end{cases}
		 \qandq
		r_k(x,y) = e^{\frac{u_k(x)}\epsilon} K_k(x,y) e^{\frac{v_k(x)}\epsilon}    \text{ for } k =1,\dots,n.
	}%
\end{theorem}
\begin{proof}
In this proof, the spaces $\Linf$ and $(\Linf)^*$ are endowed with the strong and the weak* topology respectively.
As $\Lun(X\times Y)^n$ can be identified with a subset of the topological dual of $\Linf (X\times Y)^n$, the function $\KL(\cdot|K)$ can be extended on $(\Linf (X\times Y)^{n})^*$ as $G(r)=\KL(r|K)$ if $r\in \Lun(X\times Y)^n$ and $+\infty$ otherwise.
Its convex conjugate $G^*: \Linf(X\times Y)^n \to \RR$ is
\eq{
G^*(w)=\sum_{k=1}^n \int_{X\times Y}(e^{w_k(x,y)}-1)K_k(x,y) \d x \d y = \langle e^w-1,K \rangle \, .
}
and is everywhere continuous for the strong topology \cite[Theorem 4]{rockafellar1968integrals} (this property relies on the finiteness of $\d x \d y$ and the boundedness of $K$).
The linear operator $A: \Linf(X)^n \times \Linf(Y)^n \to \Linf(X \times Y)^n$ defined by $A(u,v) : (x,y)\mapsto u(x)+v(y)$ is continuous and its adjoint is defined on $(\Linf(X\times Y)^n)^*$ (identified with a subset of $\Mm(X\times Y)^n$) by $A^*(r) = (P^X_\# r , P^Y_\# r)$. 
Since $F_1$ and $F_2$ are convex, lower semicontinuous, proper and since $G^*$ is everywhere continuous on $\Linf(X\times Y)^n$, strong duality and the existence of a minimizer for \eqref{eq-general-regul-KL} is given by Fenchel-Rockafellar theorem (see Appendix \ref{sec:ApxConvex}). More explicitely, this theorem states that
\eq{
\sup_{(u,v) \in \Linf(X)^n\times \Linf(Y)^n }  -F_1^*(-u)-F_2^*(-v) - \epsilon G^* (A(u,v)/\epsilon) 
}
and 
\eq{
\min_{r \in (\Linf(X\times Y)^n)^* }  F_1(P^X_\# r)+F_2(P^Y_\# r) + \epsilon G (r) 
}
are equal, the latter being exactly \eqref{eq-general-regul-KL} since $G$ is infinite outside of $\Lun(X\times Y)^n$. It states also that if $(u,v)$ maximizes \eqref{eq-dual-pbm}, then any minimizer of \eqref{eq-general-regul} satisfies
$r \in \partial G^*(A(u,v)/\epsilon)$ and the expression for the subdifferential of $G^*$ is an application of the result in Appendix \ref{sec:ApxDivergences}. Finally, uniqueness of the minimizer for \eqref{eq-general-regul-KL} comes from the strict convexity of $G$. 
\end{proof}

\subsection{Scaling Algorithm}
\label{sec:EntropyScaling}
The specific splitting of the problem~\eqref{eq-general-regul-KL} makes it suitable for the well-known Dykstra's algorithm (see Section~\ref{sec: previous works} for more background on this algorithm). Since the functions $F_1$ and $F_2$ operate on the marginals only, Dykstra's iterations take a very simple form which is related to the celebrated Sinkhorn algorithm. They are obtained as an alternating maximization on~\eqref{eq-dual-pbm}.

%
Let $\Gibbs$ and $\transp{\Gibbs}$ be the linear operators defined, for $a:X\to [0,\infty[^n$ and $b:Y\to [0,\infty[^n$ measurable, for $k=1,\dots,n$ as
\begin{align}
	\label{eq-gibbs-convolution}
	(\Gibbs b)(x)_k \eqdef \int_Y K_k(x,y)  b_k(y) \d y \qandq
	(\transp{\Gibbs} a)(y)_k \eqdef \int_X K_k(x,y) a_k(x) \d x\, .
\end{align}
Given $\itz{v} \in \Linf(Y)^n$, alternating maximizations on the dual problem~\eqref{eq-dual-pbm} define, for $\ell=0,1,2,\ldots$, the iterates
\eql{
\label{eq_alternating_max}
\left\{\begin{aligned}
\iiter{u}  &= \argmax_{u\in \Linf(X)^n} - F_1^*(-u) - \epsilon \langle e^{\frac{u}\epsilon}, \Gibbs e^{\frac{\iter{v}}\epsilon} \rangle_X \,, \\
\iiter{v}  &= \argmax_{v\in \Linf(Y)^n} - F_2^*(-v) - \epsilon  \langle e^{\frac{v}\epsilon}, \transp{\Gibbs} e^{ \frac{\iiter{u}}\epsilon} \rangle_Y \,,
\end{aligned}
\right.
}
where we used the fact that, by Fubini-Tonelli, one has
\eql{\label{eq_fubinigibbs}
	\langle e^{(u\oplus v)/\epsilon}, K \rangle_{X \times Y}
	= \langle e^{\frac{u}\epsilon}, \Gibbs e^{\frac{v}\epsilon}\rangle_X
	= \langle e^{\frac{v}\epsilon}, \transp{\Gibbs} e^{ \frac{u}\epsilon} \rangle_{Y} \, .
}
Conditions which ensure the existence of these iterates are given later in Theorem \ref{thm-sinkhorn-div}; for now, remark that by strict convexity, they are uniquely defined when they exist.
Given an initialization $b^{(0)}\in \Linf_+(Y)$, the main iterations of this paper are defined as follows, for $\ell =0,1,2,\ldots$:

\bigskip
\fbox{
\begin{minipage}[c]{.9\textwidth} 
\textbf{Scaling iterations:}\hfill
	\begin{align}\label{eq-scaling-iterates}\tag{S}
		\iiter{a} \eqdef \frac{ 
				\prox^{\KL}_{F_1/\epsilon}(\Gibbs \iter{b})  
			}{
				\Gibbs \iter{b}
			}
			\, , \quad	
		\iiter{b} \eqdef \frac{ 
				\prox^{\KL}_{F_2/\epsilon}(\transp{\Gibbs} \iiter{a})  
			}{
				\transp{\Gibbs} \iiter{a}
			}
	\end{align}
\end{minipage}
}
\bigskip

where the proximal operator for the $\KL$ divergence is defined for $F_1$ (and similarly for $F_2$) as
\eq{
	\prox^{\KL}_{F_1/\epsilon}(z) \eqdef \argmin_{\substack{s :X \to \RR^n \\ \text{measurable}}} F_1(s) + \epsilon \KL(s|z) \, .
}
The following Proposition shows that, as long as they are well defined, these iterates are related to the alternate dual maximization iterates.
\begin{proposition}
\label{prop_alternating}
Define $a^{(0)}=\exp (v^{(0)})$, let $(\iter{a},\iter{b})$ be the scaling iterates \eqref{eq-scaling-iterates} and $(\iter{u},\iter{v})$ the alternate dual maximization iterates defined in~\eqref{eq_alternating_max}.
If for all $\ell \in \NN$, either $\log \iter{a} \in \Linf (X)^n$ and $\log \iter{b} \in \Linf(Y)^n$, or $\iter{u}\in \Linf(X)^n$ and $\iter{v}\in \Linf(Y)^n$ then 
	\eq{
		(\iter{a},\iter{b})=(e^{\iter{u}/\epsilon},e^{\iter{v}/\epsilon}).
	}
\end{proposition}
The proof of this proposition makes use of the following Lemma. 
\begin{lemma}
\label{lem_KL_L1}
Let $(T,\d t)$ be a measure space and $v \in \Lun_{+}(T)$. For any $u : T \to \RR$ measurable, if $\KL(u|v)<\infty$ then $u\in \Lun_{+}(T)$.
\end{lemma}
\begin{proof}
Without loss of generality, one can assume $v$ positive since for $\d t$-a.e. $t$, if $v(t)=0$ then $u(t)=0$. 
The subgradient inequality at $1$ gives, for all $s\in \RR$, $s \leq \phi_{\KL}(s)+e^1-1$. 
Consequently,
\eq{
\int_T u \d t = \int_T (u(t)/v(t))v(t) \d t \leq \int_T \left(\varphi_{\KL}(u(t)/v(t))+e^1-1\right) v(t) \d t < \infty \, .
}
\end{proof}

\begin{proof}[Proof of Proposition \ref{prop_alternating}.]
Suppose that $\iter{v} \in \Linf(Y)^n$ and that $\iter{b}=e^{\iter{v}/\epsilon}$. One has $\Gibbs \iter{b} \in \Lun(X)^n$ and from Lemma \ref{lem_KL_L1}, one can compute $\prox^{\KL}_{F_1/\epsilon}(\Gibbs \iter{b})$ in $\Lun(X)^n$.
The Fenchel-Rockafellar Theorem gives (see the proof of Theorem \ref{prop_duality} for more details): 
\eq{
\sup_{u \in \Linf(X)^n} -F_1^*(-u) - \epsilon \langle e^{\frac{u}\epsilon}, \Gibbs e^{\frac{\iter{v}}\epsilon} \rangle
= \min_{s \in \Lun(X)^n} F_1(s) + \epsilon \KL (s|\Gibbs e^{\frac{\iter{v}}\epsilon})
}
and the optimality conditions state that $u^\star$ maximizes the problem on the right if and only if the minimizer $s^\star \eqdef \prox^{\KL}_{F_1/\epsilon}(\Gibbs e^{\frac{\iter{v}}\epsilon})$ of the problem on the left belongs to the subdifferential of $u \mapsto \langle e^{\frac{u}\epsilon}, \Gibbs e^{\frac{\iter{v}}\epsilon}\rangle$ at the point $u^\star$. 
That is (see Appendix \ref{sec:ApxDivergences}) if and only if for $\d x$ almost every $x\in X$, 
\eq{
s^\star(x) = e^{u^\star(x)/\epsilon} \cdot (\Gibbs e^{\iter{v}/\epsilon})(x)
}
Thus if $\epsilon \log \iiter{a}$ belongs to $\Linf(X)^n$ or if $u^\star\in \Linf(X)^n$ exists, then $u^\star=\epsilon \log \iiter{a}$. The rest of the proof is done by induction.
\end{proof}

\subsection{Existence of the iterates for integral functionals}

Our next step is to give conditions on $F_1$ and $F_2$ that guarantee the existence of the scaling iterates \eqref{eq-scaling-iterates} and an equivalence with alternate maximization on the dual~\eqref{eq_alternating_max}. This is provided by Theorem \ref{thm-sinkhorn-div}, where it is required that $F_1$ and $F_2$ are integral functionals, as we define now. 

\begin{definition}[Normal integrands and Integral functionals~\cite{rockafellar2009variational}]
\label{def_integralfunctional}
A function $f : X \times \RR^n \to \RR \cup \{\infty\}$ is called a \emph{normal integrand} if its epigraphical mapping $X \ni x \mapsto \epi f(x,\cdot)$ is closed-valued and measurable. A convex integral functional is a function $F:\Lun(X)^n\to \RR\cup \{\infty\}$ of the form
\eq{
\label{eq_F_separable}
F (s) = I_{f}(s) \eqdef \int_X f(x,s(x)) \d x
}
where $f$ is a normal integrand and $f(x,\cdot)$ is convex for all $x\in X$. 
In this paper, $F$ is an \emph{admissible} integral functional if moreover for all $x\in X$, $f(x,\cdot)$ takes nonnegative values, has a domain which is a subset of $ [0,\infty[^n$ and if there exists $s\in \Lun(X)^n$ such that $I_f(s)<\infty$.
\end{definition}
The concept of normal integrands allows to deal conveniently with measurability issues. For finite dimensional problems (when $X$ and $Y$ have a finite number of points), integral functionals are simply sums of pointwise lower semicontinuous functions. The following proposition shows that for such functionals, conjugation and subdifferentiation can be performed pointwise.
\begin{proposition}\label{prop_normalconjug}
If $F$ is an admissible integral functional associated to the convex normal integrand $f$, then $F$ is convex and weakly lower semicontinuous, $f^*$ is also a normal convex integrand, $F^*=I_{f^*}$ and
\begin{align*}
\partial F(s) &= \{ u \, ; \, u(x) \in \partial f(x,s(x)), \d x \text{ a.e.}\},  \\
\partial F^*(u) &= \{ s \, ; \, s(x) \in \partial f^*(x,u(x)), \d x \text{ a.e.}\}, 
\end{align*}
where conjugation and subdifferentiation on $f$ are w.r.t.\ the second variable.
\end{proposition}
\begin{proof}
This property can be found in~\cite{rockafellar1976integral} under the assumption of existence of a feasible point $s^\star\in \Lun(X)^n$ for $I_f$ and a feasible point $u^\star\in\Linf(X)^n$ for $I_{f^*}$. Our admissibility criterion requires the existence of $s^\star$ and one has 
\eq{
I_{f^*}(0) = \int_X f^*(x,0) \d x = - \int_X  \inf_{s\in \RR^n} f(x,s) \d x <\infty
}
since $\inf_s f(x,s) \in [0, f(x,s^\star(x))]$.
\end{proof}
We now prove a general result on the ``separability'' of the $\KL$ proximal operator.
The Kullback-Leibler divergence between two vectors $s$ and $z$ in $\RR^n$ is defined as 
\eql{
\label{eq_KLpointwise}
\KLp(s|z) \eqdef \sum_{k=1}^n s_k \log \left( \frac{s_k}{z_k} \right)-s_k + z_k
}
if $(z_k=0) \Rightarrow (s_k=0)$, and $+\infty$ otherwise, with the convention $0 \log(0/0)=0$. While we use a separate notation for the sake of clarity, this definition is actually consistent with the definition of $\KL$ if one interprets $s$ and $z$ as vector densities on a space $(X,\d x)$ which is a singleton with unit mass. 
\begin{proposition}
\label{prop_proxKLcont}
Let $s : X \to \RR^n$ be measurable. If $F=I_f$ is an admissible integral functional then for almost all $x\in X$,
\eq{
\big(\prox^{\KL}_{\frac{1}{\epsilon}F}(s)\big)(x) = \prox^{\KLp}_{\frac{1}{\epsilon}f(x,\cdot)}(s(x)) \, .
}
%
\end{proposition}
\begin{proof}
The problem which defines the proximal operator is that of minimizing $ I_f(z) + \KL (z|s)\eqdef I_g(z)$ over measurable functions $z:X\to \RR^n$ with
\eq{
	g : (x,z)\in X\times \RR^n \mapsto f(x,z) + \KLp(z|s(x))\, .
}
The function $(x,z) \mapsto \KLp(z|s(x))$ is a convex normal integrand by~\cite[Prop. 14.30 and 14.45c]{rockafellar2009variational}. Thus $g$ is itself a normal convex integrand, as the sum of normal convex integrands~\cite[Prop. 14.44]{rockafellar2009variational}.  Then a minimization interchange result~\cite[Thm. 14.60]{rockafellar2009variational}] states that minimizing $I_g$ is the same as minimizing $g$ pointwise. 
%
\end{proof}
By Proposition \ref{prop_normalconjug}, if $F_1$ and $F_2$ are admissible integral functionals then $F^*_1$ and $F^*_2$ are also integral functionals. So the alternating optimization on \eqref{eq-dual-pbm} can be relaxed to the space of measurable functions and still have a meaning:
\eql{
\label{eq_alternating_max_point}
\left\{\begin{aligned}
\iiter{u}  &= \argmax_{u:X\to \RR^n} - I_{f_1^*}(-u) - \epsilon \langle e^{\frac{u}\epsilon}, \Gibbs e^{\frac{\iter{v}}\epsilon} \rangle_X \\
\iiter{v}  &= \argmax_{v:Y\to \RR^n} - I_{f_2^*}(-v) - \epsilon  \langle e^{\frac{v}\epsilon}, \transp{\Gibbs} e^{ \frac{\iiter{u}}\epsilon} \rangle_Y\, .
\end{aligned}
\right.
}
The following Theorem gives existence, uniqueness of this iterates and a precise relation with the scaling iterates \eqref{eq-scaling-iterates}.
\begin{theorem}
\label{thm-sinkhorn-div}
Let $F_1$ and $F_2$ be admissible integral functionals as in Definition~\eqref{def_integralfunctional} associated to the normal integrands $f_1$ and $f_2$.
Assume that for all $x\in X$ and $y\in Y$, there exists points $s_1$ and $s_2$ with strictly positive coordinates such that $f_1(x,s_1)<\infty$ and $f_2 (y,s_2)<\infty$ and that $K$ takes positive values.
Define $a^{(0)}=1$ and let $(\iter{a},\iter{b})$ be the scaling iterates \eqref{eq-scaling-iterates}.
Then, with initialization $v^{(0)}=0$, the iterates $(\iter{u},\iter{v})$ in~\eqref{eq_alternating_max_point} are uniquely well-defined and for all $\ell \in \NN$ one has $(\iter{a},\iter{b}) = (e^{\frac{\iter{u}}{\epsilon}},e^{\frac{\iter{v}}{\epsilon}})$.
%
\end{theorem}
\begin{proof}
Suppose that $\iter{v}: Y\to \RR^n$ is a well-defined measurable function and that $\iter{b}=e^{\iter{v}/\epsilon}$. As $K$ is positive, $\Gibbs \iter{b}$ is positive $\d x$ a.e. Let $\iiter{a}$ be computed with  \eqref{eq-scaling-iterates}. Thanks to Proposition \ref{prop_proxKLcont}, the proximal operator can be decomposed as pointwise optimization problems and our assumptions allows to apply Fenchel-Rockafellar (see Appendix~\ref{sec:ApxConvex}) in the case where both problems reach their optima
\eq{
\max_{u\in \RR^n} - f_1^*(x,-u) - \epsilon \langle e^{u/\epsilon}, \Gibbs \iter{b} \rangle 
= \min_{s\in \RR^n} f_1(x,s) + \epsilon \KLp(s|\Gibbs \iter{b}(x))
}
with the relation between optimizers: $e^{u/\epsilon} = s/\Gibbs \iter{b}(x)$. 
This formula guarantees that the function of pointwise maximizers $x\mapsto u(x)$ is measurable since the function of pointwise minimizers $x\mapsto s(x)$ is uniquely well-defined and measurable by \cite[Thm. 14.37]{rockafellar2009variational}. Indeed, the minimized function is a strictly convex normal integrand because it is shown in Appendix \ref{sec:ApxDivergences} that $\KLp(\cdot|\Gibbs \iter{b}(x))$ is a normal integrand and sum of normal integrands are normal \cite[Prop. 14.44]{rockafellar2009variational}.
This shows that $\iiter{a}=e^{\frac{\iiter{u}}{\epsilon}}$ and one concludes by induction.
\end{proof}

\subsection{Convergence Analysis}
This Section gives a fixed point result and a convergence result for a particular case. The finite dimensional case is postponed to the next Section.

The following proposition sheds some light on the name ``scaling'' given to iterations~\eqref{eq-scaling-iterates}. It comes from the fact that these iterations allow to recover a solution to~\eqref{eq-general-regul-KL} by multiplying the kernel $K$ with positive functions, interpreted as scalings.

\begin{proposition}
\label{prop_fixpoint}
Under the assumptions of Theorem \ref{thm-sinkhorn-div}, if the scaling iterations \eqref{eq-scaling-iterates} admit a fixed point $(a,b)$ such that $\log a \in\Linf(X)^n$ and $\log b \in \Linf(Y)^n$ then $(\epsilon \log a,\epsilon \log b)$ is the unique solution of \eqref{eq-dual-pbm} and the function $r$ defined for each $k=1,\ldots,n$ by $r_k(x,y)= a_k(x) K_k(x,y) b_k(y)$ is the unique solution of \eqref{eq-general-regul-KL}.
\end{proposition}
\begin{proof}
As a consequence of Proposition \ref{prop_proxKLcont}, on can write the optimality condition of a fixed point of \eqref{eq_alternating_max_point} for almost every $x\in X$ as
\eq{
u_k(x) = \epsilon \log \left(\frac{a_k(x)\cdot \Gibbs b_k(x)}{\Gibbs b_k(x)} \right)
}
for some $u (x)\in -\partial f_1(x,a(x)\cdot\Gibbs b(x))$. Thus $-\epsilon \log(a) \in \partial I_{f_1}(P^X_\# r)$ because $P^X_\# r = a \cdot \Gibbs b$ (the dot denotes componentwise multiplication). Similar derivations for $b$ show that the couple $(\epsilon \log a,\epsilon \log b)$ and $r$ satisfies the primal dual optimality conditions \eqref{eq-primal-dual}.
\end{proof}

Sinkhorn's algorithm (the special case of the scaling iterations \eqref{eq-scaling-iterates} obtained when $F_1=\iota_{\{p\}}(s)$ and $F_2 = \iota_{\{q\}}(s)$ are the convex indicators of equality constraints) is known to converge at a linear rate. This property is usually shown (see for instance~\cite{franklin1989scaling}) by using the contraction property of the operator $\Gibbs$ for the Hilbert metric (a projective metric on the cone of nonnegative functions).
Using a similar approach, based on the related \textit{Thompson} metric, we now show the linear convergence of the scaling iterates \eqref{eq-scaling-iterates} in another special case, which is central in unbalanced optimal transport: when $F_1$ and $F_2$ are $\KL$ divergences with respect to fixed densities. 

An approach involving the Hilbert metric would still be possible, but the use of the Thompson metric allows for a very short proof and a convergence rate which does not depend on a bound on the cost. This metric is defined as follows.
\begin{definition}[Thompson part metric on $\Linf_{+}$~\cite{thompson1963certain}]
For $r,s \in \Linf_+(X)$, let $ M(r/s) \eqdef \inf \enscond{ \alpha \geq 0}{r \leq \alpha s}$ (or $\infty$ if that set is empty). The Thomson part metric is defined by 
\eq{
d(r,s) = \max \{ \log M(r/s), \log M(s/r) \}\, .
}
\end{definition}
Key properties of this metric are:
\begin{itemize}
\item if $T: \Linf(X) \to \Linf(Y)$ is a $z$-homogeneous order-preserving (or order-reversing) operator, then $d(Tx,Ty) \leq |z|\cdot d(x,y)$;
\item the equivalence relation $r \sim s \Leftrightarrow d(r,s)<+ \infty$ generates a partition of $(\Linf(X),d)$ and each part is a complete metric space. 
In particular, the set $\enscond{ s \in \Linf_+(X)}{\log s \in \Linf(X)}$ endowed with the Thompson metric form a complete metric space.
\end{itemize}
\begin{theorem}\label{th_convergenceKL}
Let $p$ and $q$ be such that $\log p \in \Linf(X)$ and $\log q \in \Linf(Y)$ and define
\eq{
F_1 = \lambda_1 \KL(\cdot |p) \qandq F_2 = \lambda_2 \KL(\cdot |q)
} 
for some $\lambda_1, \lambda_2>0$. Assume that the kernel $K$ is lower bounded by a positive real number.
If $(\iter{a},\iter{b})$ are the scaling iterates \eqref{eq-scaling-iterates} initialized with $b^{(0)}=1$, then $(x,y)\mapsto \iter{a}(x)K(x,y)\iter{b}(y)$ converges at a linear rate (for the Thompson metric) to a solution of \eqref{eq-general-regul-KL}.
\end{theorem}
%
%
\begin{proof}
Let $z_i \eqdef \lambda_i/(\lambda_i + \epsilon) \in \, ]0,1[$ for $i\in\{1,2\}$ and let $b^{(0)}= 1$. Following a simple application of Proposition \ref{prop_proxKLcont} (or see Table
\ref{table_proximalexplicit}) the iterates in this specific case read 
\eq{
\iiter{a} = \left( \frac{p}{\Gibbs \iter{b}} \right)^{z_1} \qandq \iiter{b} = \left( \frac{q}{\Gibbs \iiter{a}} \right)^{z_2}\, .
}
Our assumptions are such that $d(b^{(1)},b^{(0)})$ and $d(a^{(1)},a^{(0)})$ are finite (this is direct since the logarithms of $b^{(0)}$, $K$, $p$ and $q$ are bounded).
Using the properties of the Thompson metric, it holds
\begin{eqnarray*}
d( \iiter{a},\iter{a}) 
= z_1. d(\Gibbs \iter{b}, \Gibbs \ter{b} ) 
 \leq z_1. d(\iter{b},\ter{b}) \, ,
\end{eqnarray*}
since $\Gibbs$ is an order preserving linear operator. In a similar fashion, we obtain $d(\iiter{b},\iter{b}) \leq z_2. d(\iter{a},\iter{a})$. Thus the application $\iter{a} \mapsto \iiter{a}$ is contractant of factor $z_1.z_2 <1$ and $(\iter{a}, \iter{b})_{\ell}$ converges linearly to some $(a^\star,b^\star) \in \Linf_+(X) \times \Linf_+(Y)$ which is a fixed point of the iterations.
Moreover, $|\log a^\star|$ is bounded because $d(a^{(0)},a^\star)<\infty$, and so is $|\log b^\star|$.
The result follows then by Proposition \ref{prop_fixpoint}.
\end{proof}

\section{Algorithm for Discrete Measures}
\label{sec-algorithm}

In this section, we take a closer look at the scaling iterations~\eqref{eq-scaling-iterates} in the specific case of discrete finite spaces. We give a convergence result, we adapt the notation to obtain an algorithm that can be implemented in a straightforward manner and also discuss how to numerically stabilize the algorithm for small values of the regularization parameter $\epsilon$, to reach a higher precision.

\subsection{Convergence}
In finite dimension\footnote{when $X$ and $Y$ are finite sets, or when the reference measures are finite sums of atoms.}, the scaling iterations~\eqref{eq-scaling-iterates} converge in general. Note that the given convergence rate is pessimistic compared to that observed in practice (which is linear, see Figure \ref{fig_convergenceUOT}).

\begin{theorem}
\label{prop_convergencediscrete}
Assume that $X$ and $Y$ are finite sets and that $K$ takes positive values. Denote by $(\iter a,\iter b)$ the scaling iterates~\eqref{eq-scaling-iterates}, by $(\iter u, \iter v)=(\epsilon \log \iter a, \epsilon \log \iter b)$ the associated dual iterates and by $-I$ the dual functional~\eqref{eq-dual-pbm}. Then one has
\[
\epsilon \KL(r^*|\iter r) \leq I(\iter u, \iter v) - \inf I \leq O(1/\ell)
\]
where $\iter r (x,y) = \iter a (x) K(x,y) \iter b (y)$.
\end{theorem}

\begin{proof}
By Proposition \ref{prop_alternating}, $(\iter u, \iter v)$ are the iterates of an alternate maximization on the dual problem \eqref{eq-dual-pbm}. Moreover (minus) the dual functional has the structure
\[
I(u,v) = F^*(-(u,v)/\epsilon) + \epsilon G^*(A(u,v)/\epsilon)
\]
where $G: r \mapsto \KL(r|K)$ so that $G^*(z)=\langle e^z-1, K \rangle$ is continuously differentiable,  $A$ is defined as $A(u,v)(x,y)=u(x)+v(y)$, and $F^*(u,v)=F^*_1(u)+F^*_2(v)$ is separable. Let us check that (i) $G^*$ is Lipschitz on the sublevel sets of $I$ and (ii) $I$ admits minimizers, so that~\cite[Thm. 3.9]{beck2015convergence} applies and proves $I(\iter u, \iter v) - \inf I \leq O(1/\ell)$.

In order to check (i),  take any feasible point $s$ for $F$. The H\"older inequality $F(s)+F^*(-(u,v)/\epsilon) \geq \langle s, -(u,v)/\epsilon \rangle$ gives
\[
 \epsilon G^*(A(u,v)/\epsilon) \leq I(u,v) +\langle (u,v)/\epsilon, s\rangle  + F(s).
\]
As a consequence, the components of $A(u,v)$ are upper bounded on the set of dual variables $(u,v)$ satisfying
$
I(u,v) \leq I(u^{(1)},v^{(0)})
$
and so $I$ is uniformly Lipschitz on this set.
In order to check (ii), remark that a primal minimizer $r^*$ exists and since $K$ takes positive values, one can build a minimizer by simply inverting the primal-dual relationship $(u^*,v^*) \in A^{-1} (\epsilon \log r^*/K)$.

It remains to prove the first part of the inequality of the Theorem. We follow the strategy developed in \cite[Thm. 1]{fadili2011total}. By direct computations (standard for Bregman divergences), one has for any pair $(z,z^*)$ in the domain of $G^*$, by defining  $r(z)(x,y)=K(x,y)\exp(z(x,y))$, 
\[
\epsilon (G^*(z) - G^*(z^*) - \langle \nabla G^*(z^*), z - z^* \rangle) = \epsilon \KL(r(z^*)| r(z))
\]
which is similar to a strict convexity estimate. Moreover, for any $m \in \partial F^*(-(u^*,v^*)/\epsilon)$, one has by definition of the subdifferential
\[
F^*(-(\iter u,\iter v)/\epsilon) - F^*(-(u^*,v^*)/\epsilon) \geq - \langle m, (\iter u-u^*,\iter v-v^*) \rangle /\epsilon\, . 
\]
By optimality, it holds $0\in \partial I(u^*,v^*)$ so one can chose $m = \epsilon A^*\nabla G^*(A(u^*,v^*)/\epsilon)$.
Then one poses $z = A(\iter u,\iter v)/\epsilon$ and $z^* = A(u^*,v^*)/\epsilon$ in the first inequality, and add it to the second to obtain the result.
\end{proof}

\subsection{Numerical Algorithm for Discrete Measures}
We now adopt a more ``implementation oriented'' point of view, which leads to algorithms which are straightforward to implement. For the sake of clarity, we focus on the case of a single unknown coupling (i.e.\ $n=1$, the extension to $n>1$ merely requires to add an index).

Assume that $X = \{x_1,\ldots,x_I\}$ and $Y=\{y_1,\ldots,y_J\}$ are discrete finite spaces of cardinal $I$ and $J$ respectively, endowed with the reference probability measures $\d x$ and $\d y$ described by nonnegative vectors summing to one
\eq{
\mathbf{ \d x} \eqdef (\d x (\{ x_i\})_{i=1}^I\in \RR_+^I
\qandq
\mathbf{\d y} \eqdef (\d y (\{ y_j\})_{j=1}^J\in  \RR_+^J\, .
}
In this context, one identifies functions $a : X\to \RR\cup \{\infty\}$ or $b:Y\to \RR \cup \{\infty\}$ to vectors $\mathbf{a} =(a(x_i))_{i=1}^I$ and $\mathbf{b} =(b(y_j))_{j=1}^J$ and functions on the product space $r : X\times Y \to \RR\cup \{\infty\}$ to matrices $\mathbf{R} = (r(x_i,y_j))_{i,j}$ (denoted with capital letters for a clearer distinction between vectors and matrices). 
This notation allows for a simple expression for the projections \eqref{eq_ProjectionDens} through
\eq{
P^X_\# \mathbf{R} = \mathbf{R}\, \mathbf{\d y}
\qandq
P^Y_\# \mathbf{R} = \transp{\mathbf{R}}\mathbf{\d x}
}
where $\transp{(\cdot)}$ denotes matrix transposition and $\mathbf{R}\,\mathbf{\d y}$ denotes the standard matrix-vector product. Finally, the notation $\odot$ denotes the entrywise product and $\oslash$ the entrywise division with convention $0/0=0$.

Suppose we are given a cost function $c$ described by the $I\times J$ matrix $\mathbf{C}$ and two convex, lower semicontinuous and lower bounded functions $F_1$ and $F_2$ whose domain is nonempty and included in $\RR_+^I$ and $\RR_+^J$ respectively.
The Gibbs kernel \eqref{eq_GibbsKernel} is then given by $\mathbf{K} = (e^{-\mathbf{C}_{i,j}/\epsilon})_{i,j}$ and, by simply replacing continuous variables by their discrete counterparts, the main variational problem \eqref{eq-general-regul-KL} reads
\eql{
\label{eq_primal_discr}
\min_{\mathbf{R} \in \RR^{I\times J}} 
F_1(\mathbf{R}\, \mathbf{\d y})
+F_2( \transp{\mathbf{R}}\mathbf{\d x})
+ \epsilon \sum_{ij} \KLp(\mathbf{R}_{i,j}| \mathbf{K}_{i,j}),
}
where $\KLp$ is the pointwise Kullback-Leibler divergence (defined in \eqref{eq_KLpointwise}).
Applying the operators $\Gibbs$ and $\transp{\Gibbs}$ defined in \eqref{eq-gibbs-convolution} amounts to
\begin{align*}
	\Gibbs\,\mathbf{b} = \mathbf{K}(\mathbf{b}\odot \mathbf{\d y})
	\qandq 
	\transp{\Gibbs}\,\mathbf{a} = \transp{\mathbf{K}}(\mathbf{a}\odot \mathbf{\d x})\, .
\end{align*}
For computing the scaling iterates~\eqref{eq-scaling-iterates}, all the information we need about $F_1$ and $F_2$ can be condensed by the specification of functions $\proxdiv_{F_1} : \RR_+^I\times \RR_+ \to \RR_+^I$ and $\proxdiv_{F_2} : \RR_+^J\times \RR_+ \to \RR_+^J$ defined as
\eql{
\label{eq_proxdiv_a}
\proxdiv_{F_i} : (\mathbf{s},\epsilon) \mapsto \prox^{\KL}_{F_i/\epsilon}( \mathbf{s})\oslash  \mathbf{s}\, .
}
Given this function, which is easy to compute in many cases (see e.g.~Table \ref{table_proximalexplicit}), the scaling algorithm is straightforward to implement.

%
%
\begin{algorithm}
\caption{Scaling algorithm}\label{algo_scaling_discrete}
\begin{algorithmic}[1]
\Function{ScalingAlgo}{$\proxdiv_{F_1},\proxdiv_{F_2},\mathbf{K},\mathbf{\d x}, \mathbf{\d y}, \epsilon$}
   \State $\mathbf{b} \gets \ones_J$
   \Repeat
	\State $\mathbf{a} \gets \proxdiv_{F_1} (\mathbf{K}(\mathbf{b} \odot \mathbf{\d y}), \epsilon)$
	\State $\mathbf{b} \gets \proxdiv_{F_2} (\transp{\mathbf{K}}(\mathbf{a} \odot \mathbf{\d x}),\epsilon)$
   \Until{stopping criterion}
   \State \textbf{return} $(\mathbf{a}_i \mathbf{K}_{ij} \mathbf{b}_j)_{ij}$\Comment{The primal optimizer}
\EndFunction
\end{algorithmic}
\end{algorithm}
By virtue of Theorem \ref{prop_convergencediscrete}, Algorithm \ref{algo_scaling_discrete} is guaranteed to stop and to return an approximate minimizer of~\eqref{eq_primal_discr} if one uses a consistant stopping criterion. 
\subsection{Log-domain Stabilization}
\label{sec:LogDomain}

For small values of $\epsilon$ the entries of $\mathbf{K} = e^{-\mathbf{C}/\epsilon}$, $\iter{\mathbf{a}}= e^{\iter{\mathbf{u}}/\epsilon}$ and $\iter{\mathbf{b}}= e^{\iter{\mathbf{v}}/\epsilon}$ may become both very small and very large, leading to numerically imprecise values of their mutual products and overflow of the numerical range.
On the other hand, executing the algorithm in the $\log$ domain (i.e.\ storing the logarithms of the entries of $\mathbf{K}$, $\iter{\mathbf{a}}$ and $\iter{\mathbf{b}}$) is not completely satisfying because then the computation of $\Gibbs \iter{\mathbf{b}}$ and $\Gibbs \iter{\mathbf{a}}$  is not a simple matrix/vector product and the algorithm is considerably slowed down. We suggest a middle way, using a redundant parametrization of the iterates as follows:
\begin{align}
	\label{eq_stabilize_decompose}
	\iter{\mathbf{a}} & = \iter{\tilde{\mathbf{a}}} \odot \exp(\iter{\tilde{\mathbf{u}}}/\epsilon), &
	\iter{\mathbf{b}} & = \iter{\tilde{\mathbf{b}}} \odot \exp(\iter{\tilde{\mathbf{v}}}/\epsilon).
\end{align}
The idea is to keep $\iter{\tilde{\mathbf{a}}}$ and $\iter{\tilde{\mathbf{b}}}$ close to $1$ and to absorb the extreme values of $\iter{\mathbf{a}}$ and $\iter{\mathbf{b}}$ into the log-domain via $\iter{\tilde{\mathbf{u}}}$ and $\iter{\tilde{\mathbf{v}}}$ from time to time.

Consider the stabilized kernels $\iter{\tilde{\mathbf{K}}}$ whose entries are given by 
\eq{
\iter{\tilde{\mathbf{K}}}_{ij} = \exp ((\iter{\tilde{\mathbf{u}}}_i + \iter{\tilde{\mathbf{v}}}_j-\mathbf{C}_{i,j})/\epsilon)
}
and remark that it holds, after direct computations,
\eq{
\Gibbs \iter{\mathbf{b}} = e^{-\frac{\iter{\tilde{\mathbf{u}}}}{\epsilon}} \odot \iter{\tilde{\mathbf{K}}}(\iter{\tilde{\mathbf{b}}} \odot \mathbf{\d y})
\qandq
\transp{\Gibbs} \iter{\mathbf{a}} = e^{-\frac{\iter{\tilde{\mathbf{v}}}}{\epsilon}} \odot \transp{(\iter{\tilde{\mathbf{K}}})}(\iter{\tilde{\mathbf{a}}} \odot \mathbf{\d x})\, .
}
The scaling iterates~\eqref{eq-scaling-iterates} then read 
\begin{align*}
	\iiter{\tilde{\mathbf{a}}} = 
			\prox^{\KL}_{\tfrac{1}{\epsilon}F_1}(e^{-\frac{\iter{\tilde{\mathbf{u}}}}{\epsilon}} \odot \mathbf{s}_1)  
		\oslash
			\mathbf{s}_1
		\, , \quad	
	\iiter{\tilde{\mathbf{b}}}= 
			\prox^{\KL}_{\tfrac{1}{\epsilon}F_2}(e^{-\frac{\iter{\tilde{\mathbf{v}}}}{\epsilon}} \odot \mathbf{s}_2)  
		\oslash
			\mathbf{s}_2	
\end{align*}
where $\mathbf{s}_1 = \iter{\tilde{\mathbf{K}}}(\iter{\tilde{\mathbf{b}}} \odot \mathbf{\d y})$ and $\mathbf{s}_2= \transp{(\iter{\tilde{\mathbf{K}}})}(\iter{\tilde{\mathbf{a}}} \odot \mathbf{\d x})$.

For computing these iterates, all the information we need about $F_1$ and $F_2$ can be condensed by the specification of ``stabilized $\proxdiv$'' functions $\proxdiv_{F_1} : \RR^I\times \RR^I \times \RR_+ \to \RR^I$ and $\proxdiv_{F_2} : \RR^J\times \RR^J \times \RR_+ \to \RR^J$ defined as
\eql{
\label{eq_proxdiv}
\proxdiv_{F_i} : (\mathbf{s},\mathbf{u},\epsilon) \mapsto \prox^{\KL}_{F_i/\epsilon}(e^{-\frac{\mathbf{u}}{\epsilon}} \odot \mathbf{s}) \oslash \mathbf{s}\, .
}
We adopt the same notation as in~\eqref{eq_proxdiv_a} because it is just the special case when $\mathbf{u}=0$.
The main numerical algorithm thus obtained is displayed in Algorithm \ref{algo_scaling_stabilized}.

\paragraph{Computing $\proxdiv$.}
The issue of extreme numerical values is not completely remedied by the absorption steps in Algorithm \ref{algo_scaling_stabilized} since the $\sproxdiv$ operation still involves the potentially extreme factor $e^{-\mathbf{u}/\epsilon}$. In practice however, we find that for many problems $\sproxdiv$ can be computed without evaluating the exponential $e^{-\mathbf{u}/\epsilon}$ and the formula remains numerically stable in the limit of small $\epsilon$. Several examples for this are given in Section \ref{sec_applications}.

\paragraph{Frequency of absorptions.}
In practice, we recommend to run stabilized iterations and check for the extreme values of $(\iter{\tilde{\mathbf{a}}},\iter{\tilde{\mathbf{b}}})$ every couple of iterations. When they exceed a given threshold, an absorption step is performed. 
%
\begin{algorithm}
\caption{Scaling algorithm with stabilization}\label{algo_scaling_stabilized}
\begin{algorithmic}[1]
\Function{ScalingAlgo2}{$\sproxdiv_{F_1},\sproxdiv_{F_2},\mathbf{C},\mathbf{\d x},\mathbf{ \d y}, \epsilon$}
   \State $(\tilde{\mathbf{b}},\mathbf{u},\mathbf{v}) \gets (\ones_J,0_I,0_J)$
   \State $\tilde{\mathbf{K}}_{ij} \gets \exp(-\mathbf{C}_{ij}/\epsilon)$  \Comment{for all $i,j$.}
   \Repeat
		\State $\tilde{\mathbf{a}} \gets \sproxdiv_{F_1}(\tilde{\mathbf{K}}(\tilde{\mathbf{b}} \odot \mathbf{\d y}),\mathbf{u},\epsilon)$
		\State  $\tilde{\mathbf{b}} \gets \sproxdiv_{F_2}(\transp{\tilde{\mathbf{K}}}(\tilde{\mathbf{a}} \odot \mathbf{\d x}),\mathbf{v},\epsilon)$
		\If{a component of $|\log \tilde{\mathbf{a}}|$ or $|\log \tilde{\mathbf{b}}|$ is ``too big'' }
	\State $(\mathbf{u},\mathbf{v}) \gets (\mathbf{u}+\epsilon \log \tilde{\mathbf{a}}, \mathbf{v} + \epsilon \log \tilde{\mathbf{b}})$\label{alg_line_uv}
	\State $\tilde{\mathbf{K}}_{ij} \gets \exp((\mathbf{u}_i + \mathbf{v}_j-\mathbf{C}_{i,j})/\epsilon)$ \Comment{for all $i,j$.}\label{alg_line_Kup}
	\State $\tilde{\mathbf{b}} \gets \ones_J$
	\EndIf
   \Until{stopping criterion}
   \State \textbf{return} $(\tilde{\mathbf{a}}_i \tilde{\mathbf{K}}_{i,j} \tilde{\mathbf{b}}_j)_{i,j}$\Comment{The primal optimizer}
\EndFunction
\end{algorithmic}
\end{algorithm}

\subsection{Comments on Implementation}
\label{sec_implementation}

We wish to emphasize the simplicity of Algorithm \ref{algo_scaling_discrete} and Algorithm \ref{algo_scaling_stabilized}. 
When an optimization problem of the form \eqref{eq-general} is given one just has to (i) choose the reference measures (which also determines a discretization grid in practice) (ii) determine the functions $F_i$ by going to the space of densities and (iii) find a way to efficiently compute $\proxdiv_{F_i}$ or $\sproxdiv_{F_i}$ (in many cases, this operator has a closed form, or can be computed with a few parallelizable iterations). Let us briefly discuss some details on the practical implementation of Algorithm \ref{algo_scaling_stabilized}.


\paragraph{Computing the matrix multiplications.}
The size of the matrix $\mathbf{K}$ is $I\times J$ so the matrix-vector multiplication step or even merely storing $\mathbf{K}$ in memory can quickly become intractable as the sizes of $X$ and $Y$ increase.
For special problems it is possible to avoid dense matrix multiplication (and storage). For instance, when $X$, $Y$ are Cartesian grids in $\RR^d$ and $c(x,y) = |x-y|^2$ is the squared Euclidean distance, then $\mathbf{K}$ is the separable Gaussian kernel: multiplying by $\mathbf{K}$ can be done by successive ``1-D convolutions''. We use this trick in Section \ref{sec_applications}, when mentioned. For more general geometric surfaces, $\mathbf{K}$ can also be approximated by the heat kernel~\cite{2015-solomon-siggraph}.
These methods however cannot be combined with Algorithm \ref{algo_scaling_stabilized}, as the ``stabilized kernels'' $\tilde{\mathbf{K}}$ lose the particular structure.
%
\paragraph{Gradually decreasing $\epsilon$.}
When solving for a very small $\epsilon$, most of the entries of $\mathbf{K}$ are below machine precision, so one need to first ``estimate'' the dual variables $(\mathbf{u},\mathbf{v})$ by performing several iterations with higher values of $\epsilon$. 
Reduction of $\epsilon$ should be performed between lines \ref{alg_line_uv} and \ref{alg_line_Kup} in Algorithm \ref{algo_scaling_stabilized}: after line  \ref{alg_line_uv}, $(\mathbf{u},\mathbf{v})$ are ``approximations'' of the dual variable of the unregularized problem so one can change $\epsilon$ and start solving for a different $\epsilon$ with $(\mathbf{u},\mathbf{v})$ as a starting point. We use this heuristic in Section \ref{sec_applications} (when mentioned) as follows: starting from $\epsilon=1$, after every $100$ iteration we perform an absorption step and divide $\epsilon$ by factor chosen so that the final value $\epsilon$ is reached after $10$ divisions. Then we run the standard Algorithm \ref{algo_scaling_stabilized} until the desired convergence criterion is met.

Note that gradually decreasing $\epsilon$ has also been proposed to asymptotically solve the unregularized OT problem~\cite{sharify2011solution} (but $\epsilon$ should be of the order $1/\log(\ell)$ which is too slow to be of practical interest) and, heuristically, to accelerate the convergence of Sinkhorn's iterations~\cite{YuilleInvisibleHand1994}, in the same fashion as for interior point methods, but theoretical justification is an open problem.

%

\paragraph{Multiple couplings.}
The general optimization problem in Section \ref{sec_algorithm_analysis} involved $n$ couplings, potentially $n>1$. For simplicity, throughout Section \ref{sec-algorithm} we focussed on the case $n=1$. However, the extension to $n>1$ is rather simple.
In particular, the variables $\mathbf{a},\mathbf{b},\mathbf{u},\mathbf{v}$ of the algorithm lie in $\RR^{n\times I}$ or $\RR^{n\times J}$, the kernel $\mathbf{K}$ is a $n$-family of $I\times J$ matrices, the entrywise operations (multiplication, division) are still performed entrywise and the matrix vector multiplications are performed ``coupling by coupling'', e.g.\ for $k\in \{1,\dots,n\}$ and $j\in \{1,\dots, J\}$:
\eq{
(\mathbf{K}\, \mathbf{b})_{k,i} = \sum_j \mathbf{K}_{k,i,j}\, \mathbf{b}_{k,j}\, .
}

\paragraph{More tricks in \cite{SchmitzerScaling2016}.} 
\label{sec:RelationSparseScaling}

An efficient numerical implementation of Algorithm \ref{algo_scaling_stabilized} is studied further in \cite{SchmitzerScaling2016}. In addition to the log-domain stabilization and gradually decreasing $\epsilon$,
it is proposed to approximate $\mathbf{K}$ by a sparse matrix, obtained by adaptive truncation.
Thus, one can avoid storing of and multiplication by the dense kernel matrix, while keeping the inflicted truncation error negligible. This is more flexible than for instance the Gaussian convolution trick and can easily be extended to more general cost functions. In addition, this can be directly combined with the log-domain stabilization and therefore allows to solve larger problems with small regularization parameter (and hence, with little entropic blur). We choose however not to use these additional tricks Section \ref{sec_applications} in order to display results which are easily reproducible.



\subsection{Generalization: more spaces and pushforward operators}
Scaling algorithms similar to Algorithm~\eqref{eq-scaling-iterates} can be formulated for solving problems of more general form than \eqref{eq-general-regul-KL}. There can be more than $2$ functionals, more than $2$ spaces involved and the projection operators $P^X_\#$ and $P^Y_\#$ can be replaced by more general linear operators, such as pushforwards of functions $t$ which are not necessarily projections (i.e. not of the form $t(x,y)= x$). 
Several examples of such extensions can be found in \cite{2015-benamou-cisc} for the special case of classical optimal transport. Let us sketch this extension in the discrete setting and for the case of $n=1$ (one ``coupling'') so as to remain simple and to stick close to implementation concerns. For brevity, we limit ourselves to giving the ``scaling'' form of the alternate maximization on the dual and an example, without proof.

Let $(X^k,\mathbf{\d x}^k)_{k=1}^N$ and $(Z,\mathbf{\d z})$ be finite measured spaces of respective cardinalities $(I_k)_{k=1}^N$ and $L$ and let $(t^k:Z\to X_k)_{k=1}^N$ be surjective maps. The space $Z$ plays the role of $X\times Y$ in the previous discussions, but in this generalization the structure of a product space is lost. For conciseness, in the notations, the maps $(t^k)_k$ act on indices of points instead of points, with an obvious meaning. Given $k\in \{1,\dots,N\}$, $\mathbf{R}\in \RR^L$ and $\mathbf{u}\in \RR^{I_k}$ the pushforward operator $t^k_\#$ and its adjoint $(t^k_\#)^*$ read
\eq{
t^k_\# \mathbf{R} = \big(\sum_{l\in (t^k)^{-1}(i)} \mathbf{R}_l\cdot \mathbf{\d z}_l\, /\, \mathbf{\d x}^k_i \big)_{i=1}^{I_k}
\qandq
(t^k_\#)^* \mathbf{u} = \big(\mathbf{u}_{t^k(l)}\big)_{l=1}^L\, . 
}

Given a nonnegative vector $\mathbf{K}\in \RR^L$ and $N$ convex, proper, lower semicontinuous functions $F_k:\RR^{I_k}\to \RR\cup\{\infty\}$, the generalization of \eqref{eq_primal_discr} is (up to a constant)
\eq{\label{eq_pushforward_primal}
\min_{\mathbf{R}\in \RR^L} \sum_{k=1}^N F_k(t^k_\# \mathbf{R}) +\epsilon \sum_{l=1}^L \mathbf{R}_l \cdot \big(\log(\mathbf{R}_l \, /\, \mathbf{K}_l)-1\big)\cdot \mathbf{\d z}_l 
}
with the convention $0\log (0/0) = 0$, and the dual reads
\eql{\label{eq_pushforward_dual}
\sup_{\substack{(\mathbf{u}^k)_{k=1}^N \in \RR^{\sum_k I_k}}} -\sum_{k=1}^N F_k^*(-\mathbf{u}^k) -\epsilon \sum_{l=1}^L \exp(\tfrac1\epsilon\sum_{k=1}^N \mathbf{u}^k_{t^k(l)})\cdot \mathbf{K}_l\cdot \mathbf{\d z}_l .
}

In order to perform alternate maximization on the dual as before, one needs a ``disintegration'' relation, in the spirit of \eqref{eq_fubinigibbs}. To this end, we define, for $(\mathbf{a}^n)_{n=1}^N \in \RR_+^{\sum_n I_n}$  and $k\in \{1, \dots, N\}$ the operator $\Gibbs^k$ as
\eq{
\Gibbs^k ((\mathbf{a}^n)_{n\neq k}) \eqdef
\big(\sum_{l\in (t^k)^{-1}(i)} (\prod_{n\neq k}\mathbf{a}^n_{t^n(l)} )  \cdot  \mathbf{K}_l \cdot \mathbf{\d z}_l/\mathbf{\d x}_i ^k \big)_{i=1}^{I_k}
}
With those operators, the rightmost term of \eqref{eq_pushforward_dual} can be computed in a ``marginalized'' way, using the relation
\eq{
\langle \mathbf{a}^k\, , \, \Gibbs^k ((\mathbf{a}^n)_{n\neq k}) \rangle_{\mathbf{\d x}^k}
=
\sum_{l=1}^L (\prod_{n=1}^N \mathbf{a}^n_{t^n(l)} ) \cdot \mathbf{K}_l \cdot \mathbf{\d z}_l
}
valid for $k\in\{1,\dots,N\}$. The key feature for obtaining this relation is the fact that $((t^k)^{-1}(i))_{i=1}^{I_k}$ forms a partition of $Z$, and this explains why the scaling algorithm generalizes naturally to linear operators which are ``pushforward''. It is now simple, at least formally, to define the generalization of the scaling algorithm dispayed in Algorithm \ref{algo_scaling_generalized}, by writing the alternate optimization on the dual problem, and taking again the dual, in the spirit of Proposition \ref{prop_alternating}.
\begin{algorithm}
\caption{Generalized scaling algorithm}\label{algo_scaling_generalized}
\begin{algorithmic}[1]
\Function{GeneralScalingAlgo}{$(\proxdiv_{F_k},\Gibbs^k, t^k)_{k=1}^n, \mathbf{K},\epsilon$}
   \State $\mathbf{a}^k \gets \ones_{I_k}$ \Comment{for all $k=1,\dots,N$.}
   \Repeat
   \For{k=1,\dots,N}
		\State $\mathbf{a}^k \gets \proxdiv(\Gibbs^k((\mathbf{a}^n)_{n\neq k}), \epsilon)$
   \EndFor
   \Until{stopping criteron}
   \State \textbf{return} $\big( \mathbf{K}_l  \cdot \prod_{k=1}^N \mathbf{a}^k_{t^k(l)}\big)_{l=1}^L$\Comment{The primal optimizer}
\EndFunction
\end{algorithmic}
\end{algorithm}

As a simple illustration, consider, in the setting of equation \eqref{eq_primal_discr}, an extension where is added a function of the total mass
\eq{
\min_{\mathbf{R}\in\RR^{I\times J}} F_1(P^X_\# \mathbf{R}) + F_2(P^Y_\# \mathbf{R}) + F_3(m_\# \mathbf{R}) + \epsilon \KL(\mathbf{R}|\mathbf{K})\, .
}
where $m:(x,y) \to \{0\}$ ($m_\#$ returns the total mass) and $F_3:\RR\to \RR\cup \{\infty\}$ is a proper, lower semicontinuous and convex function.
Applying the reasoning above, and after some rearrangement (here $Z$ is still a product space and thus it is convenient to store $\mathbf{R}$ as a matrix), one obtains Algorithm \ref{algo_scaling_mass}. If $F_{3}$ is the indicator of equality with a positive real number, this algorithm solves the partial optimal transport problem \cite{caffarelli2010free,figalli2010optimal} but more general $F_{3}$ can be considered such as, for instance, a range constraint on the total mass.

\begin{algorithm}
\caption{Scaling algorithm with a function on the total mass}\label{algo_scaling_mass}
\begin{algorithmic}[1]
\Function{MassScalingAlgo}{$(\proxdiv_{F_k})_{k=1}^3,\mathbf{K},\mathbf{\d x}, \mathbf{\d y}, \epsilon$}\Comment{$\mathbf{K}$ is a matrix.}
   \State $\mathbf{b} \gets \ones_J$
   \State $z \gets 1$
   \Repeat
		\State $\mathbf{a} \gets \proxdiv_{F_1}(z\cdot \mathbf{K}(\mathbf{b} \odot \mathbf{\d y}))$
		\State  $\mathbf{b} \gets \proxdiv_{F_2}(z\cdot \transp{\mathbf{K}}(\mathbf{a} \odot \mathbf{\d x}))$
		\State  $z \gets \proxdiv_{F_3}(\transp{(\mathbf{a} \odot \mathbf{\d x})}\mathbf{K}(\mathbf{b} \odot \mathbf{\d y}))$
   \Until{convergence}
   \State \textbf{return} $z\cdot(\mathbf{a}_i \mathbf{K}_{i,j} \mathbf{b}_j)_{i,j}$\Comment{The primal optimizer}
\EndFunction
\end{algorithmic}
\end{algorithm}

\section{Applications}
\label{sec_applications}

Throughout this Section, we detail how to use the scaling iterations~\eqref{eq-scaling-iterates} for solving the problems discussed in Section \ref{sec-ot-like}. We first analyze the properties of the functionals on marginals $F_i$, then we derive the iterations in a continuous setting, and finally show numerical experiments. We extend the definition of the operator $\sproxdiv$ (defined in \eqref{eq_proxdiv} in the discrete setting) to the continuous setting as follows: for $s\in \Lun(X)$, $u \in \Linf(X)$ and $\epsilon>0$
\eql{
\label{eq_proxdivbis}
\sproxdiv_{F}(s,u,\epsilon) = \prox^{\KL}_{F/\epsilon}(s\, e^{-\frac{u}{\epsilon}} ) /s
}
with the convention $0/0=0$. Moreover, in order to avoid the heavy notation $\Divergm_\phi(a\d x|b \d x)$, we denote by straight letters the divergences between functions, i.e. for $a,b$ measurable nonnegative functions on $X$ and $\varphi$ a nonnegative entropy function (see Definition \ref{def_entropy}):
\eql{\label{eq_divergencefunctions}
 \Diverg_\phi(a|b) \eqdef \int_X \ol{\Diverg}_\phi(a(x)|b(x)) \d x
 \qwithq
 \ol{\Diverg}_\phi(a|b) = 
 \begin{cases}
 b\cdot \phi(a/b) & \text{if $b>0$}\\
 a\cdot \varphi'_\infty & \text{otherwise}\, .
 \end{cases}
}
with the convention $0\times \infty=0$. Some properties of these divergences between functions are studied in Appendix \ref{sec:ApxDivergences}.

All reported runtimes were obtained with an implementation in Julia, on a standard laptop with CPU clock rate $2.5$ GHz. 

\subsection{Balanced and Unbalanced Optimal Transport}
\label{sec_appli_UOT}

\paragraph{Derivation of the algorithm.}

The basic framework of classical and unbalanced optimal transport has been recalled in Sections \ref{subsec_balanced} and \ref{subsec_unbalanced}. Assume that we are given marginals $\mu\in \Mm_+(X)$, $\nu \in \Mm_+(\nu)$ and a cost function $c:X\times Y \to \RR \cup \{\infty\}$. By defining $p = \d \mu /\d x$ and $q=\d \nu /\d y$, the marginal functionals involved in the regularized problem are
\eql{\label{eq_UOT}
F_1(s_1) = \Diverg_{\phi_1}(s_1 | p )
\qandq
F_2(s_2) = \Diverg_{\phi_2}(s_2  | q )\, .
}
as in \eqref{eq_divergencefunctions} above. As shown in Appendix \ref{sec:ApxDivergences}, if $\phi_1$ and $\phi_2$ are a nonnegative entropy function (Definition \ref{def_entropy}) then $F_1$ and $F_2$ are admissible integral functionals (Definition \ref{def_integralfunctional}). In order to compute the associated $\proxdiv$ operator, let us apply Proposition \ref{prop_proxKLcont} in this precise case.
\begin{proposition}
\label{prop_UOTprox}
Let $\phi$ be a nonnegative entropy function and $(s,p)\in \Lun_+(X)^2$ such that $0\in \dom \phi$ or $s(x)=0 \Rightarrow p(x)=0$ a.e. Let $F(s)=\Diverg_{\phi}(s | p )$. Then $\prox^{\KL}_{F/\epsilon}(s)$ is not empty and is the singleton $s^\star$ satisfying for a.e. $x\in X$,
\eq{
\begin{cases}
0 = s^\star(x) & \text{if $s(x)=0$,}\\
0 = \epsilon \log(s^\star(x)/s(x)) + \phi'_\infty & \text{if $p(x)=0$ and $s(x)>0$,}\\
0 \in \epsilon \log(s^\star(x)/s(x)) + \partial \phi(s^\star(x)/p(x)) & \text{otherwise.}
\end{cases}
}
\end{proposition}
\begin{proof}
It is the pointwise optimality conditions associated to Proposition \ref{prop_proxKLcont}.
\end{proof}
This formula allows to compute explicitly the $\proxdiv$ operators of the examples introduced in Section \ref{sec_divergencefunc}, as listed in Table \ref{table_proximalexplicit}. These entropy functions as well as the associated $\proxdiv$ operators are displayed on Figure \ref{fig_fdiv}. Note that, in Table~\ref{table_proximalexplicit} the first line corresponds to standard Sinkhorn iterations and these iterations are recovered in the second and third line by letting $\la\to + \infty$ and by setting $\alpha=\beta=1$ in the fourth line.
In the context of the log-domain stabilization (Section \ref{sec:LogDomain}), all four $\proxdiv$ operators remain stable in the limit of small $\epsilon$: either $\proxdiv$ is independent of $u$, only a regularized exponential $e^{-u/(\lambda + \epsilon)}$ must be evaluated, or extreme values are cut off by thresholding.

\begin{table}[h] \centering
\begin{tabular}{@{}rcrcr@{}}\toprule
$F$
&& $\proxKL_{F/\epsilon}(s)$ &&$ \proxdiv_F(s,u,\epsilon) $\\ \midrule
$\iota_{\{=\}}(\cdot | p) $
&& $p$ 
&& $p/s$ 
\\ \addlinespace
$\lambda \KL(\cdot | p)  $
&& $s^{\frac{\epsilon}{\epsilon+\lambda}} \cdot p^{\frac{\lambda}{\epsilon+\lambda}}$
&& $\left( p/s \right)^{\frac{\lambda}{\lambda+\epsilon}} \cdot e^{-u/(\lambda+\epsilon)}$ \\ \addlinespace
$\lambda \TV(\cdot | p) $
&& $\min \left\{ s \cdot e^{\frac{\lambda}{\epsilon}}, \max \left\{ s \cdot e^{-\frac{\lambda}{\epsilon}}, p \right\} \right\}$
&& $\min \left\{ e^{\frac{\lambda-u}{\epsilon}}, \max \left\{ e^{-\frac{\lambda+u}{\epsilon}}, p/s \right\} \right\}$  
\\ \addlinespace
$\RG_{[\alpha,\beta]} $
&& $ \min \left\{ \beta \, p, \max \left\{ \alpha\, p, s \right\} \right\} $
&& $\min \left\{ {\beta \, p}/{s}, \max \left\{ {\alpha\, p}/{s}, e^{-u/\epsilon} \right\} \right\}$
\\ \addlinespace \bottomrule
\end{tabular}
\caption{Some divergence functionals and the associated $\prox$ and $\proxdiv$ operators for functions $s$ and $u$ defined on $X$. All operators are acting pointwise and $\lambda>0$, $0\leq\alpha\leq\beta$ are real parameters (see Section \ref{sec_divergencefunc} for the definitions).}
\label{table_proximalexplicit}
\end{table}
\begin{figure}
\centering
\begin{subfigure}{0.45\linewidth} 
\centering
 \resizebox{1.\linewidth}{!}{
\includegraphics[clip,trim=0cm 0cm 0cm 0cm]{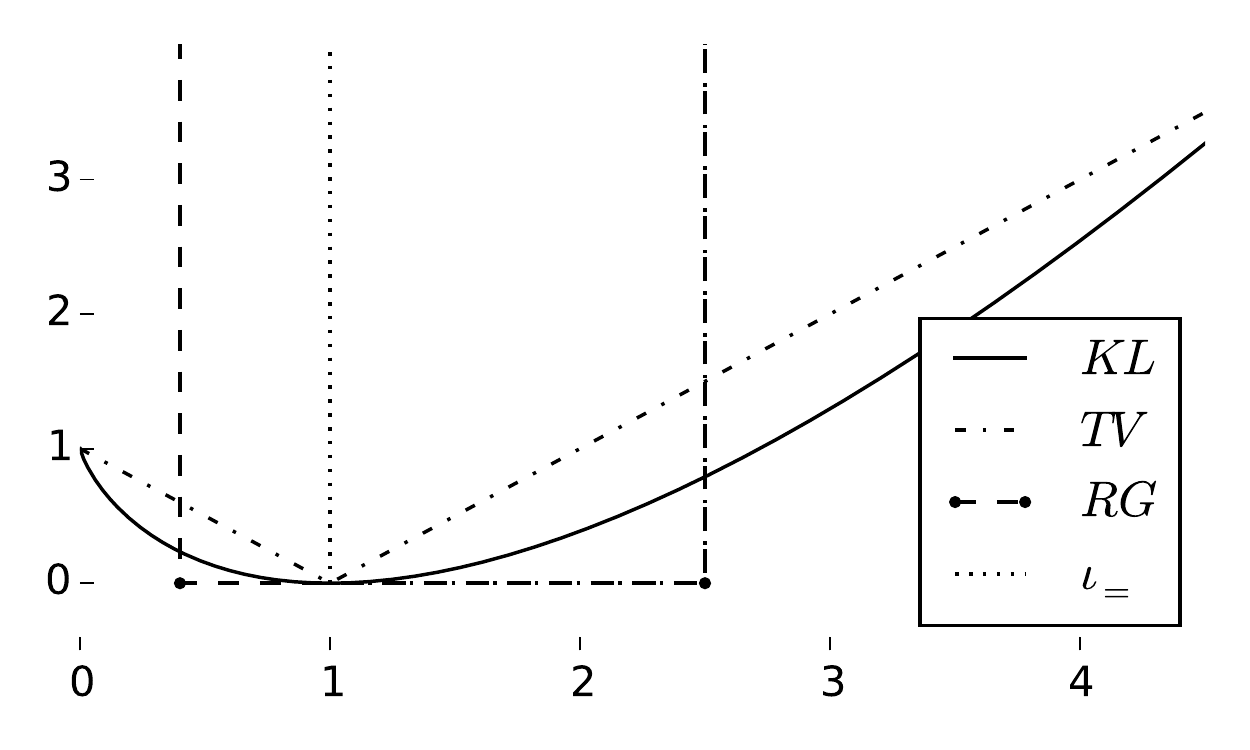}
}
\caption{Graph of the entropy functions $\varphi$ used to define the divergences in the legend.}
\end{subfigure}%
 \centering
 \qquad
\begin{subfigure}{0.45\linewidth} 
\centering
 \resizebox{1.\linewidth}{!}{
\includegraphics[clip,trim=0cm 0cm 0cm 0cm]{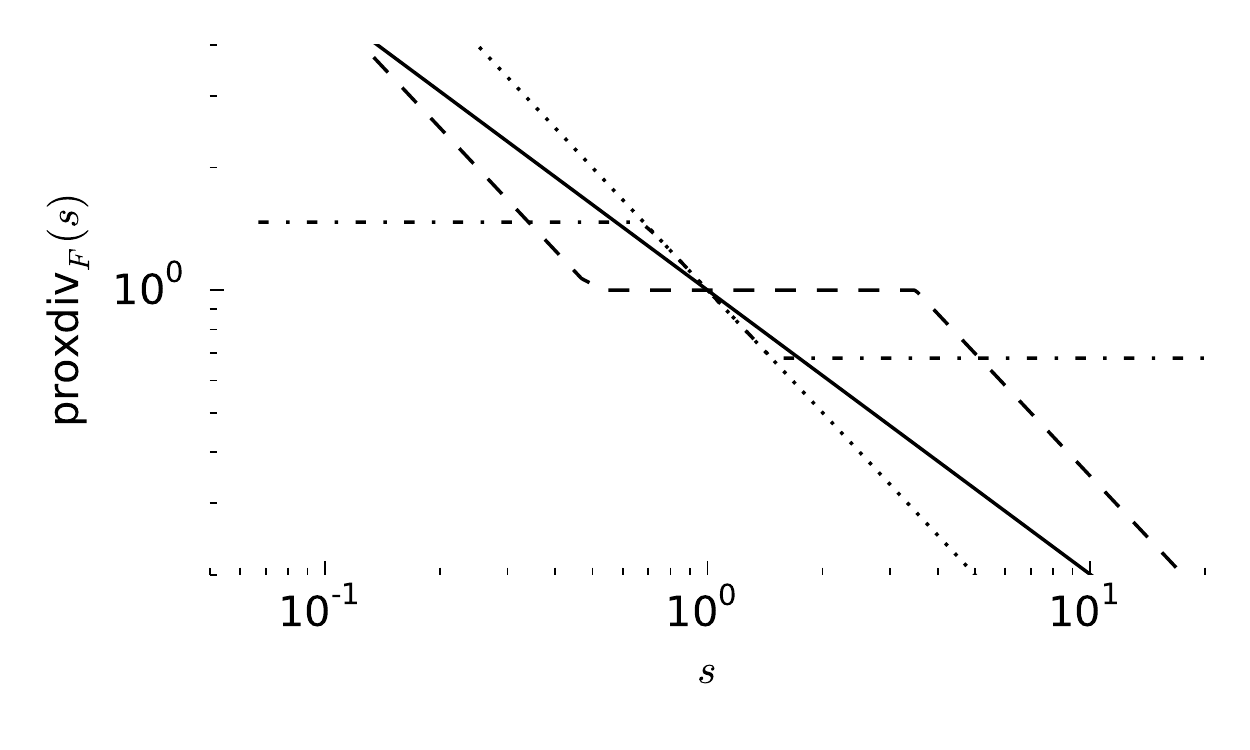}
}
\caption{Operator $\proxdiv_F$ associated to these divergence functions, for $p=1$, $u=0$ and varying $s\in \RR$ (in $\log$ scale)}
\end{subfigure}%
\caption{Divergences and $\proxdiv$ operators for the examples of Table \ref{table_proximalexplicit}}
\label{fig_fdiv}
\end{figure}


%


\paragraph{Numerical examples for $X=[0,1]$.}

Let $X=Y$ be the discretization of the interval $[0,1]$ into $I=J=1000$ uniform samples and $\mathbf{\d x} = \mathbf{\d y} = \frac{1}{I} \ones_I$ (the discretized Lebesgue measure). Let $\mathbf{p}$, $\mathbf{q}$ be the (discrete) marginals displayed on Figures \ref{fig_marg1d_p}-\ref{fig_marg1d_q}. We solve the discrete entropic regularized problem \eqref{eq_primal_discr} using the stabilized scaling Algorithm \ref{algo_scaling_stabilized} for marginal functions $F_i$ of the form \eqref{eq_UOT} with several choices of divergences listed in Figure \ref{fig_1Dmarginals}. The algorithm was stopped after $10^3$ iterations with a runtime of approximately $30$ seconds. 

For the optimal solution $\mathbf{R}\in \RR^{I\times J}$, obtained after convergence, Figure \ref{fig_1Dmarginals} displays its projections $\mathbf{R}\, \mathbf{\d y}$ and $\transp{\mathbf{R}}\, \mathbf{\d x}$ on the domain $X$ and Figure \ref{fig_1Dplans} illustrates the entries of $\mathbf{R}$ which are greater than $10^{-10}$, which can be interpreted as the approximate support of the optimizers. Thanks to the stabilization of Algorithm \ref{algo_scaling_stabilized}, the parameter $\epsilon$ could be made extremely small, until the plan is quasi-deterministic. Here we set it to $\epsilon=10^{-7}$ for the support to be clearly visible on Figure \ref{fig_1Dplans}. Remark that for the $\TV$ case (orange support), the straigth segments correspond to a density of $1$ on the diagonal: this is the plan with minimal entropy among the (non-unique) optimal plans, in accordance with Proposition \ref{prop_converg_regul}.

Figure \ref{fig_convergenceUOT} displays the primal-dual gap $\eqref{eq-general-regul-KL}-\eqref{eq-dual-pbm}$ as a function of the iteration number $\ell$ when running Algorithm \ref{algo_scaling_discrete} for solving unbalanced optimal transport problems with $\epsilon=0.01$. More precisely, we display
\[
P_\epsilon(\iter r) - D_\epsilon(\iter u, \iter v)
\]
where $P_\epsilon$ and $D_\epsilon$ are the primal and dual functionals (where set constraints are replaced by an exponential function of the distance to the set) as in Theorem \ref{prop_duality} and $\iter r, \iter u, \iter v$ are the primal and dual iterates as in Theorem \ref{prop_convergencediscrete}.
The marginals are the same as shown on Figures \ref{fig_marg1d_p}-\ref{fig_marg1d_q} and the cost is quadratic. The discretization is $I=J=500$ except for the thin black line where $I=J=1000$. This plot leads to 3 observations which generalize the conclusion of Theorem \ref{th_convergenceKL}: (i) convergence in function values is empirically linear (ii) convergence is faster when the divergences are multiplied by a small weight (iii) convergence speed is insensitive to dimension of the problem. 
%
%
\begin{figure}
 \centering
\begin{subfigure}{.5\linewidth}
\centering
 \resizebox{1.\linewidth}{!}{
\includegraphics[clip,trim=0cm 0cm 0cm 0cm]{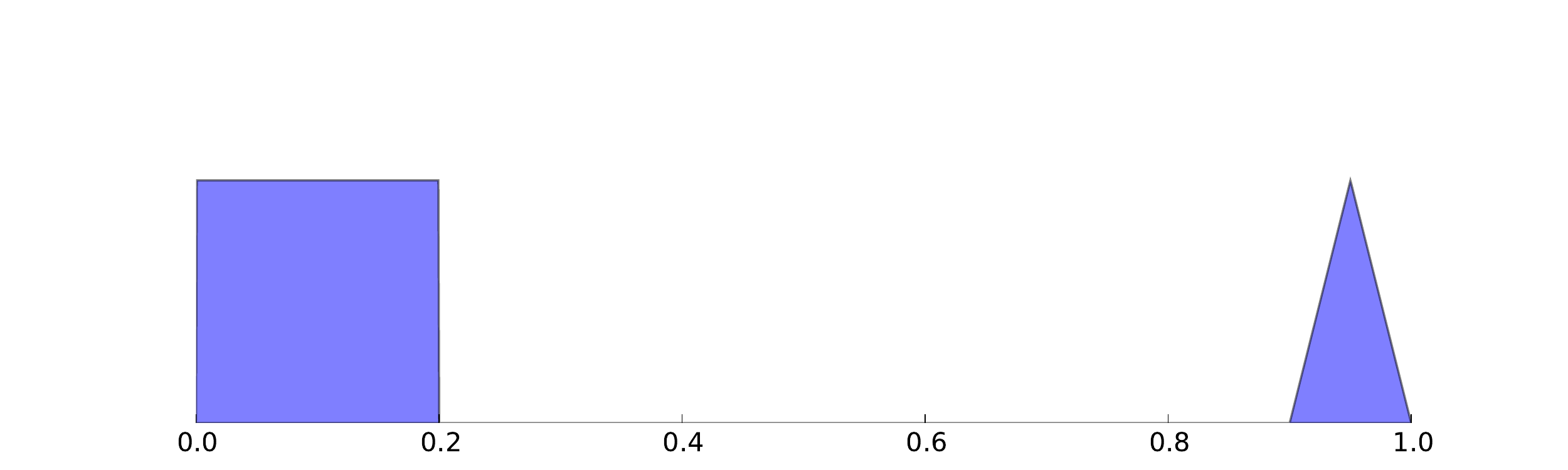}
}%
\caption{Marginal $p$}\label{fig_marg1d_p}
\end{subfigure}%
 \begin{subfigure}{.5\linewidth}
\centering
 \resizebox{1.\linewidth}{!}{
\includegraphics[clip,trim=0cm 0cm 0cm 0cm]{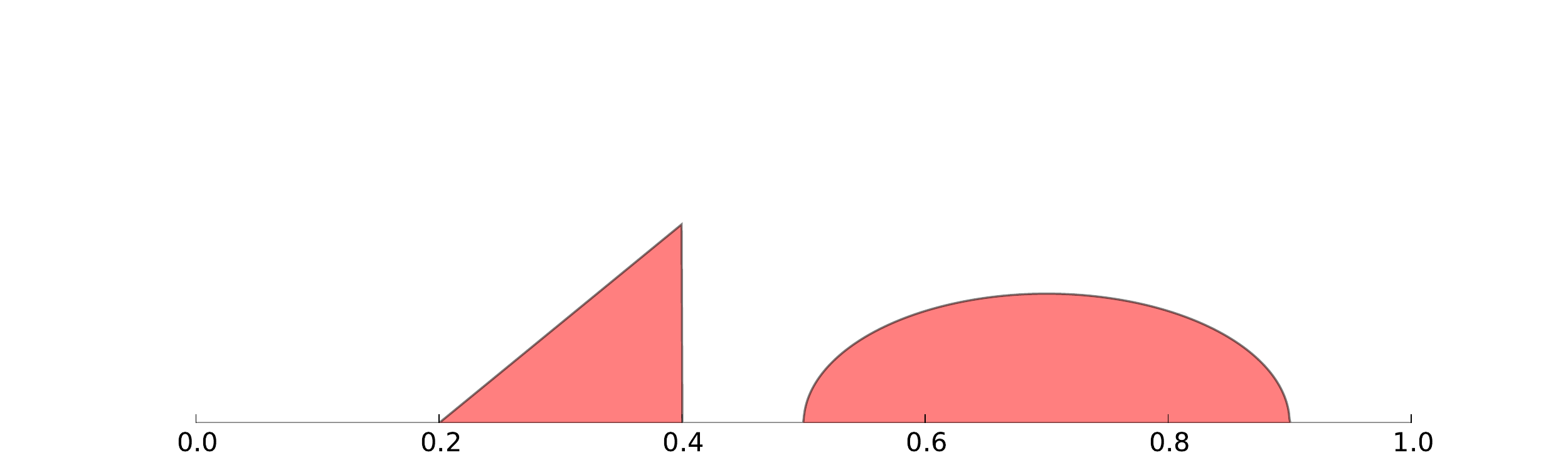}
}%
\caption{Marginal $q$}\label{fig_marg1d_q}
\end{subfigure}
\begin{subfigure}{0.5\linewidth} 
\centering
 \resizebox{1.\linewidth}{!}{
\includegraphics[clip,trim=0cm 0cm 0cm 0cm]{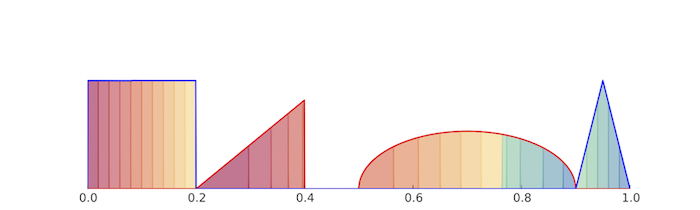}
}
\caption{$\iota_{\{=\}}$ (classical OT)}
\end{subfigure}%
\begin{subfigure}{0.5\linewidth} 
\centering
 \resizebox{1.\linewidth}{!}{
\includegraphics[clip,trim=0cm 0cm 0cm 0cm]{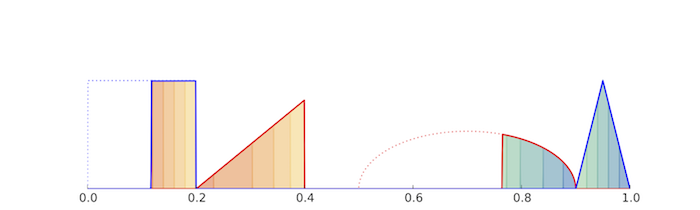}
}
\caption{$0.05\times \TV$}
\end{subfigure}
\begin{subfigure}{0.5\linewidth} 
\centering
 \resizebox{1.\linewidth}{!}{
\includegraphics[clip,trim=0cm 0cm 0cm 0cm]{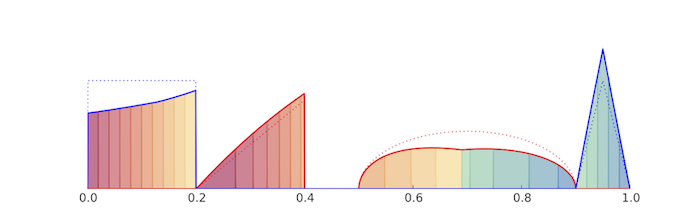}
}%
\caption{$0.5\times \KL$}
\end{subfigure}%
\begin{subfigure}{0.5\linewidth} 
\centering
 \resizebox{1.\linewidth}{!}{
\includegraphics[clip,trim=0cm 0cm 0cm 0cm]{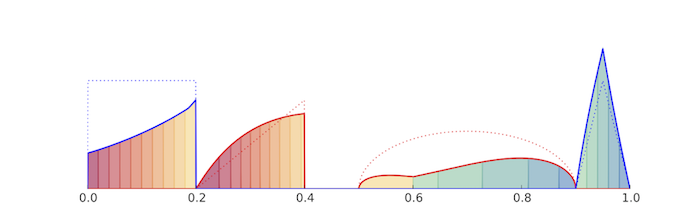}
}
\caption{$0.1\times \KL$}
\end{subfigure}
\begin{subfigure}{0.5\linewidth} 
\centering
 \resizebox{1.\linewidth}{!}{
\includegraphics[clip,trim=0cm 0cm 0cm 0cm]{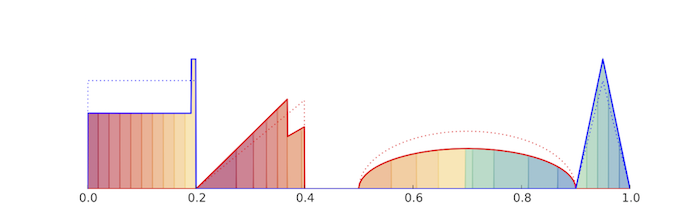}
}%
\caption{$\RG_{[0.7,\, 1.2]}$}
\end{subfigure}%
\begin{subfigure}{0.5\linewidth} 
\centering
 \resizebox{1.\linewidth}{!}{
\includegraphics[clip,trim=0cm 0cm 0cm 0cm]{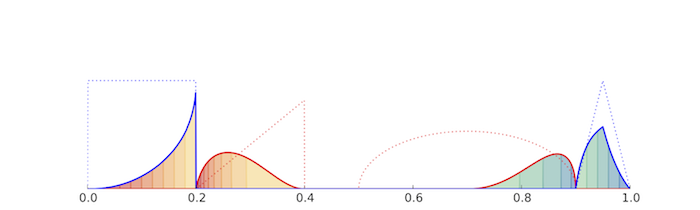}
}%
\caption{$\WassF$ (with cut locus at $0.2$)}
\end{subfigure}%
\caption{(A)-(B) Input marginals. (C)-(H) Marginals of the optimal plan $\mathbf{R}$ displayed together for several divergences (specified in the caption). The color shows the location of the same subset of mass before and after transportation. The cost is the quadratic cost $c(x,y)=|y-x|^2$ except for (H) which is a computation of the optimal plan for $\WF$ that is $F_1=F_2=\KL$ and $c$ as in \eqref{logcost} (with a spatial rescaling so that $c(x,y)=\infty \Leftrightarrow |y-x|\geq0.2$).}
\label{fig_1Dmarginals}
\end{figure}
\begin{figure}
 \centering
 \resizebox{.40\linewidth}{!}{
\includegraphics[clip,trim=0cm 0cm 0cm 0cm]{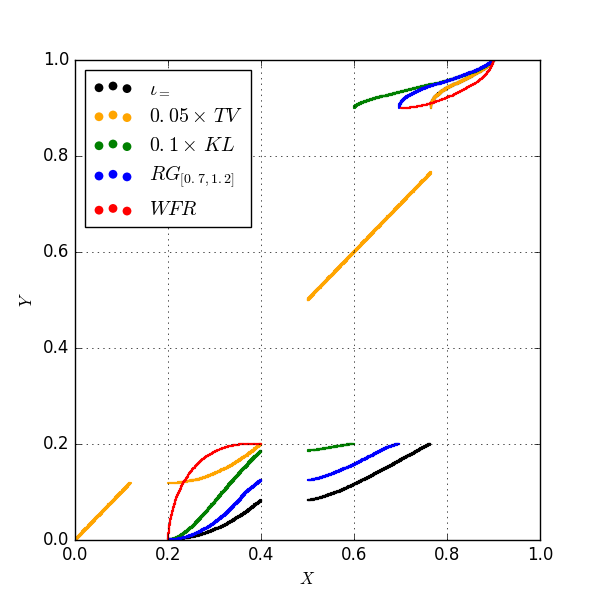}
}%
\caption{Support of the optimal plans $\enscond{(x_i,y_j)}{\mathbf{R}_{i,j}>10^{-10}}$ for all the examples displayed on Figure \ref{fig_1Dmarginals}, except (E). Orange and black are superimposed at the top.}
\label{fig_1Dplans}
\end{figure}
\begin{figure}
\centering
 \resizebox{.7\linewidth}{!}{
\includegraphics[clip,trim=0cm 0cm 0cm 0cm]{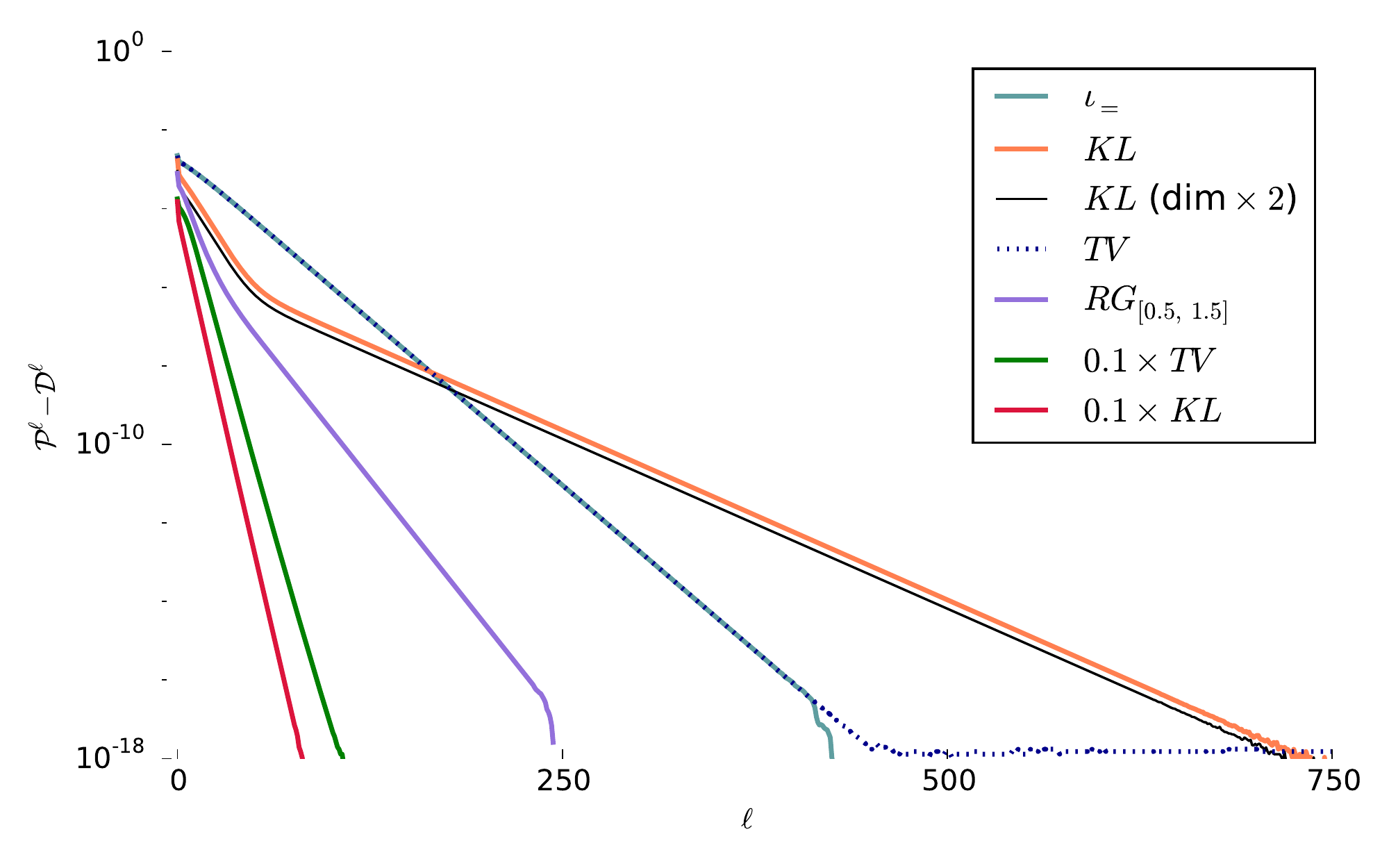}
}
\caption{Primal dual gap as a function of the iterations for Algorithm \ref{algo_scaling_discrete} applied to unbalanced optimal transport problems.}
\label{fig_convergenceUOT}
\end{figure}
\paragraph{Numerical examples for $X=Y=[0,1]^2$.}

Let $X=Y$ be the discretization of the $2$-dimensional domain $[0,1]^2$ into $I=J=200\times 200$ uniform samples and $\mathbf{\d x} = \mathbf{\d y} $ be the discrete Lebesgue measures. Let $p$, $q$ be the densities displayed together on Figure \ref{fig_marg2d} and $c$ be the quadratic cost $c(x,y)=|y-x|^2$. We solve the discrete entropic regularized problem \eqref{eq_primal_discr} using Algorithm \ref{algo_scaling_discrete} for several unbalanced optimal transport problems. 
We apply the ``separable kernel'' method (see Section \ref{sec_implementation}) to accelerate the algorithm. Since this cannot be combined with the log-stabilization, the regularization parameter $\epsilon$ has been fixed to the rather high value $\epsilon=10^{-4}$ in order to avoid numerical issues.
Figure \ref{fig_2Dmarginals} shows the marginals of the optimal coupling $\mathbf{R}$ and Figure \ref{fig_2Dplans} illustrates the resulting transport plan: points with the same color correspond roughly to the same mass particle before and after transport. 
The algorithm was stopped after $10^3$ iterations and the running time was of $35$ seconds approximately.
%

\begin{figure}
 \centering
\begin{subfigure}{0.22\linewidth} 
\centering
 \resizebox{1.1\linewidth}{!}{
\includegraphics[clip]{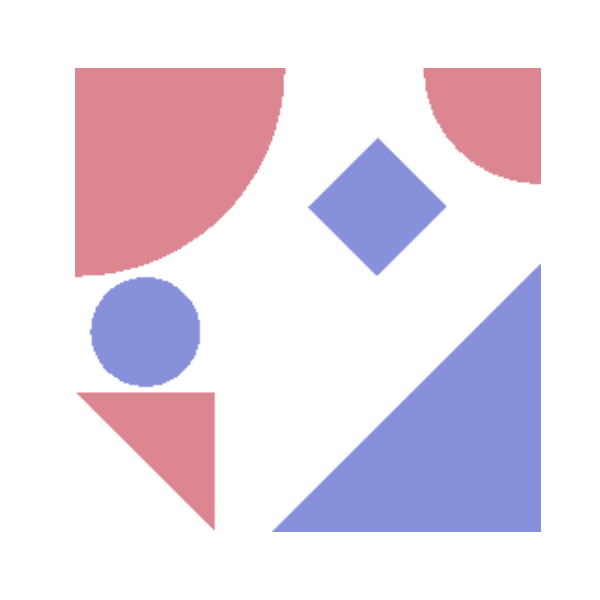}
}
\caption{$p$ and $q$}\label{fig_marg2d}
\end{subfigure}%
\begin{subfigure}{0.22\linewidth} 
\centering
 \resizebox{1.1\linewidth}{!}{
\includegraphics[clip,trim=0cm 0cm 0cm 0cm]{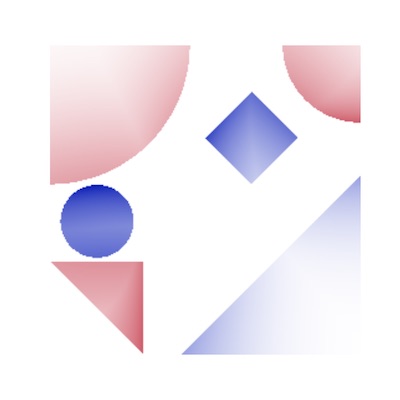}
}%
\caption{\small{$ 0.1\! \times\! \KL$}}
\end{subfigure}
\begin{subfigure}{0.22\linewidth} 
\centering
 \resizebox{1.1\linewidth}{!}{
\includegraphics[clip,trim=0cm 0cm 0cm 0cm]{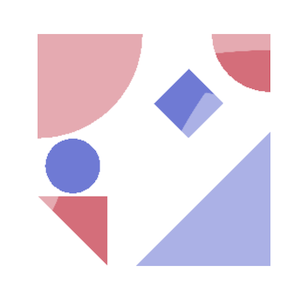}
}%
\caption{\small{$\RG_{[0.7,\, 1.2]}$}}
\end{subfigure}%
\begin{subfigure}{0.22\linewidth} 
\centering
 \resizebox{1.1\linewidth}{!}{
\includegraphics[trim=0cm 0cm 0cm 0cm]{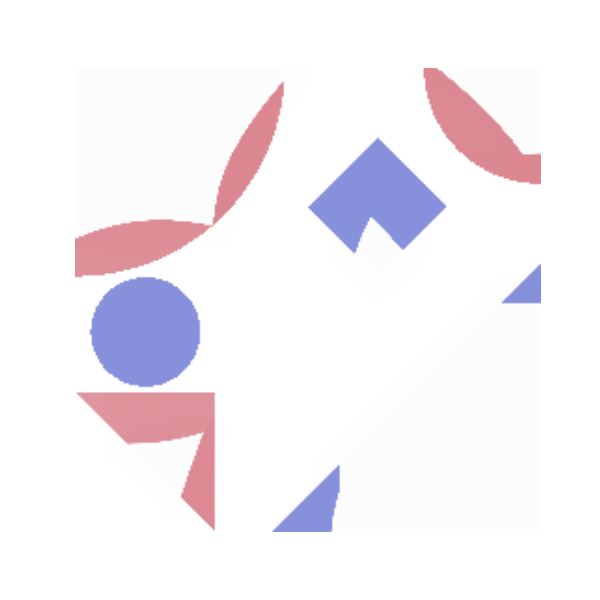}
}
\caption{\small{$0.05\times \TV$ }}
\end{subfigure}
\begin{subfigure}{0.1\linewidth} 
\centering
 \resizebox{2\linewidth}{!}{
\includegraphics[clip,trim=0cm 0cm 0cm 0cm]{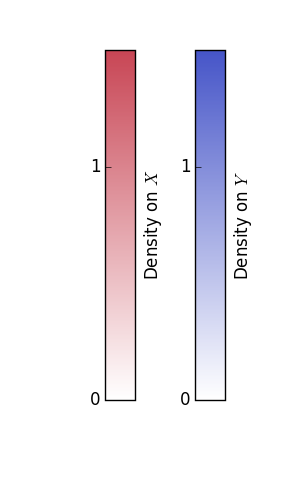}
}
\end{subfigure}
\caption{Marginals of the optimal plan $\mathbf{R}$ for several unbalanced optimal transport problems (quadratic cost, divergence specified in the captions).}
\label{fig_2Dmarginals}
\end{figure}
\begin{figure}
 \centering
\begin{subfigure}{0.24\linewidth} 
\centering
 \reflectbox{\rotatebox[origin=c]{180}{\resizebox{1.\linewidth}{!}{
\includegraphics[clip,trim=3cm 1cm 3cm 1cm]{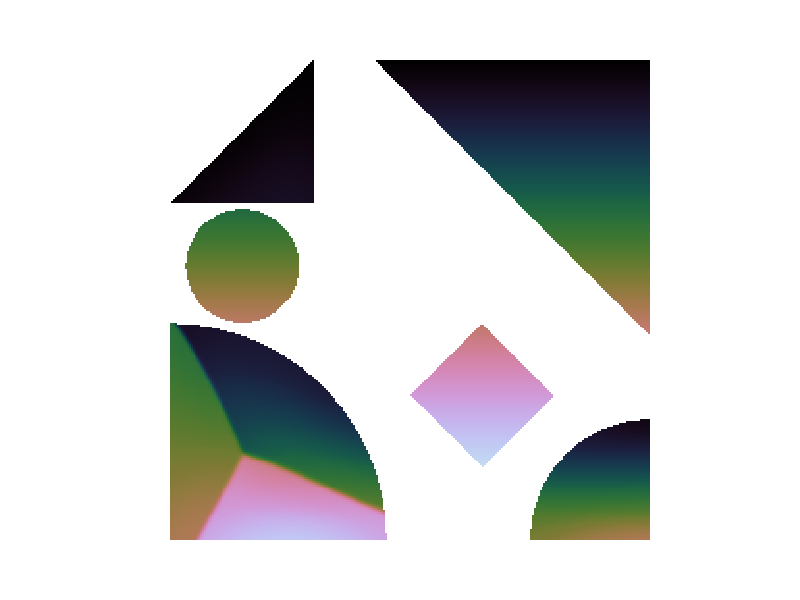}
}}}
\caption{$\iota_{\{=\}}$}
\end{subfigure}%
\begin{subfigure}{0.24\linewidth} 
\centering
 \reflectbox{\rotatebox[origin=c]{180}{\resizebox{1.\linewidth}{!}{
\includegraphics[clip,trim=3cm 1cm 3cm 1cm]{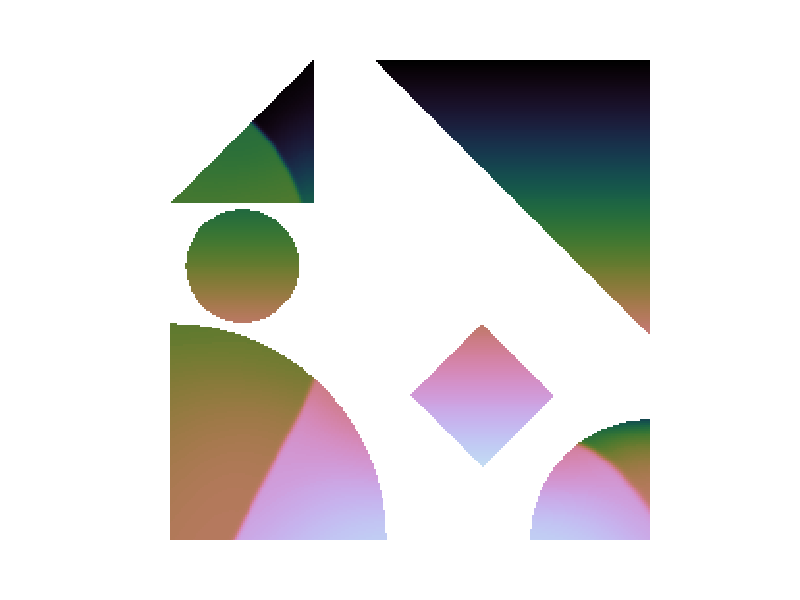}
}}}%
\caption{$0.1\times \KL$}
\end{subfigure}
\begin{subfigure}{0.24\linewidth} 
\centering
 \reflectbox{\rotatebox[origin=c]{180}{\resizebox{1.\linewidth}{!}{
\includegraphics[clip,trim=3cm 1cm 3cm 1cm]{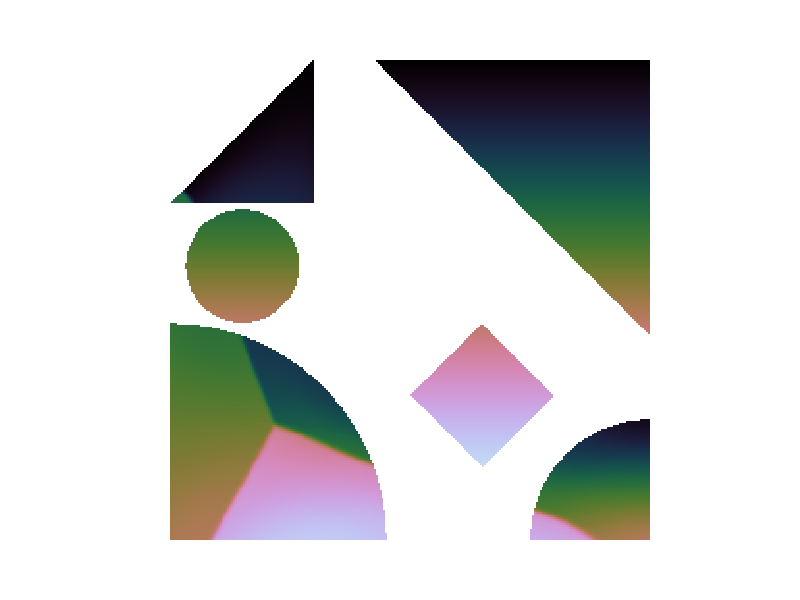}
}}}%
\caption{$\RG_{[0.7,\, 1.2]}$}
\end{subfigure}%
\begin{subfigure}{0.24\linewidth} 
\centering
 \reflectbox{\rotatebox[origin=c]{180}{\resizebox{1.\linewidth}{!}{
\includegraphics[clip,trim=3cm 1cm 3cm 1cm]{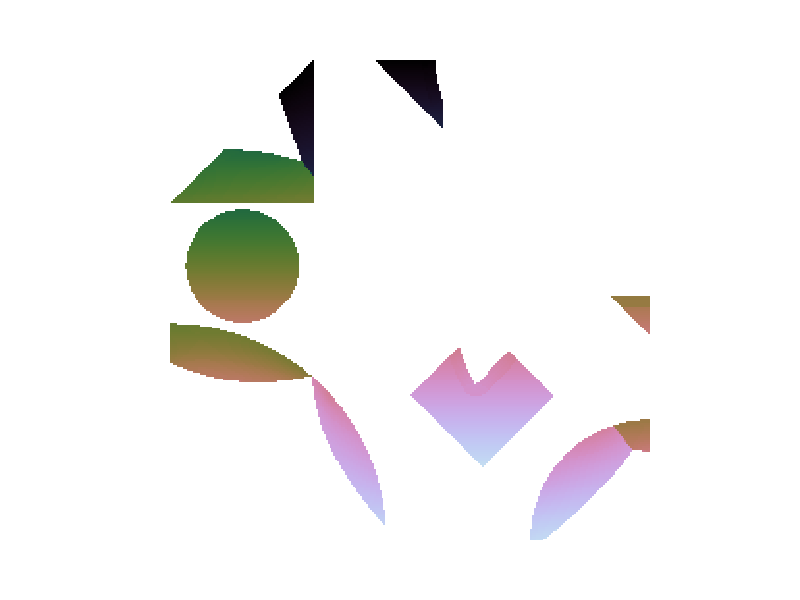}
}}}
\caption{$0.05\times \TV$}
\end{subfigure}
\caption{Representation of the transport map for the experiments of Figure \ref{fig_2Dmarginals}. 
}
\label{fig_2Dplans}
\end{figure}

\paragraph{Color transfer.}
In general it is difficult to display optimal transport maps for three-dimensional problems. An interesting application which allows intuitive visualization is color transfer: a classical task in image processing where the goal is to impose the color histogram of one image onto another image. 
Optimal transport between histograms has proven useful for problems of this sort such as contrast adjustment~\cite{delon2006movie} and color transfer via 1D transportation~\cite{pitie2007automated}. Indeed, optimal transport produces a correspondence between histograms which minimizes the total amount of color ``distortion'' (where the notion of distortion is specified by the cost function) and thus maintains maximal visual consistency. 
%
%

In our experiments we represent colors in the three-dimensional ``CIE-Lab'' space (one coordinate for luminance and two for chrominance), resized to fit into a cuboid  $X=Y=[0,1]^3$, discretized into $I=J=64 \times 32 \times 32$ uniform bins and we choose the quadratic cost $c(x,y) = |x-y|^2$. The anisotropic discretization of $X$ account for the fact that the eye is more sensitive to variations in luminance than variations in chrominance.

Let $\Omega \subset \RR^2$ be the image domain. An image is described by a function $g : \Omega \rightarrow X$ and its color histogram is the pushforward of the Lebesgue measure on $\Omega$ by $g$.
%
Let $g_X : \Om \to X$ and $g_Y: \Om \to Y(=X)$ be two images and $p$, $q$ be the densities of the associated color histograms with respect to the reference measures $\mathbf{\d x} = \mathbf{\d y} = \ones_I$ which gives unit mass to each point of $X$. We run Algorithm \ref{algo_scaling_discrete} with the ``separable kernel'' method (see Section \ref{sec_implementation}) to obtain an (unbalanced) optimal transport plan $\mathbf{R}$.
An approximate transport map $T$ is then computed according to the barycentric projection (as already used for instance in~\cite{2015-solomon-siggraph})
\eq{
T(x_i)=\mathbf{T}_i \qwhereq
\mathbf{T} \eqdef (\mathbf{R}. \mathbf{y}^1,\mathbf{R}.\mathbf{y}^2,\mathbf{R}. \mathbf{y}^3)/\mathbf{R}. \ones_J
} 
where $\mathbf{y}=(\mathbf{y}^1,\mathbf{y}^2,\mathbf{y}^3)\in \RR^{J\times 3}$ is the vector of coordinates of the points in $Y(=X)$. The modified image is finally obtained as $T\circ g_X$.

On Figure \ref{fig_colortransfer}, we display the color transfer between very dissimilar images, computed with parameter $\epsilon = 0.002$. The algorithm was stopped after $2000$ iterations and the running time was approximately $160$ seconds. This intentionally challenging example is insightful as it exhibits a strong effect of the choice of the divergence.
There are no quantitative measures for the quality of a transformed image, but the application of unbalanced optimal transport allows to select the ``right amount of colors'' in the target histogram so as to match the modes of the initial histogram and yields meaningful results.

\begin{figure}
 \centering
\begin{subfigure}{0.34\linewidth} 
\centering
 \resizebox{1.\linewidth}{!}{
\includegraphics[clip,trim=0cm 0cm 0cm 0cm]{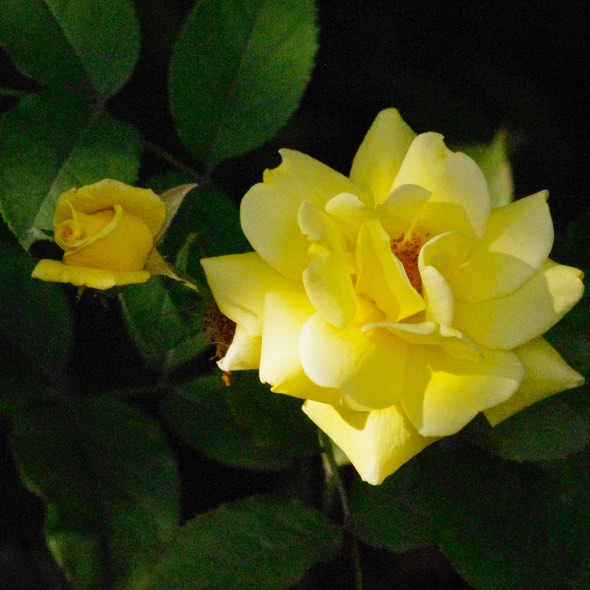}
}
\caption{Image $g_X$}
\end{subfigure}%
\quad
\begin{subfigure}{0.34\linewidth} 
\centering
 \resizebox{1.\linewidth}{!}{
\includegraphics[clip,trim=0cm 0cm 0cm 0cm]{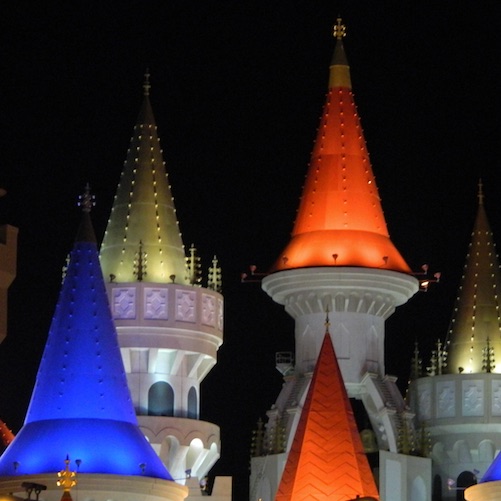}
}
\caption{Image $g_Y$}
\end{subfigure}
\hfill
\begin{subfigure}{0.24\linewidth} 
\centering
 \resizebox{1.\linewidth}{!}{
\includegraphics[clip,trim=0cm 0cm 0cm 0cm]{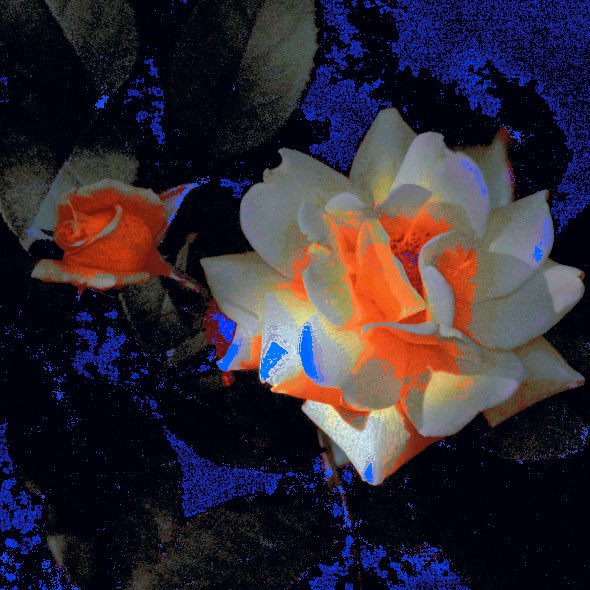}
}
\caption{$\iota_{\{=\}}$}
\end{subfigure}%
\begin{subfigure}{0.24\linewidth} 
\centering
 \resizebox{1.\linewidth}{!}{
\includegraphics[clip,trim=0cm 0cm 0cm 0cm]{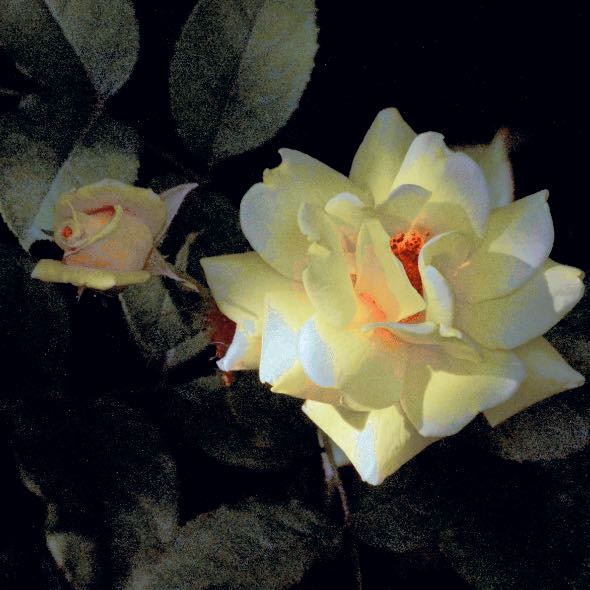}
}
\caption{$0.03\times \KL$}
\end{subfigure}%
\begin{subfigure}{0.24\linewidth} 
\centering
 \resizebox{1.\linewidth}{!}{
\includegraphics[clip,trim=0cm 0cm 0cm 0cm]{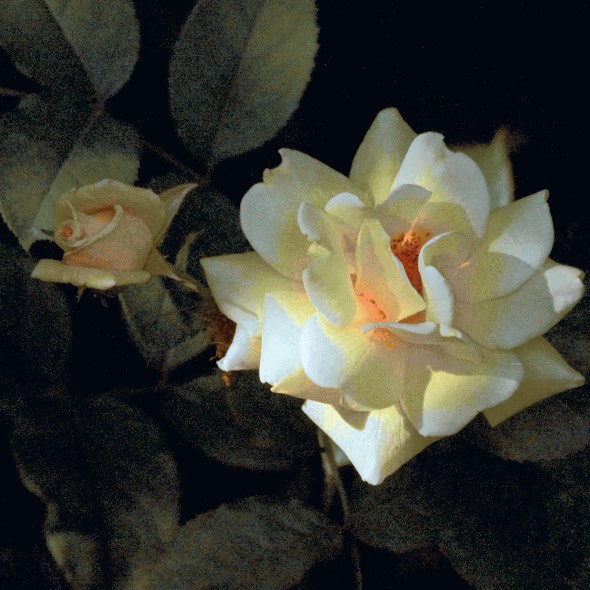}
}
\caption{$\RG_{[0, \, 5]}$}
\end{subfigure}%
\begin{subfigure}{0.24\linewidth} 
\centering
 \resizebox{1.\linewidth}{!}{
\includegraphics[clip,trim=0cm 0cm 0cm 0cm]{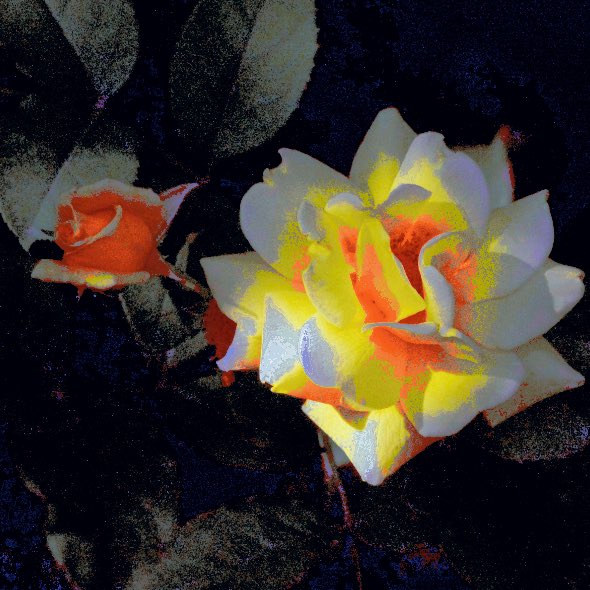}
}
\caption{$0.2\times \TV$}\label{subfig:colorTV}
\end{subfigure}%
\caption{A challenging color transfer experiment where the colors of the image $g_Y$ are transferred to the image $g_X$. In all cases $F_1=\iota_{\{=\}}(\cdot\,|\,p)$ and $F_2$ is the divergence with respect to $q$ specified in the caption. Note that in (F) some colors are not ``displaced''. The parameters for $F_{2}$ are chosen arbitrarily.}
\label{fig_colortransfer}
\end{figure}


\subsection{Unbalanced Barycenters}
\label{sec_appli_bary}

\paragraph{Well-posedness.}
Barycenters and related problems have been defined in Section \ref{subsec_barycenter}. Assume we are given a family of $n$ measures $(p_k \d x )_{k=1}^n$, families of entropies $(\phi_{1,k})_{k=1}^n$ and $(\phi_{2,k})_{k=1}^n$ and families of cost functions $(c_k)_{k=1}^n$. For a clearer picture, let us focus on the case where 
\eq{
 \phi_{2,k}=\alpha_k \cdot \lambda \cdot \varphi \quad \text{for} \quad k=1,\ldots,n
 }
where $\varphi$ is a nonnegative entropy function, $(\alpha_k)_{k=1}^n \in ]0,\infty[^n$ are weights and $\lambda>0$ is a (redundant) parameter. It is also convenient to slightly modify (for this Section only) the definition of the $\KL$ divergence given in \eqref{eq_KLdens} by introducing weights as
\eq{
\KL(r|s) \eqdef \sum_k \alpha_k \cdot \KL(r_k|s_k) \, .
}
No theoretical aspect is affected by this small change, but the definition of the kernel has to be adapted, for each component, as
\eq{
K_k = \exp (-{c_k}/(\alpha_k\, \epsilon))
}
so that one still has $\langle c, r \rangle + \epsilon \KL(r|1) = \epsilon \KL(r|K)+\textnormal{const}$.
By equation \eqref{eq_barycenter}, the barycenter problem with entropic regularization corresponds to defining   
\eq{
F_1(r)  = \sum_{k=1}^n \Divergm_{\phi_{k,1}}(r_k \d x| p_k \d x)
\qandq
F_2(s) = \inf_{\sigma\in \Mm_+(Y)} \lambda \sum_{k=1}^n  \alpha_k \Divergm_{ \phi}(s_k \d y| \sigma)
}
(for all $r\in \Lun(X)^n$ and $s\in \Lun(Y)^n$) in \eqref{eq-general-regul-KL}. It is direct to see that the proximal operator of $F_1$ for the $\KL$ divergence (according to our specific definition of $\KL$) can be computed componentwise as 
\eql{\label{eq_proxcomponentbary}
\prox^{\KL}_{F_1}(r) = (\prox_{\Diverg_{\phi_{1,1}}(\cdot | p_1)}^{\KL}(r_1),\dots,\prox_{\Diverg_{\phi_{n,1}}(\cdot | p_n)}^{\KL}(r_n)) 
}
so we can use the results from the previous section (recall that $\Diverg$ denotes divergences between functions as in \eqref{eq_divergencefunctions}). Let us turn our attention to the function $F_2$. As it is assumed $\phi(0)\geq0$, $F_2$ is not changed by taking the infimum over $\sigma$ of the form $h \cdot \d y$ for $h\in \Lun(Y)$.
%
So one can express $F_2$ for $s\in \Lun(Y)^n$, with notation of \eqref{eq_divergencefunctions} as
\eql{\label{eq_Fbary}
F_2(s) = \min_{h\in \Lun(Y)}  \lambda \sum_{k=1}^n \alpha_k \Diverg_{\phi} (s_k|h) =  \min_{h\in \Lun(Y)}  \lambda \sum_{k=1}^n \alpha_k \int_X \ol{\Diverg}_{\phi} (s_k(x)|h(x)) \d x\, .
}
\begin{proposition}
If $\phi'_\infty>0$ then $F_2$ is an admissible integral functional in the sense of Definition \ref{def_integralfunctional} and for all $s\in \Lun(Y)^n$, there exists a minimizer $h\in \Lun(Y)$. 
Moreover, if $r^\star \in \Lun(X\times Y)^n$ minimizes \eqref{eq-general-regul-KL}, then the associated minimizer $h^\star$ is a pointwise minimizer of \eqref{eq_Fbary} at the point $s=P^Y_\# r^\star$.
\end{proposition}
\begin{proof}
Let $G(s,h)$ be the function on the right side, which is an admissible integral functional (see Proposition \ref{prop_divergnormalint} in Appendix \ref{sec:ApxDivergences}). Let us verify the assumptions of the ``reduced minimization Theorem'' \cite[Corollary 3B]{rockafellar1976integral}. Assumption (i) is satisfied because $\phi'_\infty>0$ guarantees the growth condition \cite[2R]{rockafellar1976integral} and assumption (ii) is guaranteed by the fact that if $u\in \Lun(Y)$ and $\Diverg_\phi(u|v)<\infty$ then $v\in \Lun(Y)$ (this is proven by adapting slightly the proof of Lemma \ref{lem_KL_L1}, using the---at least linear--- growth of $\phi$). Thus the cited Corollary applies.
\end{proof}
%
%
\paragraph{Derivation of the algorithm.}
According to Proposition \ref{prop_proxKLcont}, computing the $\prox$ and $\proxdiv$ operators requires to solve, for each point $y \in Y$, a problem of the form
\eql{
\label{eq_prox_barycenter}
\min_{(\tilde{s},h)\in \RR^n\times \RR}
\sum_{k=1}^n \alpha_k \left( \epsilon \cdot \ol{\KL} (\tilde{s}_k |s_k) + \lambda \cdot \ol{\Diverg}_\phi (\tilde{s}_k | h) \right)\,.
}
If $\phi$ is smooth, first order optimality conditions for \eqref{eq_prox_barycenter} are simple to obtain. The next Proposition deals with the general case where more care is needed.
\begin{proposition}
\label{prop_barycenter_general}
Let $(s_i)_{i=1}^n\in \RR_+^n$. Assume that there exists a feasible candidate $(\tilde{s}^0,h^0)$ for \eqref{eq_prox_barycenter} such that $(s_i>0) \Rightarrow (\tilde{s}^0_i>0)$. A candidate $(\tilde{s},h)$ is a solution of \eqref{eq_prox_barycenter} if and only if 
\begin{itemize}
\item $(s_i=0) \Leftrightarrow (\tilde{s}_i=0)$ and 
\item there exists $b \in\RR^n$ such that $\sum_{k=1}^n \alpha_k b_k =0$ and, for all $k\in \{1,\dots, n\}$, $(a_k,b_k) \in \partial \Diverg_{\phi}(\tilde{s}_k | h)$ with $a_k \eqdef \frac{\epsilon}{\la} \log \frac{s_k}{\tilde{s}_k}$ if $s_k>0$ and $b_k \in \partial_2 \Diverg_\phi(0,h)$ otherwise.
\end{itemize}
\end{proposition}
%
\begin{proof}
First assume that $s_i>0$ for all $i\in \{1,\dots ,n\}$. The positivity assumption on the feasible point implies that the sum of the subdifferential is the subdifferential of the sum by continuity of $\KL$ for positive arguments.  Moreover, a minimizer necessarily satisfies $\tilde{s}_i>0$ for all $i$. Consequently, the subderivative of the function in \eqref{eq_prox_barycenter} at $((\tilde{s}_i)_n,h)$ is the set of vectors in $\RR^{n+1}$ of the form
\eq{
\begin{pmatrix}
\vdots \\
\epsilon \alpha_i \log (\tilde{s}_i/s_i) + \lambda \alpha_i a_i  \\
\vdots \\
\lambda \sum_k \alpha_k b_k
\end{pmatrix}
}
with $(a_i,b_i) \in \partial \Diverg_{\phi}(\tilde{s}_i | h)$. Writing the first order optimality condition yields the second condition of the Proposition. Now for all $i$ such that $s_i$ is null, set $\tilde{s}_i=0$ (this is required for feasibility) and do the reasoning above by withdrawing the variables $\tilde{s}_i$ which have been fixed.
\end{proof}
Note that once the optimal $h$ is found (possibly with the help of the optimality conditions of Proposition \ref{prop_barycenter_general}), determining the optimal values for $\tilde{s}$ can be done componentwise as in \eqref{eq_proxcomponentbary} with the help of Proposition \ref{prop_UOTprox}.
In Table \ref{prop_estimatebary} we provide formulas for $h$ for some examples (proofs can be found in Appendix \ref{sec:AppendixBarycenterIterates}), for the subsequent computation of $\tilde{s}_k$ (and the $\proxdiv$ step) we refer the reader to Table \ref{table_proximalexplicit} (where of course one must replace $p$ by $h$).
\begin{table}[h] \centering
\begin{tabular}{@{}lrcl@{}}\toprule
&$\Diverg_\phi$ && Formula for $h$ as a function of $s\in \RR^n$ \\ \midrule \addlinespace 
(i)&$ \iota_{\{=\}} $
&&$h= \left( \prod_k s_k^{\alpha_k} \right)^{\frac{1}{\sum_k \alpha_k}}$
\\ \addlinespace \addlinespace 
(ii)&$\lambda \KL  $
&&  $h = \left( \frac{\sum_k \alpha_k s_k^{\frac{\epsilon}{\epsilon+\la}} }{\sum_k \alpha_k} \right)^{\frac{\epsilon + \la}{\epsilon}}$
\\ \addlinespace \addlinespace 
(iii)&$\lambda \TV$
&& if $\sum_{k\notin I_+} \alpha_k \geq \sum_{k\in I_+} \alpha_k $ then $h=0$ otherwise solve:\\ \addlinespace
&&& $\sum_{k\notin I_+} \alpha_k+ \sum_{k\in I_+}\alpha_k \max \left( -1 , \min \left( 1 , \frac{\epsilon}{\lambda } \log \frac{h}{s_i}\right) \right) = 0$
\\ \addlinespace \addlinespace
(iv)&$ \RG_{[\beta_1,\beta_2]} $
&&  if $s_k=0$ for some $k$ then $h=0$ otherwise solve:\\ \addlinespace
&&&$\sum_k \alpha_k \left[ \beta_2 \min\left( \log \frac{\beta_2\, h}{s_k} ,0 \right) + \beta_1 \max \left( \log \frac{\beta_1\, h}{s_k} ,0\right) \right] =0$
\\ \addlinespace \addlinespace \bottomrule
\end{tabular}
\caption{Expression for the minimizer $h$ of \eqref{eq_prox_barycenter} as a function of $s\in \RR^n$ where $I_+ = \{k \, ; \, s_k>0 \}$ (proofs in Appendix \ref{sec:AppendixBarycenterIterates}). For the implicit equations of cases (iii) and (iv), an exact solution can be given quickly because computing $\log h$ consists in finding the root of a piecewise linear non-decreasing function (with at most $2n$ pieces), which is guaranteed to change its sign.}
\label{prop_estimatebary}
\end{table}
\paragraph{Numerical experiments.}
Geodesics in Wasserstein and Wasserstein--Fisher--Rao space can be computed as weighted barycenters between their endpoints. In Figure \ref{fig:Geodesics} geodesics for the Wasserstein--Fisher--Rao distance on $X=Y=[0,1]$ for different cut-loci $d_{\max}$ and the Wasserstein distance (corresponding to $d_{\max}=\infty$) are compared and the influence of entropy regularization is illustrated.
The interval $[0,1]$ was discretized into $I=J=512$ uniform samples. $\mathbf{\d x}$ and $\mathbf{\d y}$ were chosen to be the discretized Lebesgue measure.

In Figure \ref{fig_1Dbarycenter}, we display Fr\'echet means-like experiments for a family of $4$ given marginals where $X=Y$ is the segment $[0,1]$ (discretized as above). The (discrete) densities of the marginals $(\mathbf{p}_k)_{k=1}^4$ consist each of the sum of three bumps (centered near the points $x=0.1$, $x=0.5$ and $x=0.9$). These computations where performed with Algorithm \ref{algo_scaling_stabilized} which was stopped after $1500$ iterations (running time of $30$ seconds approximately) and with $\epsilon=10^{-5}$. We observe that relaxing the marginal constraints (Figures \ref{fig_bary1d_KL}-\ref{fig_bary1d_WF}) allows to conserve this structure in three bumps in contrast to classical optimal transport (Figure \ref{fig_bary1d_OT}).

\begin{figure}
	\centering %
	\includegraphics[width=12cm]{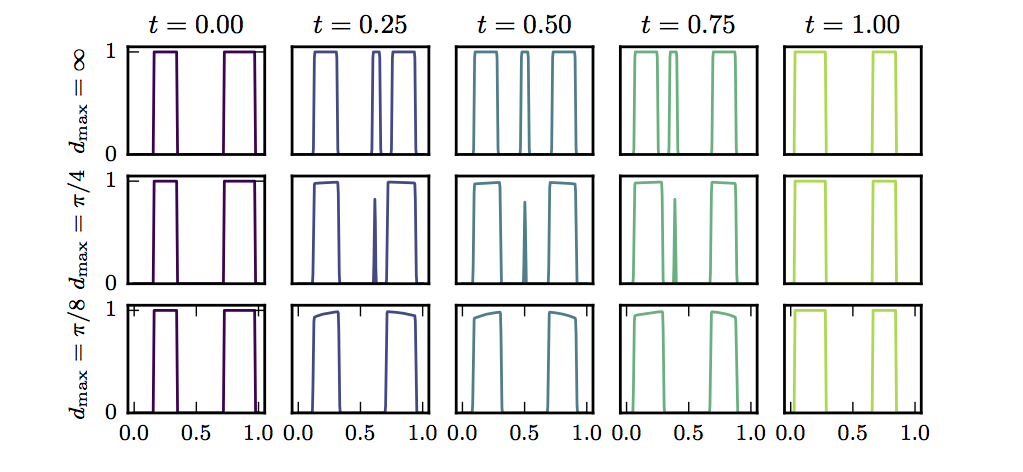} \\%
	\includegraphics[width=8cm]{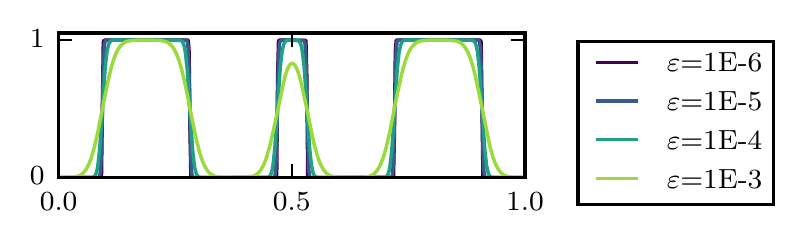} %
	\caption{%
	\textit{Top:} Geodesics in Wasserstein distance ($d_{\max}=\infty$) and Wasserstein--Fisher--Rao distance ($d_{\max}$ gives the cut-locus distance) as weighted (unbalanced) barycenters between endpoints. %
	For $d_{\max}=\infty$ the mass difference between left and right blocks must be compensated by transport. As $d_{\max}$ is reduced, the mass difference is increasingly compensated by growth and shrinkage. In all experiments $\epsilon=10^{-6}$. %
	\textit{Bottom:} Midpoint of the Wasserstein geodesic for various $\epsilon$.
	} %
	\label{fig:Geodesics}
\end{figure}

Figure \ref{fig_wasserstein_triangle} and \ref{fig_GHK_triangle} display barycenters for the Wasserstein and the $\GHK$ metric (defined in Section \ref{subsec_unbalanced}) between three densities on $[0,3]^2$ discretized into $200\times 200$ samples. Computations where performed using Algorithm \ref{algo_scaling_discrete} and the ``separable kernel'' method (see Section \ref{sec_implementation}) which was stopped after $1500$ iterations (running time of $70$ seconds approximately) with $\epsilon=9.10^{-4}$. The barycenter coefficients are the following:
\begin{gather*}
(0,1,0)\\
(1,3,0)/4 \qquad (0,3,1)/4 \\
(2,2,0)/4 \qquad (1,2,1)/4 \qquad (0,2,2)/4 \\
(3,1,0)/4 \qquad (2,1,1)/4 \qquad (1,1,2)/4 \qquad (0,1,3)/4\\
(1,0,0)    \qquad (3,0,1)/4 \qquad (2,0,2)/4 \qquad (1,0,3)/4 \qquad (0,0,1)
\end{gather*}
The input densities have a similar global structure: each is made of three distant ``shapes'' of varying mass. The comparison between Figures  \ref{fig_wasserstein_triangle} and \ref{fig_GHK_triangle} lead to a similar remark than for Figure \ref{fig_1Dbarycenter}: relaxing the strict marginal constraints allows to maintain the global ``structure'' of the input densities.

\begin{figure}
 \centering
\begin{subfigure}{0.5\linewidth} 
\centering
 \resizebox{1.\linewidth}{!}{
\includegraphics[clip,trim=0cm 0cm 0cm 0cm]{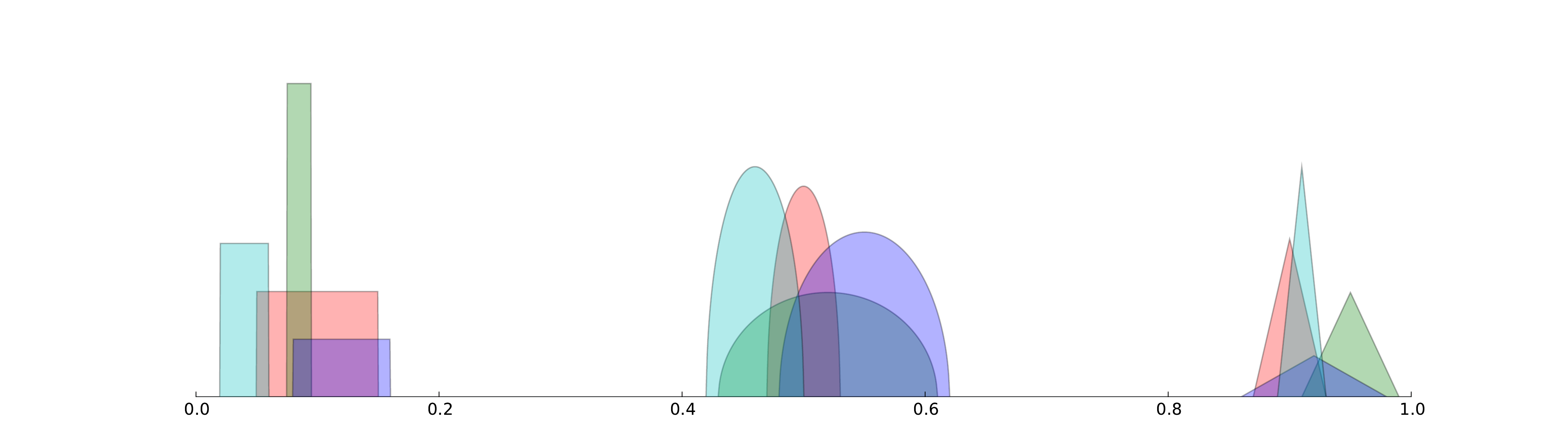}
}
\caption{Marginals $(\mathbf{p}_k)^4_{k=1}$}
\end{subfigure}%
\begin{subfigure}{0.5\linewidth} 
\centering
 \resizebox{1.\linewidth}{!}{
\includegraphics[clip,trim=0cm 0cm 0cm 0cm]{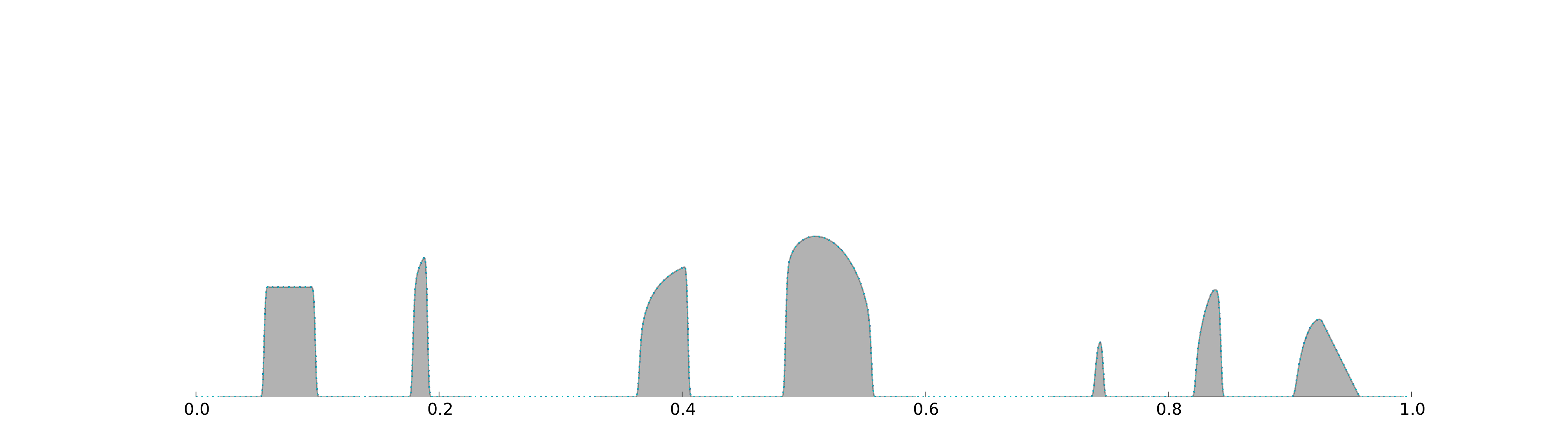}
}
\caption{$\iota_{\{=\}}$}\label{fig_bary1d_OT}
\end{subfigure}
\begin{subfigure}{0.5\linewidth} 
\centering
 \resizebox{1.\linewidth}{!}{
\includegraphics[clip,trim=0cm 0cm 0cm 0cm]{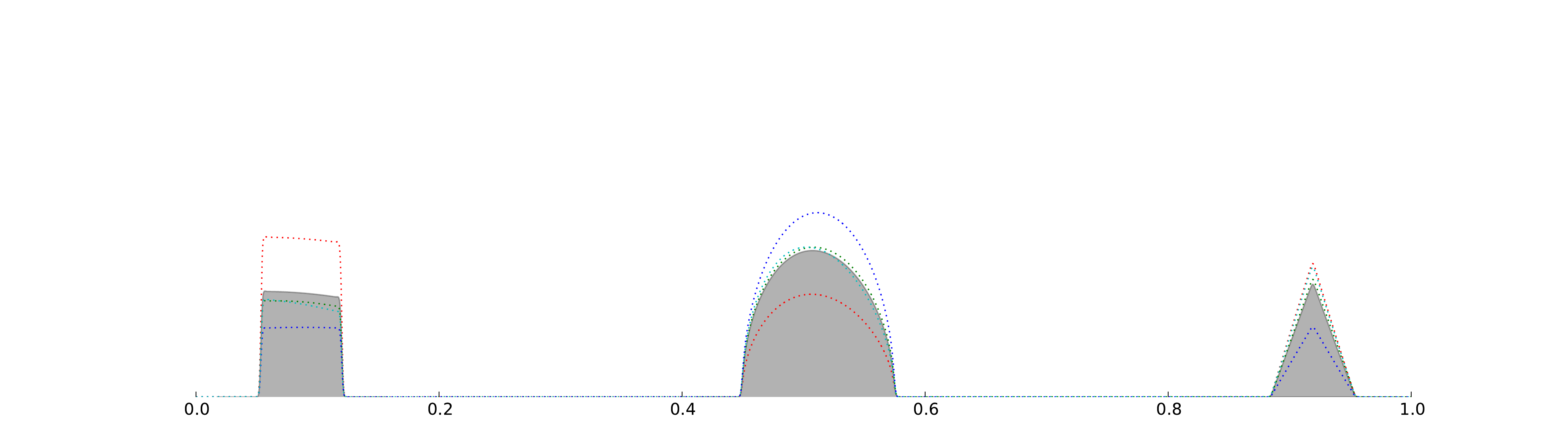}
}
\caption{$0.07\times \KL$}\label{fig_bary1d_KL}
\end{subfigure}%
\begin{subfigure}{0.5\linewidth} 
\centering
 \resizebox{1.\linewidth}{!}{
\includegraphics[clip,trim=0cm 0cm 0cm 0cm]{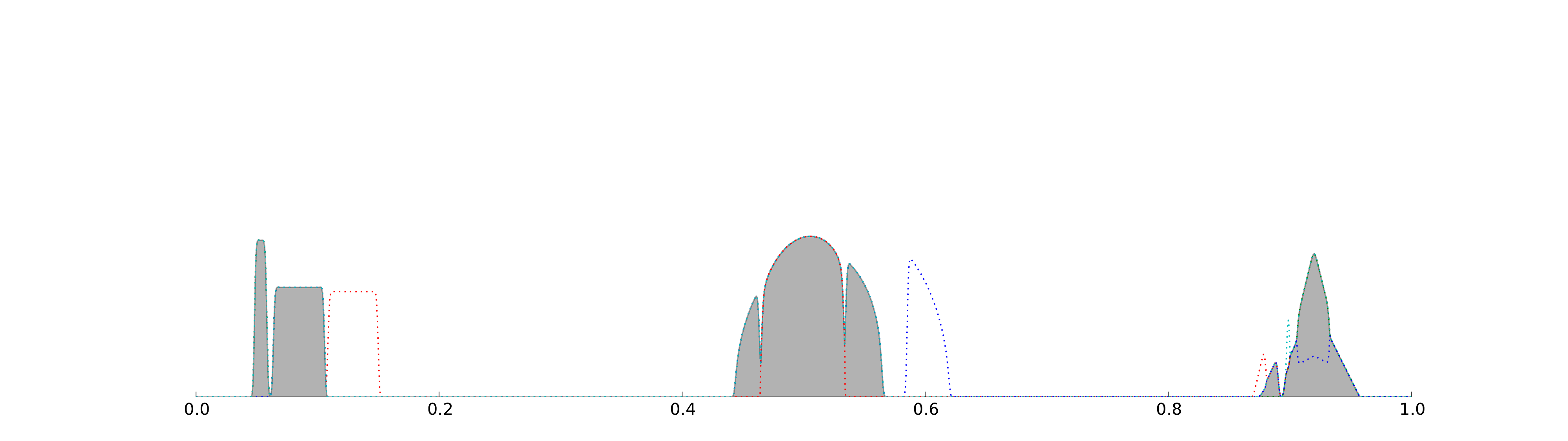}
}
\caption{$0.02\times \TV$}
\end{subfigure}
\begin{subfigure}{0.5\linewidth} 
\centering
 \resizebox{1.\linewidth}{!}{
\includegraphics[clip,trim=0cm 0cm 0cm 0cm]{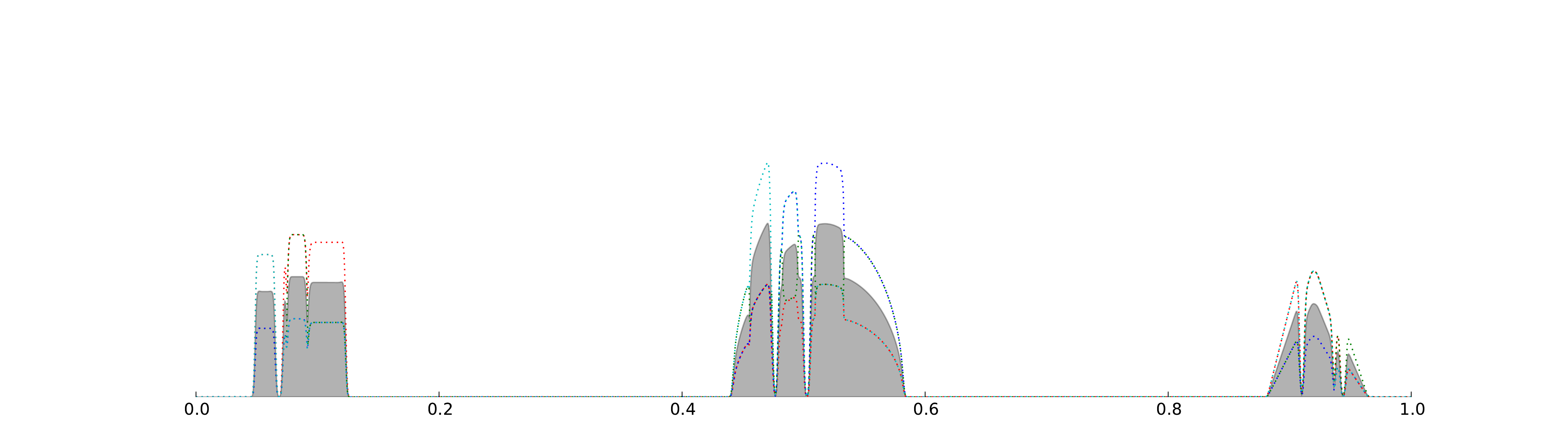}
}
\caption{$\RG_{[0.65,\,1.35]}$}
\end{subfigure}%
\begin{subfigure}{0.5\linewidth} 
\centering
 \resizebox{1.\linewidth}{!}{
\includegraphics[clip,trim=0cm 0cm 0cm 0cm]{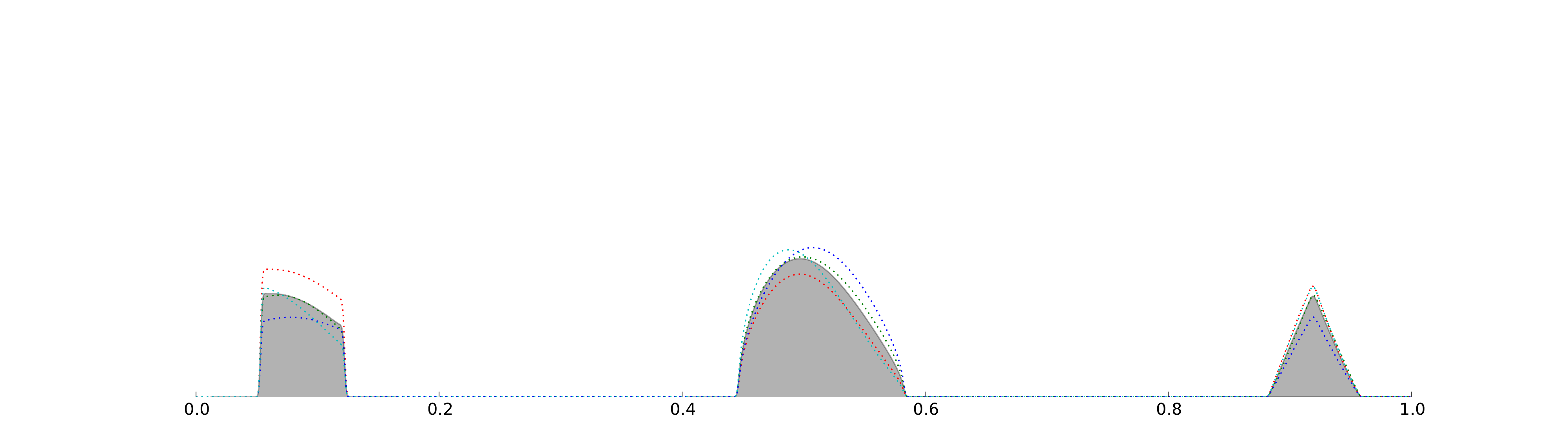}
}
\caption{$\WF$ (cut locus at $0.2$)}\label{fig_bary1d_WF}
\end{subfigure}%
\caption{Illustration of barycenter-like problems on $X=Y=[0,1]$. Except for (F), the function $F_1$ is the equality constraint with respect to the densities $(\mathbf{p}_k)_{k=1}^4$, the cost is $c(x,y)=|y-x|^2$, the weights are $(\tfrac14,\tfrac14,\tfrac14,\tfrac14)$ and the function $F_2$ is of the type \eqref{eq_Fbary} with the divergence specified in the legend. Figure (F) represents the Fr\'echet mean for the $\WF$ metric. The dotted lines display the second marginal of the optimal plans.}
\label{fig_1Dbarycenter}
\end{figure}

\pgfmathsetmacro{\ax}{33}
\pgfmathsetmacro{\ay}{25}
\pgfmathsetmacro{\b}{1.5}
\begin{figure}
 \centering 
 \CompressedVersion{
 	\includegraphics{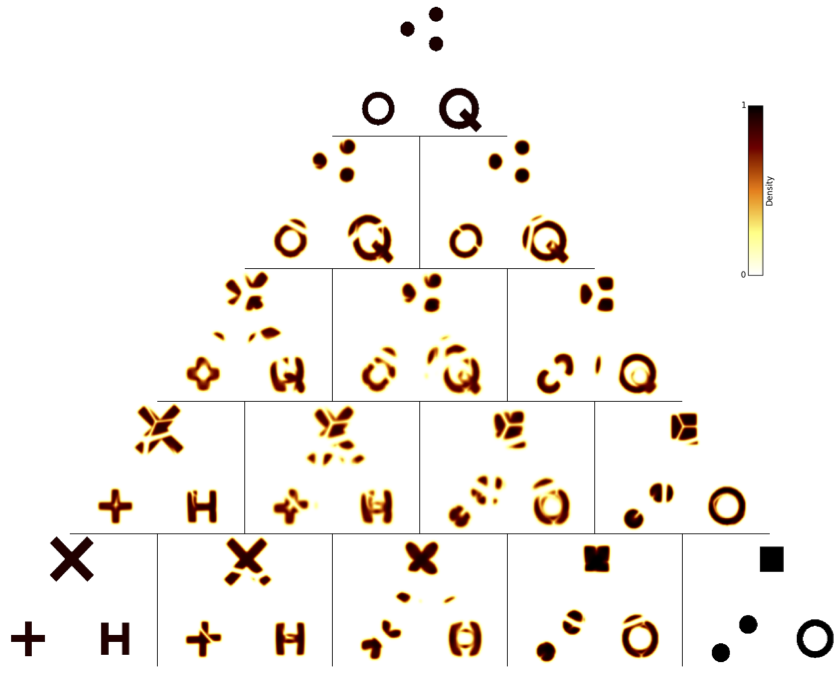}
 }{
  \resizebox{.9\linewidth}{!}{
\begin{tikzpicture}
\node (1) at (-\ax,-\ay) {\includegraphics[clip,trim=\b cm \b cm \b cm \b cm]{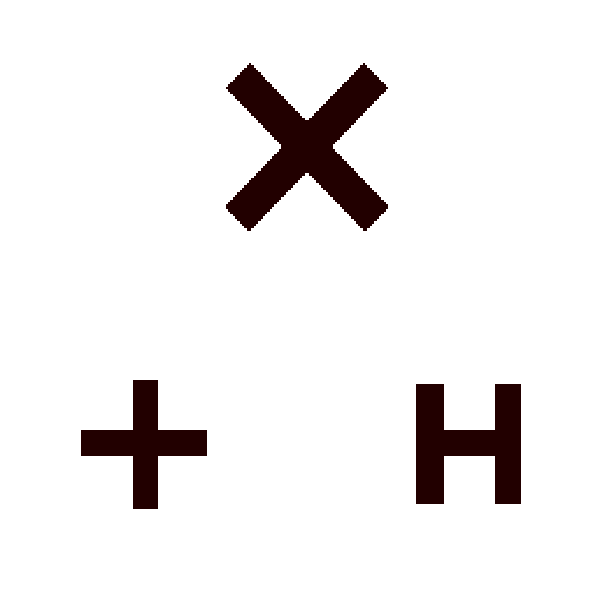}};
\node (2) at (0,\ay) {\includegraphics[clip,trim=\b cm \b cm \b cm \b cm]{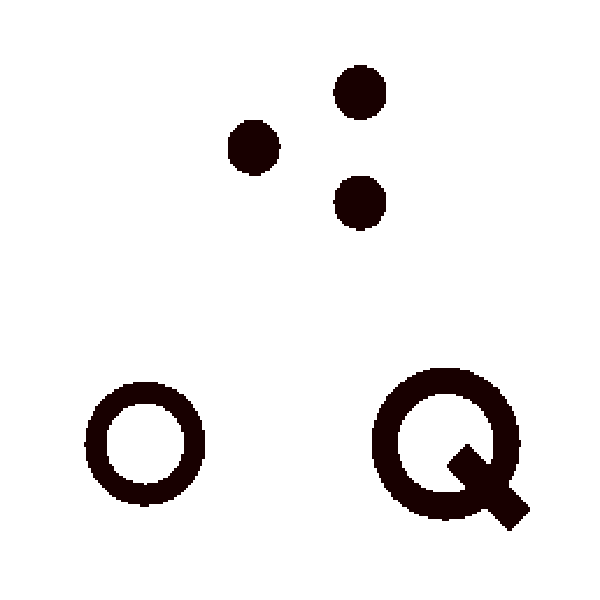}};
\node (3) at (\ax,-\ay) {\includegraphics[clip,trim=\b cm \b cm \b cm \b cm]{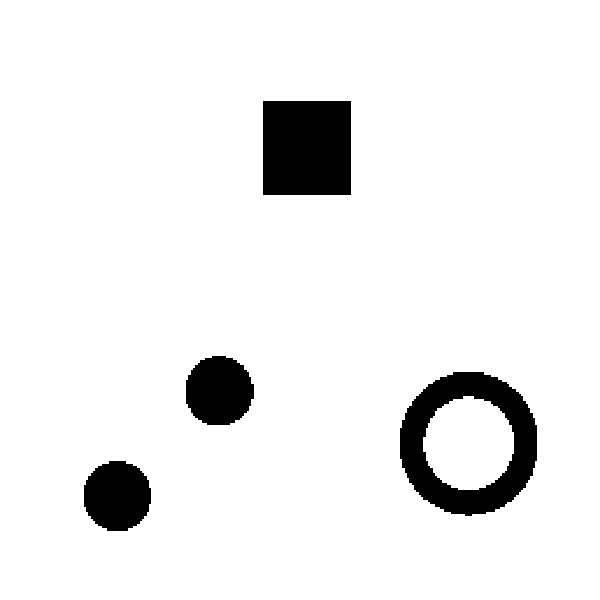}};

\node (4) at (-\ax/2,-\ay) {\includegraphics[clip,trim=\b cm \b cm \b cm \b cm]{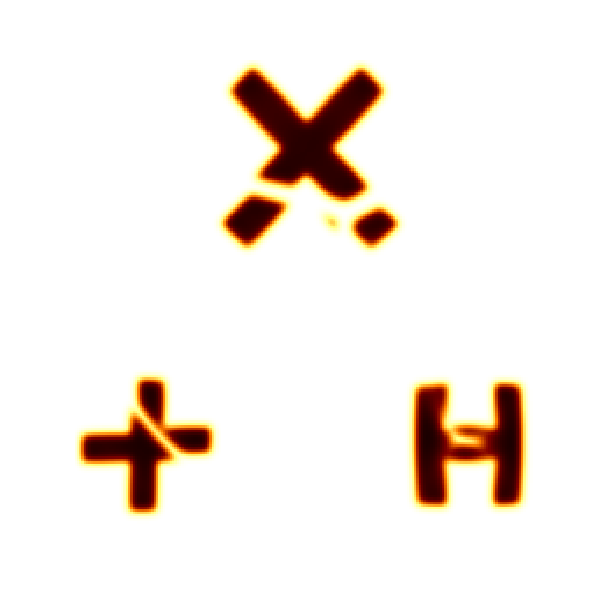}};
\node (5) at (0,-\ay) {\includegraphics[clip,trim=\b cm \b cm \b cm \b cm]{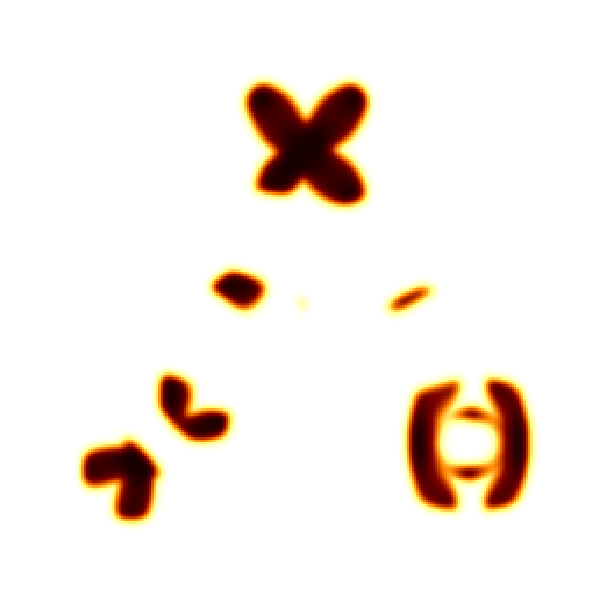}};
\node (6) at (\ax/2,-\ay) {\includegraphics[clip,trim=\b cm \b cm \b cm \b cm]{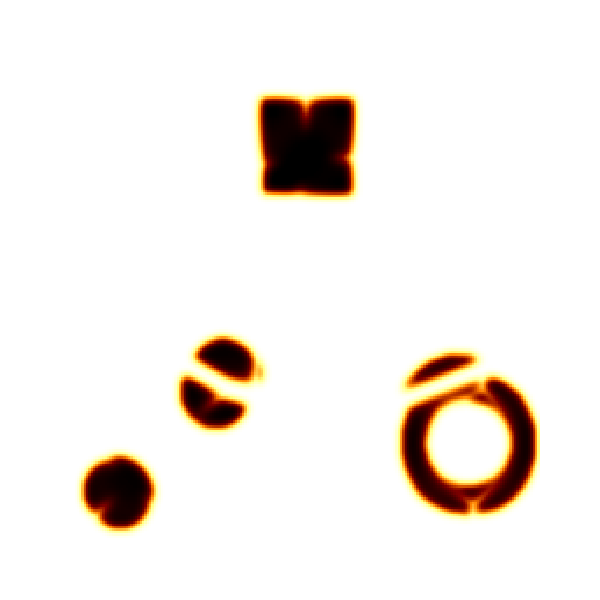}};

\node (7) at (-3*\ax/4,-2*\ay/4) {\includegraphics[clip,trim=\b cm \b cm \b cm \b cm]{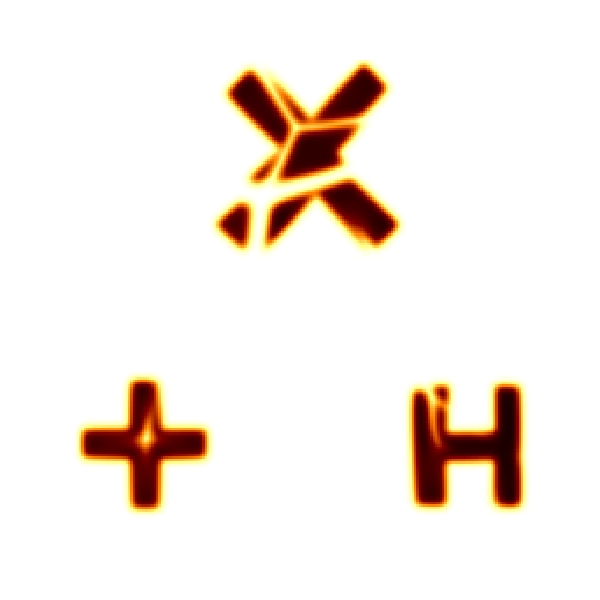}};
\node (8) at (-\ax/4,-2*\ay/4) {\includegraphics[clip,trim=\b cm \b cm \b cm \b cm]{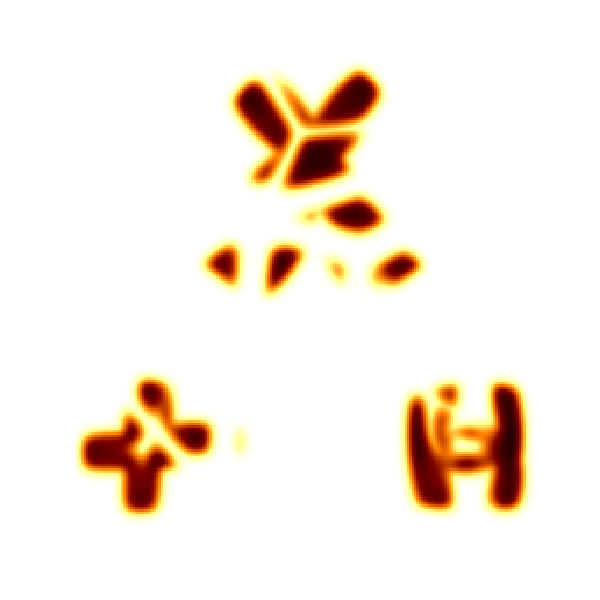}};
\node (9) at (\ax/4,-2*\ay/4) {\includegraphics[clip,trim=\b cm \b cm \b cm \b cm]{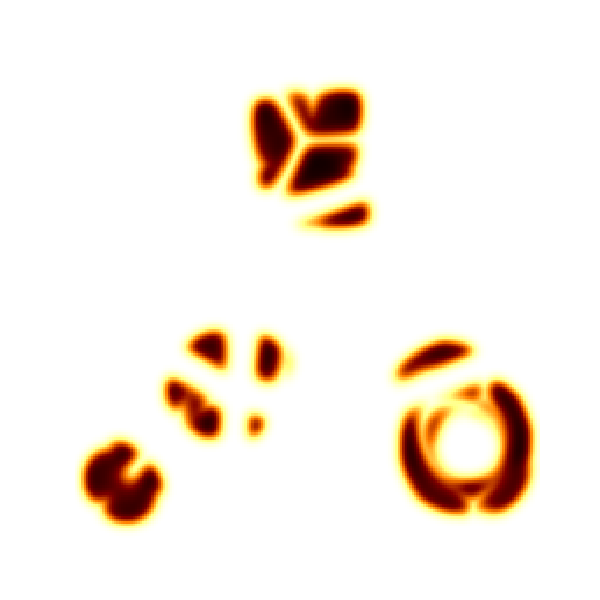}};
\node (10) at (3*\ax/4,-2*\ay/4) {\includegraphics[clip,trim=\b cm \b cm \b cm \b cm]{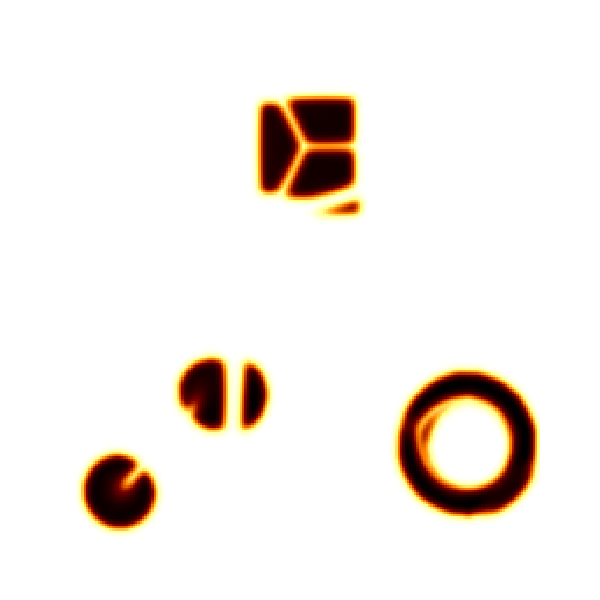}};

\node (11) at (-2*\ax/4,0) {\includegraphics[clip,trim=\b cm \b cm \b cm \b cm]{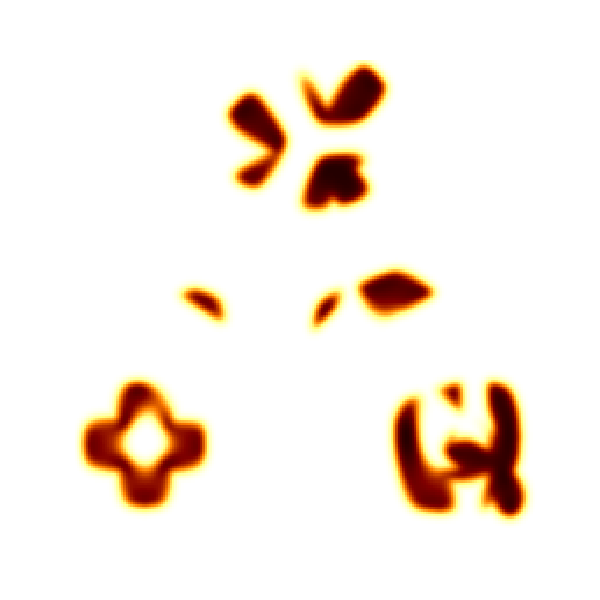}};
\node (12) at (0,0) {\includegraphics[clip,trim=\b cm \b cm \b cm \b cm]{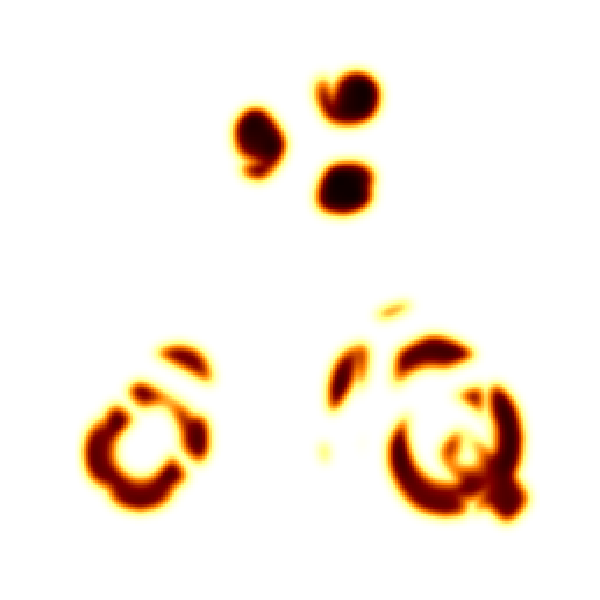}};
\node (13) at (2*\ax/4,0) {\includegraphics[clip,trim=\b cm \b cm \b cm \b cm]{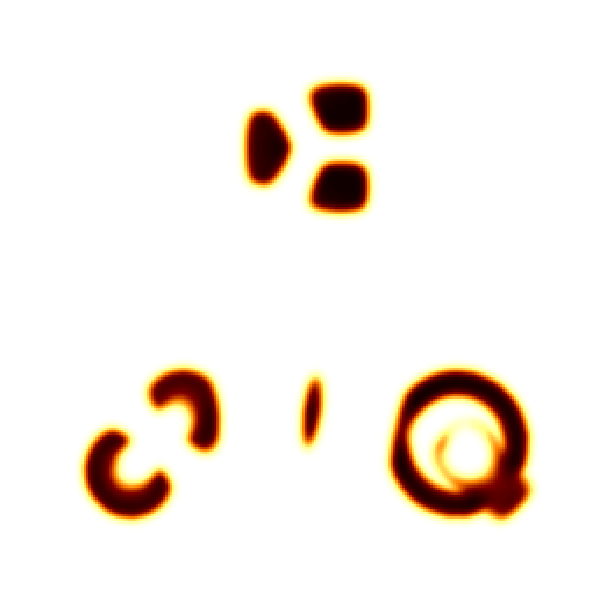}};

\node (14) at (-\ax/4,\ay/2) {\includegraphics[clip,trim=\b cm \b cm \b cm \b cm]{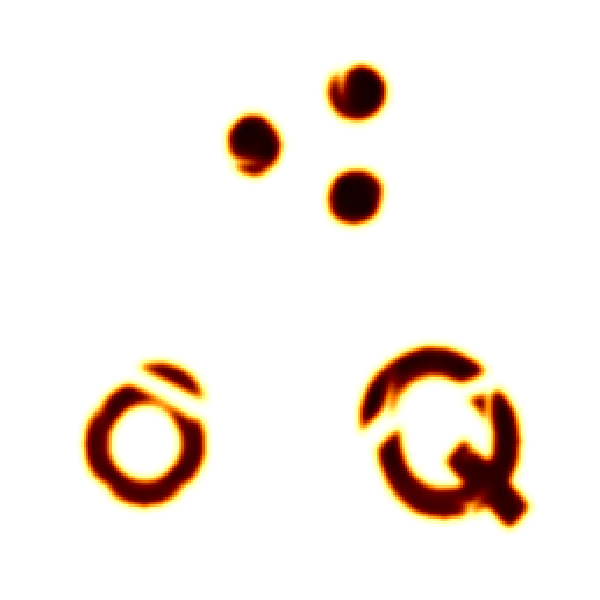}};
\node (15) at (\ax/4,\ay/2) {\includegraphics[clip,trim=\b cm \b cm \b cm \b cm]{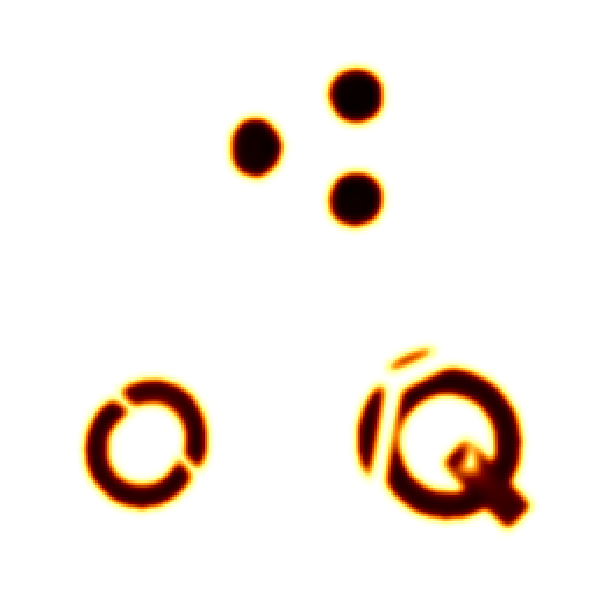}};

\node (15) at (\ax,\ay/2) {\includegraphics[clip,scale=2.5]{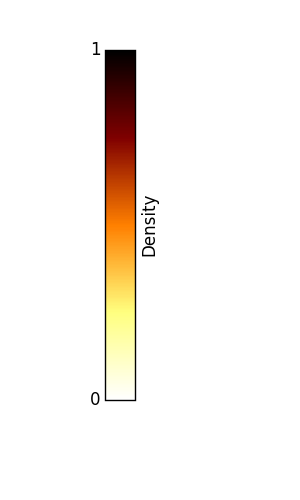}};

\draw[thick] 	(-\ax/4,3*\ay/4)    	-- 	(\ax/4,3*\ay/4);
\draw[thick] 	(-2*\ax/4,\ay/4)    	-- 	(2*\ax/4,\ay/4);
\draw[thick] 	(-3*\ax/4,-\ay/4)    	-- 	(3*\ax/4,-\ay/4);
\draw[thick] 	(-4*\ax/4,-3*\ay/4)    	-- 	(4*\ax/4,-3*\ay/4);

\draw[thick] 	(0,3*\ay/4)    	-- 	(0,1*\ay/4);
\draw[thick] 	(-\ax/4,1*\ay/4)    	-- 	(-\ax/4,-1*\ay/4);
\draw[thick] 	(\ax/4,1*\ay/4)    	-- 	(\ax/4,-1*\ay/4);
\draw[thick] 	(0,-\ay/4)    	-- 	(0,-3*\ay/4);
\draw[thick] 	(-\ax/2,-\ay/4)    	-- 	(-\ax/2,-3*\ay/4);
\draw[thick] 	(\ax/2,-\ay/4)    	-- 	(\ax/2,-3*\ay/4);
\draw[thick] 	(-\ax/4,-3*\ay/4)    	-- 	(-\ax/4,-5*\ay/4);
\draw[thick] 	(\ax/4,-3*\ay/4)    	-- 	(\ax/4,-5*\ay/4);
\draw[thick] 	(-3*\ax/4,-3*\ay/4)    	-- 	(-3*\ax/4,-5*\ay/4);
\draw[thick] 	(3*\ax/4,-3*\ay/4)    	-- 	(3*\ax/4,-5*\ay/4);

\end{tikzpicture}
        }
        }
\caption{Wasserstein barycenters with entropic smoothing}
\label{fig_wasserstein_triangle}
\end{figure}
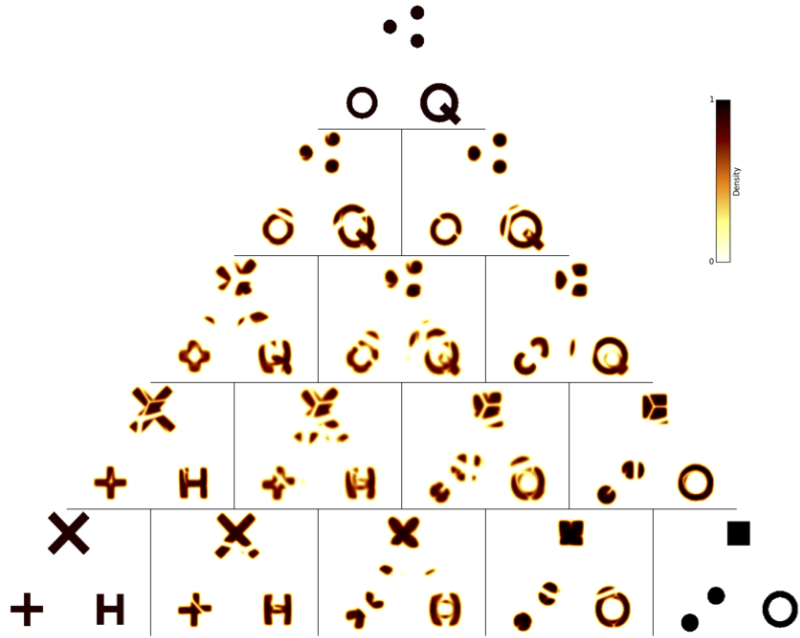

\begin{figure}
 \centering
  \CompressedVersion{
 	\includegraphics{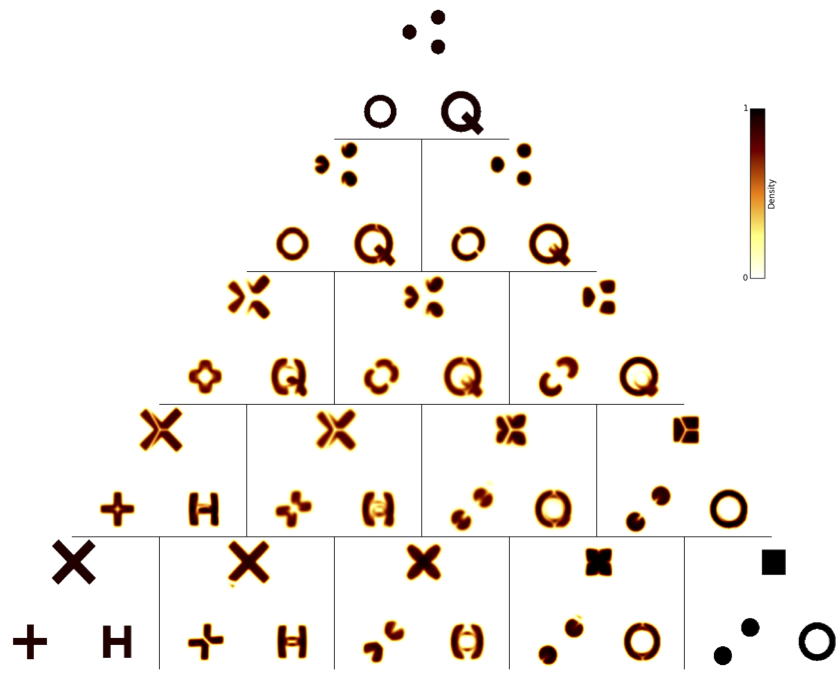}
 }{
  \resizebox{.9\linewidth}{!}{
\begin{tikzpicture}
\node (1) at (-\ax,-\ay) {\includegraphics[clip,trim=\b cm \b cm \b cm \b cm]{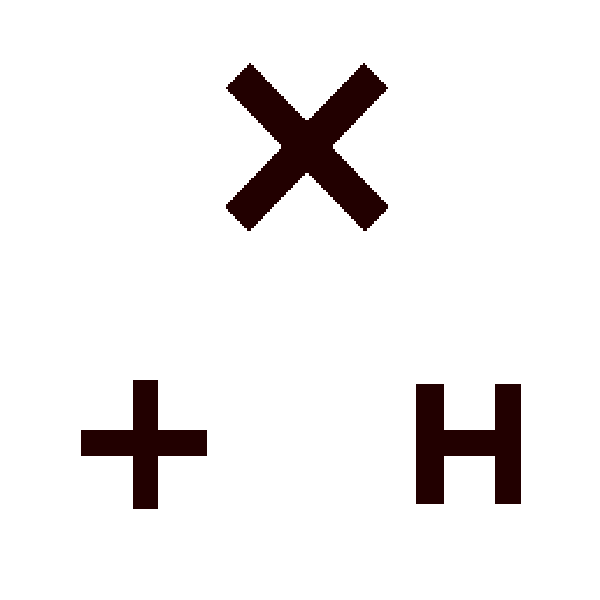}};
\node (2) at (0,\ay) {\includegraphics[clip,trim=\b cm \b cm \b cm \b cm]{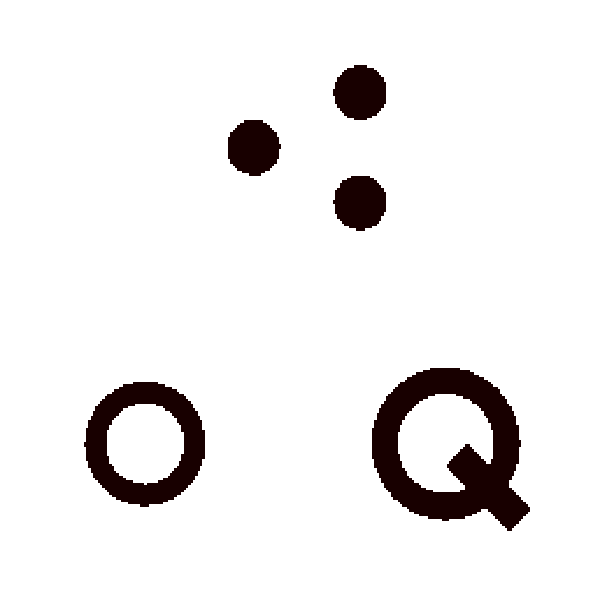}};
\node (3) at (\ax,-\ay) {\includegraphics[clip,trim=\b cm \b cm \b cm \b cm]{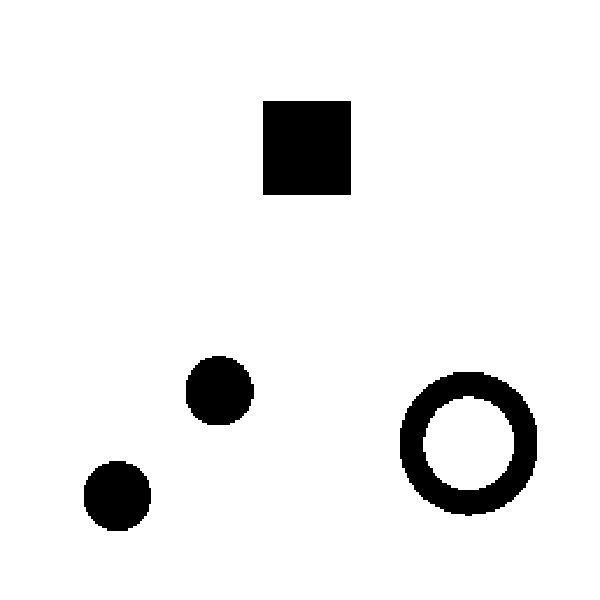}};

\node (4) at (-\ax/2,-\ay) {\includegraphics[clip,trim=\b cm \b cm \b cm \b cm]{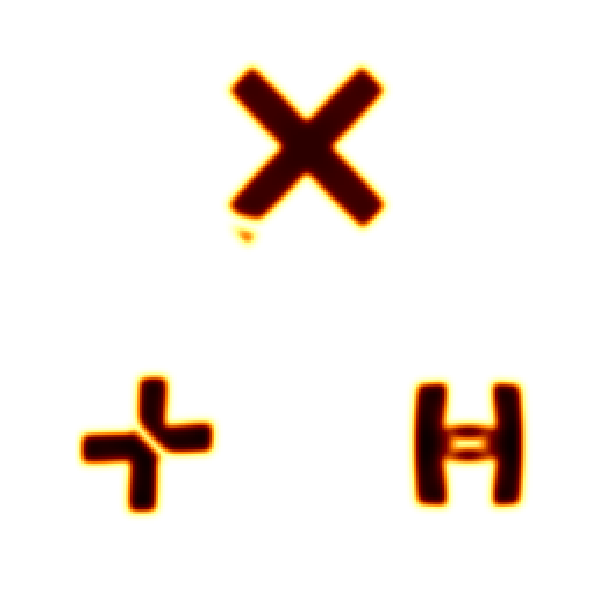}};
\node (5) at (0,-\ay) {\includegraphics[clip,trim=\b cm \b cm \b cm \b cm]{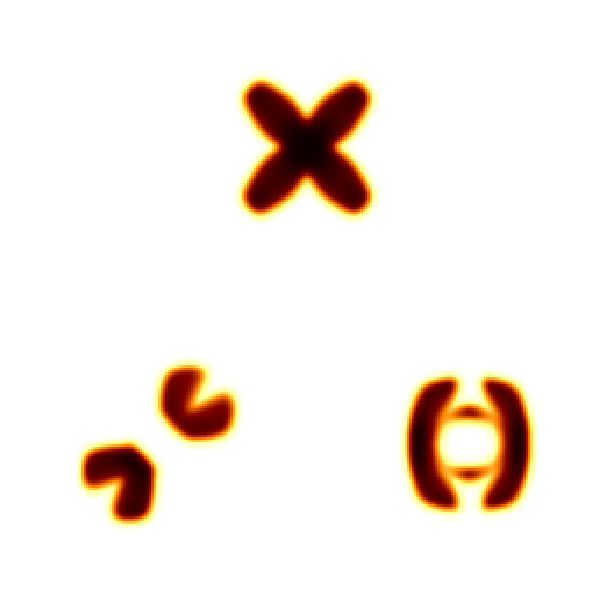}};
\node (6) at (\ax/2,-\ay) {\includegraphics[clip,trim=\b cm \b cm \b cm \b cm]{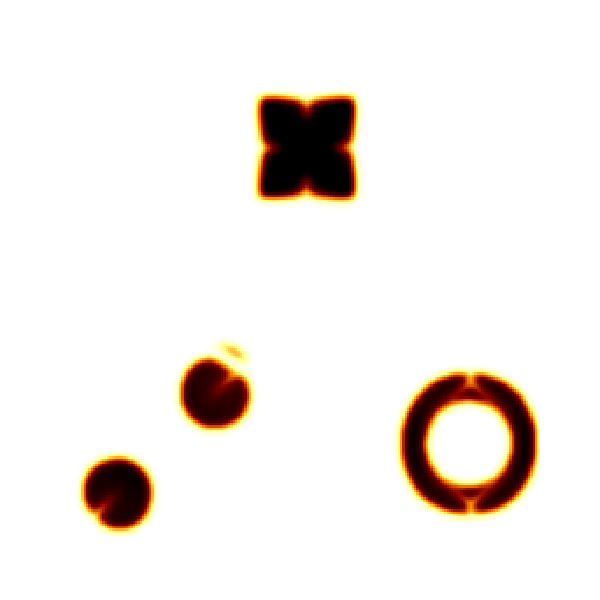}};

\node (7) at (-3*\ax/4,-2*\ay/4) {\includegraphics[clip,trim=\b cm \b cm \b cm \b cm]{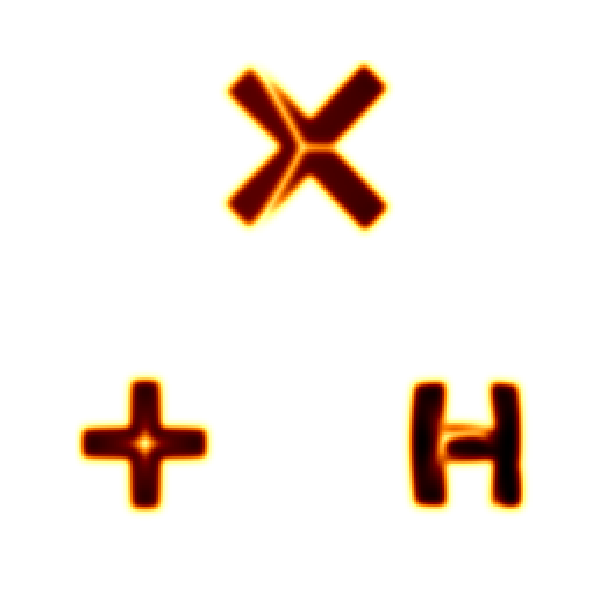}};
\node (8) at (-\ax/4,-2*\ay/4) {\includegraphics[clip,trim=\b cm \b cm \b cm \b cm]{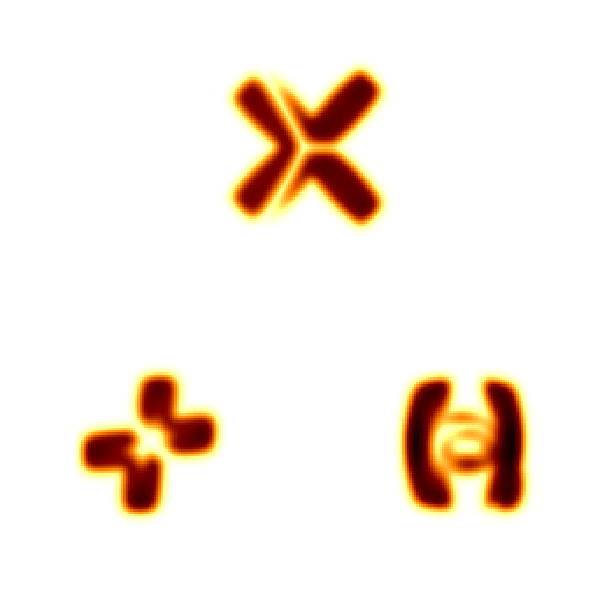}};
\node (9) at (\ax/4,-2*\ay/4) {\includegraphics[clip,trim=\b cm \b cm \b cm \b cm]{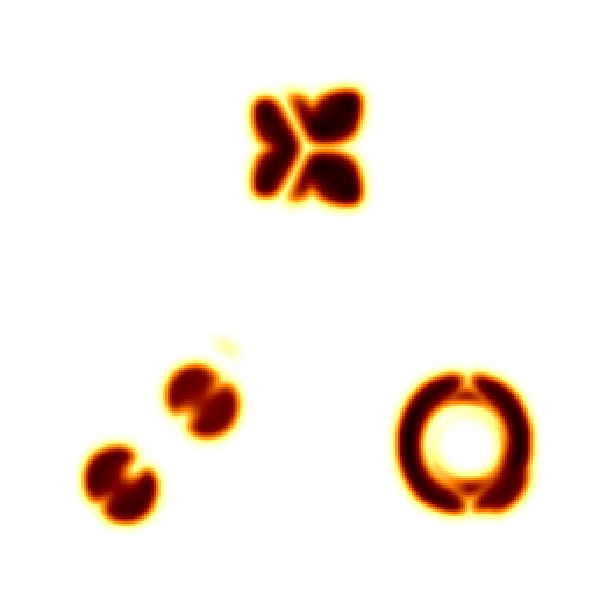}};
\node (10) at (3*\ax/4,-2*\ay/4) {\includegraphics[clip,trim=\b cm \b cm \b cm \b cm]{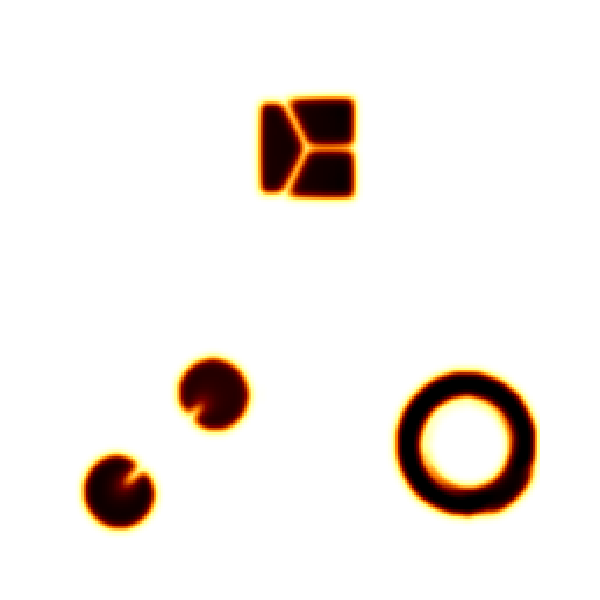}};

\node (11) at (-2*\ax/4,0) {\includegraphics[clip,trim=\b cm \b cm \b cm \b cm]{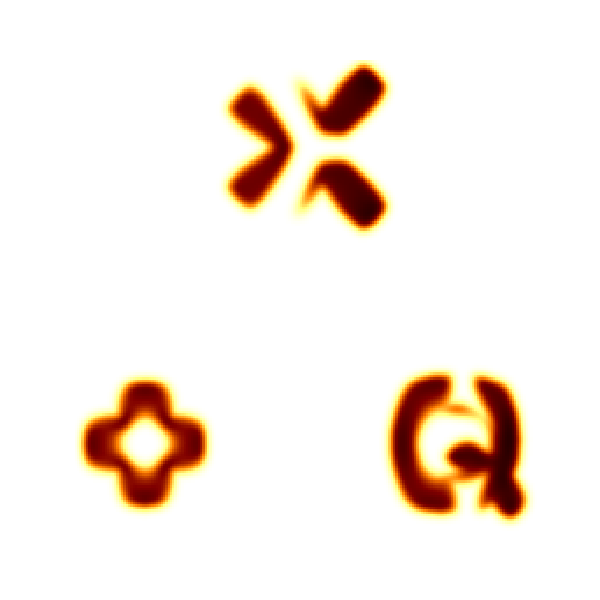}};
\node (12) at (0,0) {\includegraphics[clip,trim=\b cm \b cm \b cm \b cm]{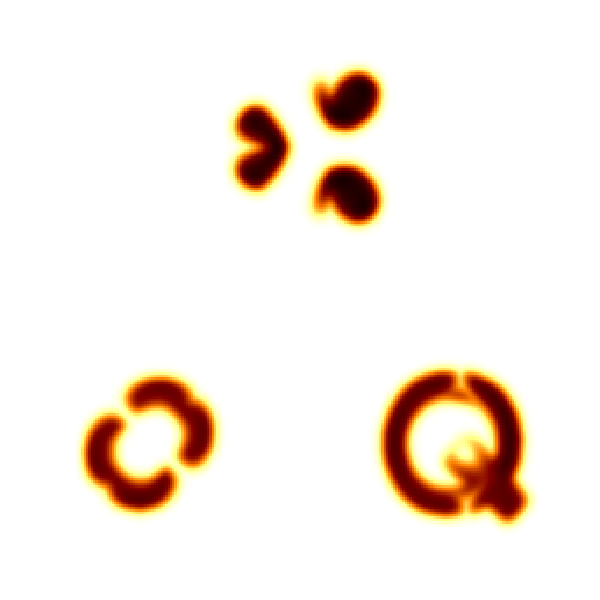}};
\node (13) at (2*\ax/4,0) {\includegraphics[clip,trim=\b cm \b cm \b cm \b cm]{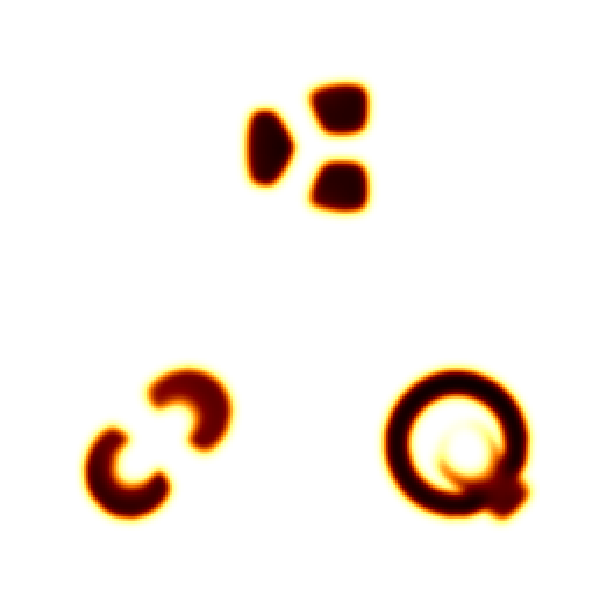}};

\node (14) at (-\ax/4,\ay/2) {\includegraphics[clip,trim=\b cm \b cm \b cm \b cm]{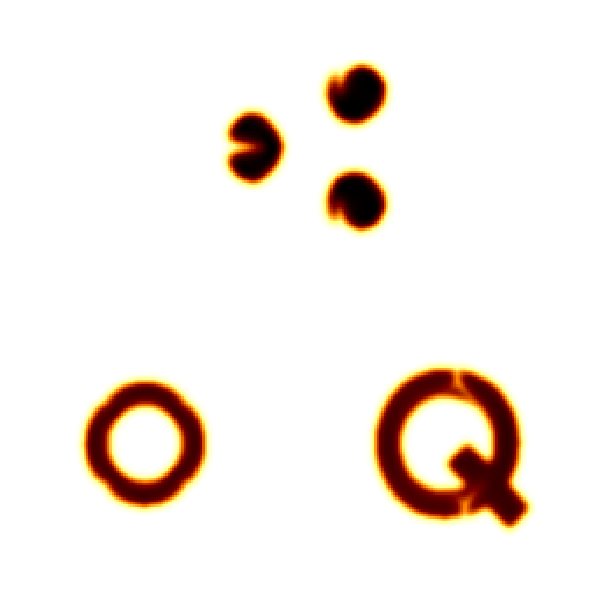}};
\node (15) at (\ax/4,\ay/2) {\includegraphics[clip,trim=\b cm \b cm \b cm \b cm]{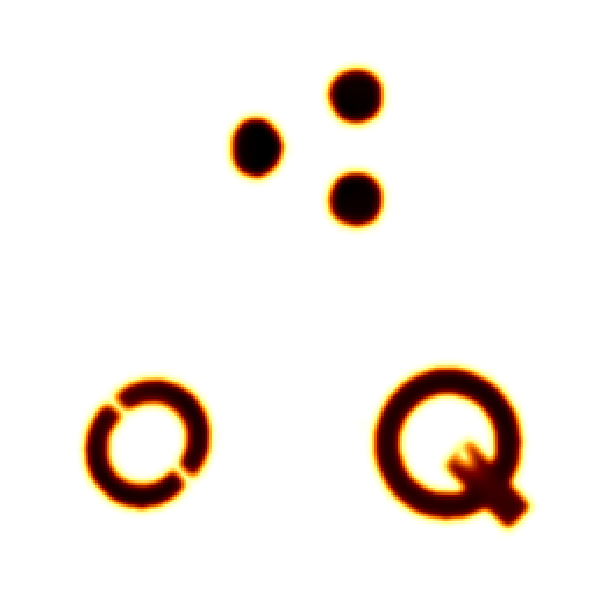}};

\node (15) at (\ax,\ay/2) {\includegraphics[scale=2.5]{triangleKL2/cmapbary}};

\draw[thick] 	(-\ax/4,3*\ay/4)    	-- 	(\ax/4,3*\ay/4);
\draw[thick] 	(-2*\ax/4,\ay/4)    	-- 	(2*\ax/4,\ay/4);
\draw[thick] 	(-3*\ax/4,-\ay/4)    	-- 	(3*\ax/4,-\ay/4);
\draw[thick] 	(-4*\ax/4,-3*\ay/4)    	-- 	(4*\ax/4,-3*\ay/4);

\draw[thick] 	(0,3*\ay/4)    	-- 	(0,1*\ay/4);
\draw[thick] 	(-\ax/4,1*\ay/4)    	-- 	(-\ax/4,-1*\ay/4);
\draw[thick] 	(\ax/4,1*\ay/4)    	-- 	(\ax/4,-1*\ay/4);
\draw[thick] 	(0,-\ay/4)    	-- 	(0,-3*\ay/4);
\draw[thick] 	(-\ax/2,-\ay/4)    	-- 	(-\ax/2,-3*\ay/4);
\draw[thick] 	(\ax/2,-\ay/4)    	-- 	(\ax/2,-3*\ay/4);
\draw[thick] 	(-\ax/4,-3*\ay/4)    	-- 	(-\ax/4,-5*\ay/4);
\draw[thick] 	(\ax/4,-3*\ay/4)    	-- 	(\ax/4,-5*\ay/4);
\draw[thick] 	(-3*\ax/4,-3*\ay/4)    	-- 	(-3*\ax/4,-5*\ay/4);
\draw[thick] 	(3*\ax/4,-3*\ay/4)    	-- 	(3*\ax/4,-5*\ay/4);
\end{tikzpicture}
        }
       }
\caption{$\GHK$ barycenters with entropic smoothing (the definition of the $\GHK$ metric is similar to Wasserstein but the marginal constraints are replaced by $\KL$ divergences, see Section  \ref{subsec_unbalanced}).}
\label{fig_GHK_triangle}
\end{figure}

\subsection{Gradient Flows and Evolution of Densities}
\label{sec_applicationGF}

The basic framework of gradient flows has been briefly laid out in Section \ref{subsec_gradientflows}. This Section details the application of Algorithm \ref{algo_scaling_stabilized} for solving them.
As the transition from the measures formulation to the density formulation and further to the algorithm with pointwise optimality conditions was carefully detailed in Sections \ref{sec_appli_UOT} and \ref{sec_appli_bary}, we skip some of these intermediate steps here.

\subsubsection{In the Wasserstein Space}
Scaling algorithms for solving Wasserstein gradient flows are not new~\cite{2015-Peyre-siims}, but our framework allows to simplify the derivation of the algorithm and the stabilized Algorithm \ref{algo_scaling_stabilized} allows to use much smaller regularization parameter $\epsilon$ yielding sharper and more precise flows. Given a convex, lower semicontinuous function on measures $\mathcal{G}$ with compact sublevel sets, each step requires to find the minimizer of
\eq{
\min_{\gamma \in \Mm_+(X\times X)} \langle c, \gamma\rangle + \iota_{\{=\}}(P^1_\# \gamma| \mu_k^\tau) + 2\, \tau \, \mathcal{G}(P^2_\# \gamma)
}
where $c :(x,y)\mapsto |y-x|^2$ is the quadratic cost. This directly fits in our framework by choosing $F_1(s) = \iota_{\{=\}}(s\, \d x| \mu_k^\tau)$ and $F_2(s) = 2 \, \tau \, \mathcal{G}(s\,\d x)$. 
\begin{example}
If the energy is given by the relative entropy $\mathcal{G}=\KLm(\cdot|p \, \d x)$ for some reference measure $p \, \d x$, then the proximal step for $F_2$ is given by 
\eq{
\prox^{\KL}_{\frac1\epsilon F_2}(q) = \argmin_{s : X \to \RR} \left( \epsilon \KL(s|q) + 2\tau \KL(s|p)\right) =q^\frac{\epsilon}{\epsilon+2\tau}\cdot p^\frac{2\tau}{\epsilon+2\tau}\, .
}
This problem is identical to an unbalanced transport problem where one of the marginal is fixed and the other is controlled by a $\KLm$ divergence. The difference is, that we want to solve a whole sequence of such problems, each time taking the second marginal of the optimal coupling and plugging it into the next problem as constraint for the first marginal. The gradient flow associated to this example allows to recover the heat flow when $p\,\d x$ is the Lebesgue measure on $X=\RR^d$.
\end{example}
\begin{example}
Some models of crowd motion with congestion \cite{roudneff2011handling} can be approximatively simulated by computing a sequence of measures $(\mu^\tau_{k})$ from an initial measure $\mu_0$, where for all $k\in \NN$, $\mu^\tau_{2k+1}$ is obtained from $\mu^\tau_{2k}$ by performing a free evolution during a time $\tau$ (this requires another algorithm) and then defining $\mu^\tau_{2k+2}$ as the Wasserstein projection of $\mu_{2k+1}^\tau$ (which we write as $p_{2k+1}\,\d x$) onto the set of measures with densities smaller than $1$. The second step can be performed by setting $F_1(s) = \iota_{\{=\}}(s|p_{2k+1})$ and $F_2(s)=\iota_{\leq 1}(s)$. The proximal operator of $F_2$ at a point $s$ is given by $\min \{1,s\}$ and thus $p_{2k+2}$ can be obtained from $p_{2k+1}$ with the stabilized Algorithm \ref{algo_scaling_stabilized} with 
\eq{
\proxdiv_{F_1}(s,u,\epsilon) = \tfrac{p_{2k+1}}{s}\,,
\qandq
\proxdiv_{F_2}(s,u,\epsilon) = \min \{\textstyle\frac1s , e^{-u/\epsilon}\}\, .
}
\end{example}

\subsubsection[WF Gradient Flows]{$\WF$ Gradient Flows}
For $\WF$ gradient flows, each step requires to solve
\eq{
\inf_{\substack{\mu \in \Mm_+(X) \\ \gamma \in \Mm_+(X\times X)}} \langle c, \gamma \rangle + \KLm(P^1_\# \gamma | \mu_k^\tau) +  \KLm(P^2_\# \gamma | \mu) + 2 \tau \mathcal{G}(\mu)
}
where $c$ is the cost given in \eqref{logcost}. By denoting $p_k$ the density of $\mu_k^\tau$ with respect to $\d x$, one defines the functions
\eq{
F_1(s) = \KL(s|p_k) \qandq
F_2(s) = \inf_{p\in \Lun(X)} \KL(s|p) + 2\tau \, G(p)
}
where $G(p) \eqdef \mathcal{G}(p\d x)$.
If $G$ is an integral functional of the form $G(s) = \int_X g_x(s(x)) \d x$, the proximal operator is given pointwise by
\eq{
\inf_{(\tilde{s},p)\in \RR^2} \epsilon \KL(\tilde{s}|s)+ \KL(\tilde{s}|p) + 2\tau \, g_x(p)
}
for which the first order optimality conditions read
\eq{
\begin{cases}
0 = \epsilon \log (\tilde{s}/s) + \log(\tilde{s}/p) &\\
(\tilde{s}/p-1)/(2\tau)\in \partial g_x(p)
\end{cases}
}
and $\tilde{s}=0$ if $s=0$ or $p=0$. In many cases, this system can be easily solved, as in the following example.
\begin{example}
\label{exemple_GFtumor}
One of the simplest functional which generates non-trivial gradient flows is
\eq{
\mathcal{G}(\mu)=-\alpha \cdot \mu(X) + \iota_{\leq 1}({\d \mu}/{\d x})
}
where $\alpha \in ]0,\infty[$. Since the distance $\WF$ measures both the displacement and the rate of growth, one can interpret the gradient flow of $\mathcal{G}$ as describing the evolution of a density of cells (a tumor, say) which have a tendency to multiply---hence increase the total mass---but which density cannot exceed $1$.
 One can solve the time discretized gradient flow with Algorithm \ref{algo_scaling_stabilized} by choosing $F_2(s) =  \inf_{p} \KL(s|p) - 2\alpha\tau \, p + \iota_{\leq 1}(p)$. With the optimality conditions above, one obtains
\eq{
\proxdiv_{F_2}(s,u,\epsilon) = 
\begin{cases}
(e^u(1-2\tau\,\alpha))^{-\frac{1}{\epsilon}} & \text{if $s\leq e^{\frac{u}{\epsilon}}(1-2\tau \alpha)^\frac{1+\epsilon}{\epsilon}$} \\
(s\, e^u)^{-\frac{1}{1+\epsilon}} &\text{otherwise.}
\end{cases}
}
Moreover, given an optimal coupling $r$, one has $\mu_{k+1}^\tau = \min \left(\frac{P^2_\# r}{1-2\tau \alpha}, 1 \right)\d x$.
\end{example}
A numerical illustration is given in Figure \ref{fig_WFGF} where we used Algorithm \ref{algo_scaling_stabilized} (stopped after $500$ iterations) for solving each step , with the following parameters: $X$ is the segment $[0,1]$ discretized into $3000$ uniformly spaced samples, the initial density $p_0$ is the black line on Figure \ref{fig_WFGFprofile}, $\tau= 0.006$, $\alpha=1$ and $\epsilon = 10^{-8}$. The running time was $315$ seconds. Remark how, by using a very small value for $\epsilon$, the ``smoothing'' effect of the entropy disappears: the contours of the free-boundary which evolve with time remain sharp.

\begin{figure}
 \centering
\begin{subfigure}{1.\linewidth}
\centering
 \resizebox{.9\linewidth}{!}{
\includegraphics[clip,trim=0cm 0cm 15.5cm 0cm]{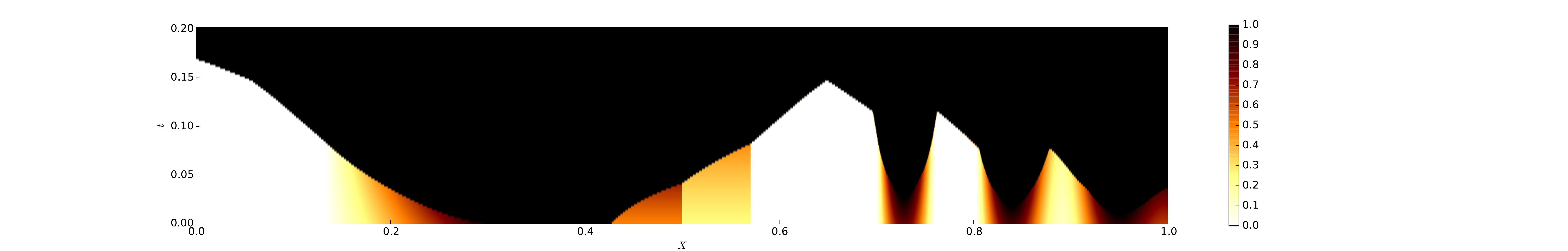}
}%
 \resizebox{.09\linewidth}{!}{
\includegraphics[clip,trim=1.5cm 2cm 1.5cm 1cm]{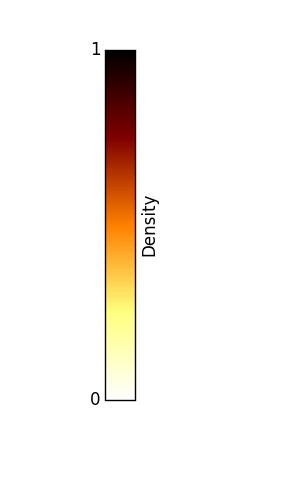}
}%
\caption{Evolution in color scale: time evolves from bottom to top.}\label{fig_WFGFtop}
\end{subfigure}
 \begin{subfigure}{.9\linewidth}
\centering
 \resizebox{1.\linewidth}{!}{
\includegraphics[clip,trim=0cm 0cm 0cm 0cm]{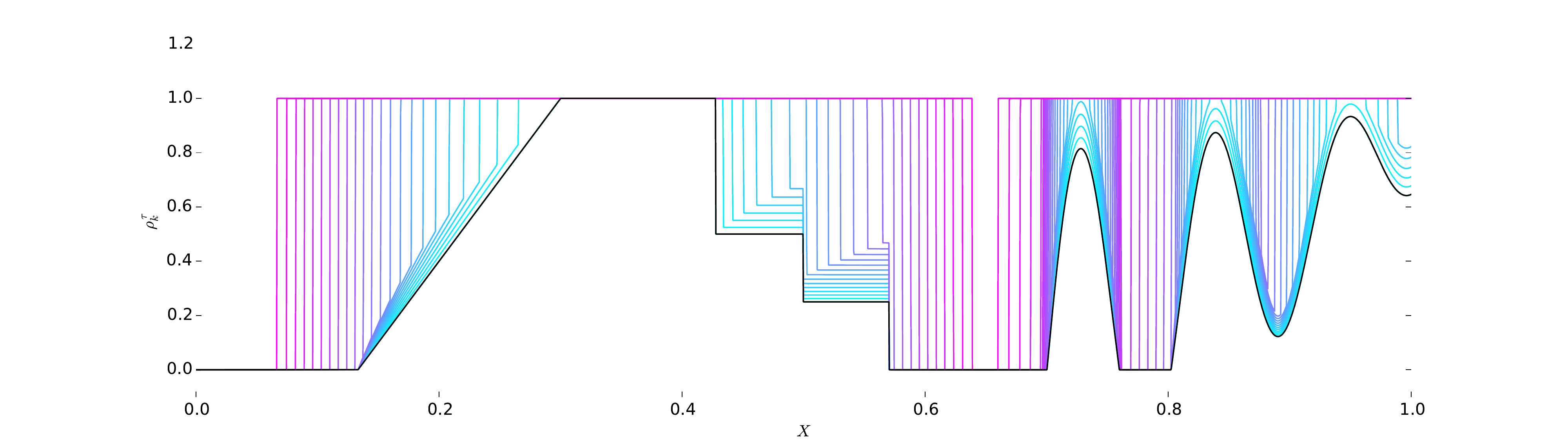}
}%
\caption{Evolution viewed ``laterally'' : the initial density $p_0$ is the black line, then the density is displayed at every time step with a color ranging from blue (small times) to pink (bigger times).}\label{fig_WFGFprofile}
\end{subfigure}
\caption{Evolution of the density with respect to time for the growth model of Example \ref{exemple_GFtumor}.}
\label{fig_WFGF}
\end{figure}

\subsubsection[WF Gradient Flows with Multiple Species]{$\WF$ Gradient Flows with Multiple Species}
The generic form \eqref{eq-general} also includes gradient flows with multiple species with a mutual interaction (with $n>1$, similar to the barycenter problem). Such systems have been theoretically studied in \cite{laborde2015some,zinsl2015transport}. Here we consider this class of problems in order to illustrate the versatility of the algorithm and it is not our purpose to make a link with the theory of PDEs.
Let us consider a simple example which is a direct extension of Example \ref{exemple_GFtumor}. Consider the following functional:
\eq{
\mathcal{G}(\mu^a,\mu^b) = -\alpha\cdot \mu^a -\alpha\cdot\mu^b + \iota_{\leq1}(\d (\mu^a+\mu^b)/\d x))\, .
}
In this model, one has two species which have a tendency to grow in mass (with the same incentive $\alpha>0$, for simplicity of the algorithm), and their sum cannot exceed the reference measure. The corresponding $F_2$ is given by 
\eq{
F_2(s^a,s^b) =\!\!\!\! \inf_{(r^a,r^b)\in \Lun(X)^2}\!\!\!\! \KL(s^a|r^a) +  \KL(s^b|r^b) -2\alpha\tau \int_X (r^a+r^b)\d x + \iota_{\leq1}(r^a+r^b)\, .
}
The optimality conditions yield  
\eq{
\proxdiv_{F_2}(s,u,\epsilon) = (e^{-\frac{u^a}{\epsilon}},e^{-\frac{u^b}{\epsilon}})/\beta(s^a\,e^{-\frac{u^a}{\epsilon}}, s^b\,e^{-\frac{u^b}{\epsilon}})
}
where $\beta(x,y) \eqdef \max \left\{ (x+y)^{\frac1{1+\epsilon}}, \, (1-2\tau\alpha)^{\frac1\epsilon} \right\}$. Morevoer, given an optimal pair of couplings $r^a, r^b$, the next densities are given by $(p^a_{k+1},p^b_{k+1}) = (P^2_\# r^a, P^2_\# r^b) / \beta(P^2_\# r^a, P^2_\# r^b)$.
Remark than in this model, as in Example \ref{exemple_GFtumor}, if the domain $X$ is compact and the initial densities are not null, a steady state is reached in finite time, where the sum of the two densities is constant and equal to $1$.

Some initial densities and the associated final steady state are shown on Figure \ref{fig_WFGF2}. For this illustration, we started with input densities $p^a$ and $p^b$ on the segment $[0,1]$ discretized into $3000$ uniform samples, as displayed on Figure \ref{fig_GFWF2_input} (where the red density $p^b$ is layered over $p^a$). We computed the evolution with Algorithm \ref{algo_scaling_stabilized} for solving each step of the discretized gradient flow, with the parameters $\tau=0.004$, $\alpha=1$ and $\epsilon=10^{-7}$.

Note that although the incentive of growing mass $\alpha$ is the same for the two species, the resulting interaction is non trivial: for instance the small amount of blue mass is pushed to the right by the action of the expanding red mass. This behavior is explained by the fact that for the $\WF$ metric, it requires less effort (i.e.\ the distance is smaller) to add a given amount of mass to a high density than to a small one.
\begin{figure}
\begin{subfigure}{.49\linewidth}
\centering
 \resizebox{1.\linewidth}{!}{
\includegraphics[clip,trim=3cm 0cm 2cm 0cm]{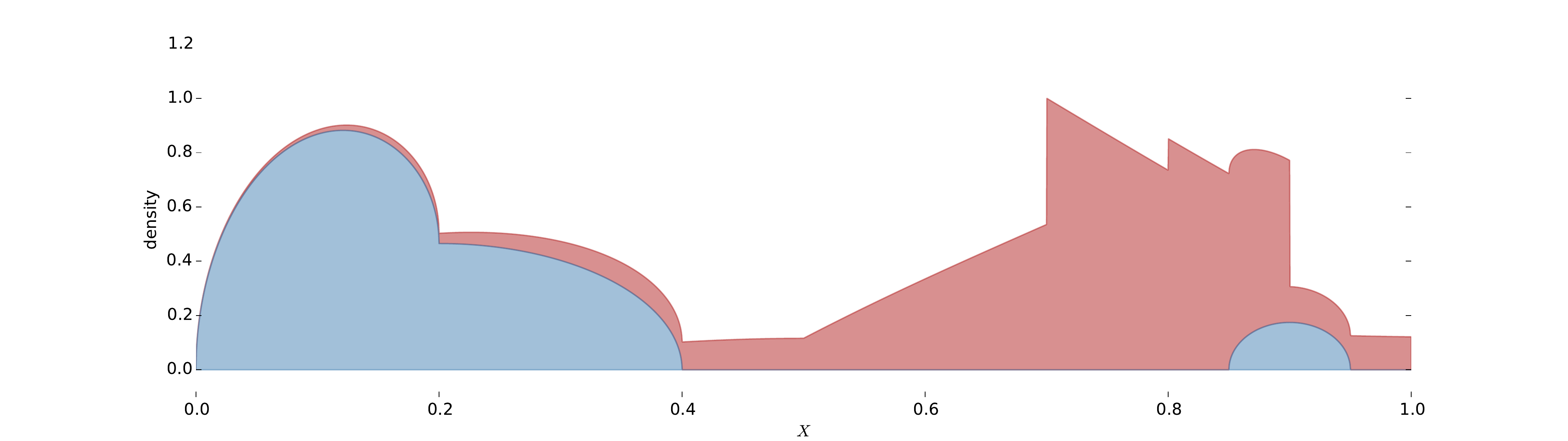}
}%
\caption{Initial state (blue) $p_0^a$ (red)$p_0^b$}\label{fig_GFWF2_input}
\end{subfigure}
 \begin{subfigure}{.49\linewidth}
\centering
 \resizebox{1.\linewidth}{!}{
\includegraphics[clip,trim=2cm 0cm 3cm 0cm]{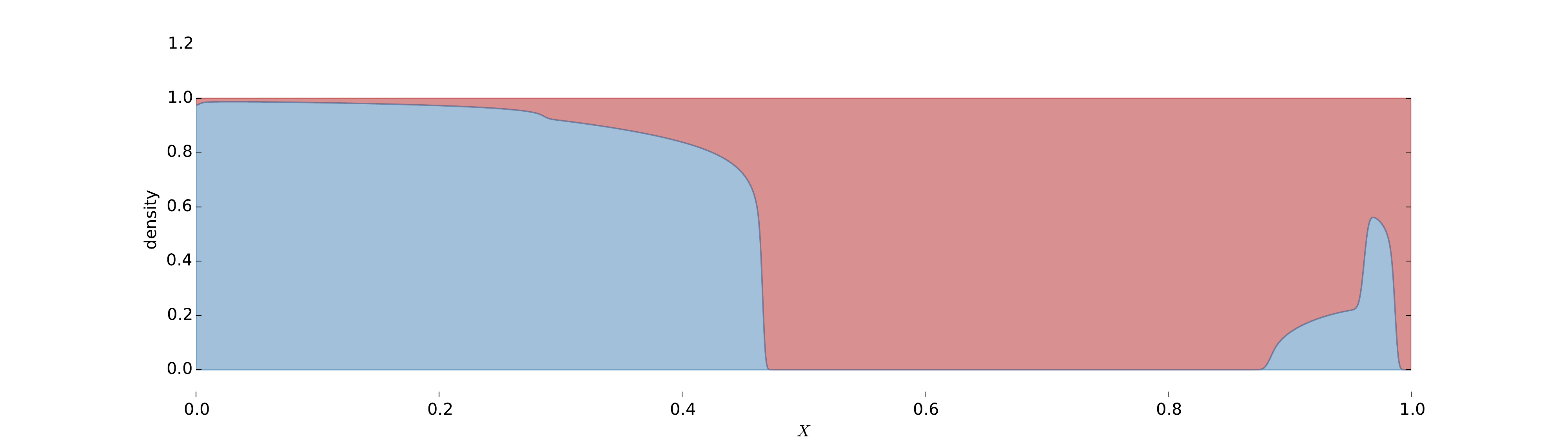}
}
\caption{Steady state densities (for $t$ big)}
\end{subfigure}
\begin{subfigure}{.49\linewidth}
\centering
\resizebox{.09\linewidth}{!}{
\includegraphics[clip,trim=2cm 0cm 2cm 0cm]{GF/cmapbary}
}%
 \resizebox{.9\linewidth}{!}{
\includegraphics[clip,trim=1.5cm 3cm 1cm 3cm]{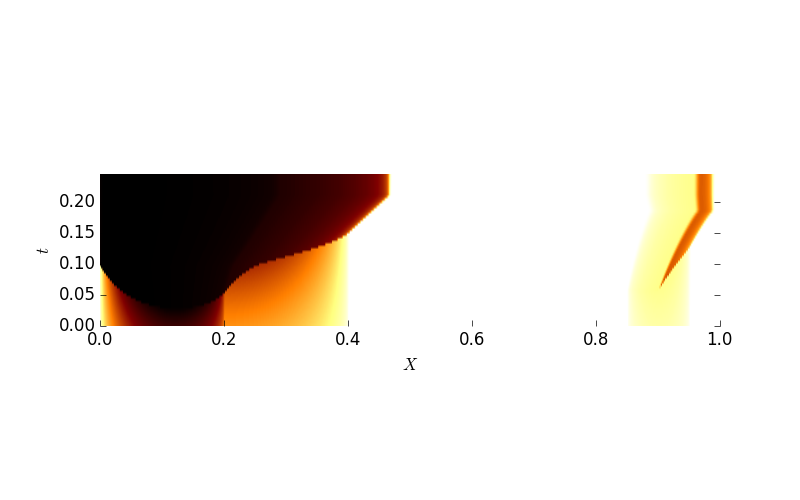}
}%
\caption{Evolution of $p^a$}
\end{subfigure}
\begin{subfigure}{.49\linewidth}
\centering
 \resizebox{.9\linewidth}{!}{
\includegraphics[clip,trim=1cm 3cm 1.5cm 3cm]{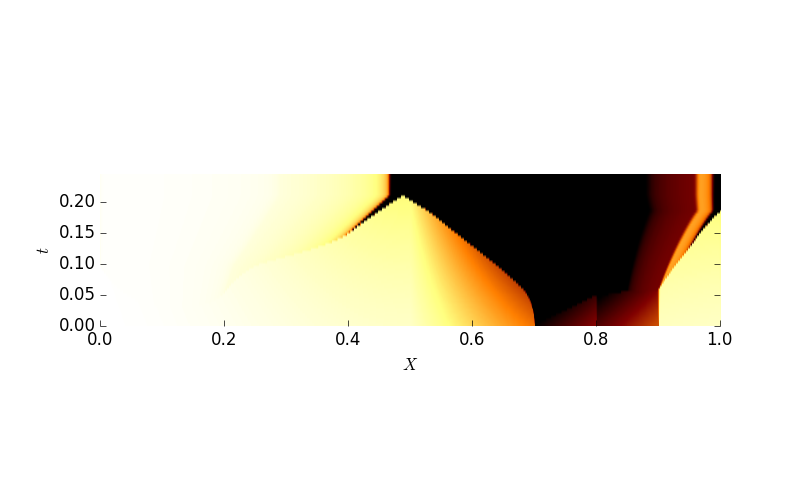}
}%
 \resizebox{.09\linewidth}{!}{
\includegraphics[clip,trim=2cm 0cm 2cm 0cm]{GF/cmapbary}
}
\caption{Evolution of $p^b$}
\end{subfigure}
\caption{(top row) initial state and steady state (density $p^b$ is layered over $p^a$) (bottom row) time evolution (time evolves from bottom to top) with density in color scale.}
\label{fig_WFGF2}
\end{figure}




\section*{Acknowledgements}
The work of Bernhard Schmitzer has been supported by the French National Research Agency (ANR) as part of the `Investissements d'avenir', program-reference ANR-10-LABX-0098 via the Fondation Sciences Math\'ematiques de Paris.
The work of Bernhard Schmitzer and Gabriel Peyr\'e has been supported by the European Research Council (ERC project SIGMA-Vision).

\appendix
\section{Appendix}

\subsection{Reminders on Convex Analysis}
\label{sec:ApxConvex}
Let $E$ and $E^*$ be \textit{topologically paired} vector spaces i.e.\ vector spaces assigned with locally convex Hausdorff topology such that all continuous linear functionals on each space can be identified with the elements of the other. The pairing between those spaces is the bilinear form $\langle \cdot , \cdot \rangle : E \times E^* \to \RR$. The convex conjugate of a function $f:E\to \RR \cup \{+\infty\}$ is defined for each $y\in E^*$ by
\eq{
f^*(y) \eqdef \sup_{x\in E}\; \langle x, y \rangle - f(x) \, .
}
The subdifferential operator is defined at a point $x\in E$ as
\eq{
\partial f (x) \eqdef \left\{ y \in E^* ; f(x')-f(x) \geq \langle y, x'-x \rangle \text{ for all $x'\in E$}\right\}
}
and is empty if $f(x)=\infty$. Those definitions admit their natural counterparts for functions defined on $E^*$. 

\begin{theorem}[Fenchel-Rockafellar \cite{rockafellar1967duality}]
\label{thm_FR}
Let $(E,E^*)$ and $(F,F^*)$ be two couples of topologically paired spaces. Let $A : E \to F$ be a continuous linear operator and $A^*:F^* \to E^*$ its adjoint. Let $f$ and $g$ be lower semicontinuous and proper convex functions defined on $E$ and $F$ respectively. If there exists $x \in \dom f$ such that $g$ is continuous at $Ax$, then
\eq{
\sup_{x \in E} - f(-x) - g(Ax) = \min_{y^* \in F^*} f^*(A^*y^*) + g^*(y^*)
}
and the $\min$ is attained. Moreover, if there exists a maximizer $x\in E$ then there exists $y^*\in F^*$ satisfying $Ax \in \partial g^*(y^*)$ and $A^*y^* \in \partial f(-x)$.
\end{theorem}

\subsection{Properties of Divergence Functionals}
\label{sec:ApxDivergences}

Here we collect a few results on divergences functionals when they are defined on functions as in \eqref{eq_divergencefunctions} (as opposed to Section \ref{sec_divergencefunc} where they are defined between measures).
\begin{proposition}\label{prop_divergnormalint}
Let $\varphi$ be a nonnegative entropy function as in Definition \ref{def_entropy} and $(X,\d x)$ a measured space. Then $(u,v)\in \Lun(X)^2 \mapsto \Diverg_\phi(u|v)$ is an admissible integral functional (in the sense of Definition \ref{def_integralfunctional}) which is positively $1$-homogeneous, convex and weakly lower semicontinuous. Moreover, $\Diverg_\phi^*= \iota_{B_\phi}$ with $B_\phi = \{(a,b)\in\RR^2 ; b \leq -\phi^*(a)\}$.
\end{proposition}
\begin{proof}
As a preliminary, remark that for $(u,v)\in \RR^2$,
\eq{
\ol{\Diverg}^*(u|v) = \sup_{a,b\in \RR_+} \begin{cases}
b(u\cdot a/b + v - \phi(a/b)) & \text{if $b>0$}\\
a(u-\phi'_\infty) & \text{otherwise}
\end{cases}
 = 
\begin{cases}
0 & \text{if $v\leq \phi^*(u)$}\\
\infty & \text{otherwise.}
\end{cases}
}
Now, the function $(x,u,v)\in X\times\RR^2 \mapsto \ol{\Diverg}_\phi(u,v)$ defined in \eqref{eq_divergencefunctions} is a normal integrand since it does not depend on $x$. Moreover, there exists feasible points for $\Diverg_\phi$ (take $v\in \Lun(X)$ and $u=\alpha v$ where $\alpha\in \dom \phi$) and for the integral functional associated to $\ol{\Diverg}^*_\phi$ (given its expression above). The conclusion follows by \cite[Theorem 3C]{rockafellar1976integral}. 
\end{proof}
\begin{proposition}
\label{prop_integralconj}
Let $(X,\d x)$ be a measured space, $v \in \Lun_+(X)$ and $\phi$ a nonnegative entropy function as in Definition \ref{def_entropy}. Then $\Diverg_\phi(\cdot|v)$ is a proper, weakly lower semicontinuous convex function on $\Lun(X)$ and its convex conjugate is given, for $a \in \Linf(X)$, by
\begin{equation*}
\Diverg^*_\phi ( a | v) \eqdef  \int_X \phi^*(a(x)) v(x) \d x + \int_X \iota_{\leq \phi'_\infty}(a(x)) \d x
\end{equation*}
where $\phi^*$ is the convex conjugate of $\phi$.

Moreover, the subdifferential $\partial \Diverg_\phi (\cdot|v)$ at a point $u\in \Lun(X)$ is the set of functions $a\in \Linf(X)$ such that $\phi'_\infty-a$ is nonnegative and such that, for a.e.\ $x$ where $v(x)>0$, $a(x)\in \partial \phi(u(x)/v(x))$ .

Similarly, the subdifferential $\partial \Diverg^*_\phi (\cdot|v)$ at a point $a\in \Linf(X)$ bounded above by $\phi'_\infty$ is the set of nonnegative functions $u\in \Lun(X)$ such that, for a.e.\ $x$, $u(x)\in \partial \phi^*(a(x)) v(x)$ if $v(x)>0$ and $u(x) = 0$ if $v(x)=0$ and $a(x)<\phi'_\infty$.
\end{proposition}

\begin{proof}
By \cite[Proposition 14.45c]{rockafellar2009variational}, $(x,u)\in X\times \RR \mapsto \ol{\Diverg}_\phi(u|v(x))$ is a normal integrand. Then  \cite[Theorem 3C]{rockafellar1976integral} and Corollaries apply and conjugation and subdifferentiation can be performed pointwise.
\end{proof}

\subsection{Proof of the Iterates for the Barycenter Problems}
\label{sec:AppendixBarycenterIterates}
We explain below how to derive the expression for $h$ which are given in Table \ref{prop_estimatebary}, by applying Proposition \ref{prop_barycenter_general}. Assume that $(s_i)_{i=1}^n\in \RR^n\geq0$ is given. 

\paragraph{Case $\Diverg_\phi = \iota_{\{=\}}$.}
This case is simple because solving \eqref{eq_prox_barycenter}  boils down to solving the one dimensional problem $\min_h \sum \alpha_k \ol{\KL}(h|s_k)$, which is direct with first order optimality conditions.

\paragraph{Case $\Diverg_\phi = \lambda \KL$.}
First remark that the assumption of Proposition \ref{prop_barycenter_general} is satisfied and that $h=0$ if and only if for all $k$, $s_k=0$ (otherwise, the joint subdifferential is empty). 
Since $\phi$ is smooth, its joint subdifferential is a singleton $\partial \ol{\KL}(\tilde{s}|h) = \{(\log(\tilde{s}/h), 1-\tilde{s}/h) \}$ if $\tilde{s},h>0$. Also, since $\ol{\KL}(0|h)=h+ \iota_{[0,\infty[}(h)$, one has $\partial_2 \ol{\KL}(0,h) = \{1\}$ if $h>0$. Thus, optimality conditions in Proposition \ref{prop_barycenter_general} yields the system
\eq{
\begin{cases}
\log \frac{\tilde{s}_k}{h} = \frac\epsilon\la \log \frac{s_k}{\tilde{s}_k} & \text{if $s_k>0$,} \\
\tilde{s}_k = 0 & \text{if $s_k=0$,} \\
\sum \alpha_k (1-\frac{\tilde{s}_k}{h}) = 0\, .
\end{cases}
}

\paragraph{Case $\Diverg_\phi = \lambda \TV$.}
By Proposition \ref{prop_divergnormalint}, one has $\ol{\Diverg}_{\phi_{\TV}}(\tilde{s}|h) = \sup_{(a,b)\in B} a\cdot x + b\cdot y $ with $B = \{ (a,b)\in \RR^2\, ; \, a\leq1, \, b\leq 1, \, a+b\leq 0 \}$. The set of points in $B$ at which this supremum is attained is easy to see graphically and gives the set $\partial \ol{\Diverg}_{\phi_{\TV}}(\tilde{s}|h)$. With the notations of Proposition \ref{prop_barycenter_general}, one has with $a_k=\frac\epsilon\la \log \frac{s_k}{\tilde{s}_k}$,
\begin{align*}
(1)\; \tilde{s}_k > h > 0 &\Leftrightarrow -b_k=a_k=1 &
(2)\;  h > \tilde{s}_k > 0 &\Leftrightarrow -b_k=a_k=-1 \\
(3)\; \tilde{s}_k = h > 0 &\Leftrightarrow -b_k=a_k\in [-1,1] &
(4)\; h > \tilde{s}_k = 0&\Leftrightarrow  b_k = 1 \\
(5)\; \tilde{s}_k > h = 0 &\Leftrightarrow a_k = 1 \text{ and } b_k\leq-1 &
(6)\; \tilde{s}_k = h = 0 &\Leftrightarrow b_k \leq 1 \, .
\end{align*}
Let us first deal with the case $h=0$ (cases (5) and (6)). Condition $\sum \alpha_k b_k=0$ from Proposition \ref{prop_barycenter_general} says that it is the case if and only if $\sum_{k\notin I_+} \alpha_k \geq \sum_{k \in I_+} \alpha_k$. 
Now assume that $h>0$. If $\tilde{s}_k>0$ (cases (1), (2) and (3)) then $b_k$ can be expressed as $\max ( -1 , \min ( 1 , \frac{\epsilon}{\lambda } \log \frac{h}{s_k}))$ otherwise $b_k=1$. The implicit expression given for $h$ is thus the condition $\sum \alpha_k b_k =0$.

\paragraph{Case $\Diverg_\phi = \RG_{[\beta_1, \beta_2]}$.}
In this case, $\ol{\Diverg}_\phi$ is the support function of $ B= \{ (a,b)\in \RR^2\, ; \,\text{for } i\in \{1,2\}, b\leq -\beta_i \cdot a\} $. With the notations of Proposition \ref{prop_barycenter_general}, one has with $a_k=\frac\epsilon\la \log \frac{s_k}{\tilde{s}_k}$, 
\begin{align*}
(1)\; 0 < \beta_1 h < \tilde{s}_k < \beta_2 h &\Leftrightarrow a_k = b_k = 0 &
(2)\; 0 < \beta_1 h = \tilde{s}_k &\Leftrightarrow b_k = -\beta_1 a_k \\
(3)\; 0 < \beta_2 h = \tilde{s}_k & \Leftrightarrow b_k = -\beta_2 a_k &
(4)\; 0 = h = \tilde{s}_k & \Leftrightarrow (b_k , a_k) \in B \, .
\end{align*}
If $s_k=0$ for some $k\in \{1, \dots, n\}$ then $h=0$ (this is the only feasible point). Otherwise, $h>0$ and the condition $\sum \alpha_k b_k = 0$ gives the implicit equation.

\bibliographystyle{amsplain}
\bibliography{EntropicNumeric}

\end{document}